\documentclass{amsart}
\usepackage{amsmath,amssymb,amscd}
\usepackage{color}
\usepackage{mathtools}
\usepackage{bigints}
\usepackage{comment}
\usepackage{mathdots}

\begin{document}

\makeatletter
%\@addtoreset{equation}{section}
\def\theequation{\thesection.\arabic{equation}}
\makeatother

\theoremstyle{definition}
\newtheorem{dfn}{Definition}[section]
\newtheorem{thm}[dfn]{Theorem}
\newtheorem*{th*}{Theorem}
\newtheorem{lem}[dfn]{Lemma}
\newtheorem{prop}[dfn]{Proposition}
\newtheorem{rem}[dfn]{Remark}
\newtheorem{cor}[dfn]{Corollary}
\newtheorem*{cor*}{Corollary}
\newtheorem*{prop*}{Proposition}
\newtheorem{quest}[dfn]{Question}
\newtheorem{conj}[dfn]{Conjecture}

%mathbb
\newcommand{\bbA}{\mathbb{A}}
\newcommand{\bbC}{\mathbb{C}}
\newcommand{\bbR}{\mathbb{R}}
\newcommand{\bbQ}{\mathbb{Q}}
\newcommand{\bbZ}{\mathbb{Z}}
\newcommand{\bbH}{\mathbb{H}}

%mathcal
\newcommand{\calA}{\mathcal{A}}
\newcommand{\calC}{\mathcal{C}}
\newcommand{\calD}{\mathcal{D}}
\newcommand{\calE}{\mathcal{E}}
\newcommand{\calH}{\mathcal{H}}
\newcommand{\calI}{\mathcal{I}}
\newcommand{\calN}{\mathcal{N}}
\newcommand{\calO}{\mathcal{O}}
\newcommand{\calP}{\mathcal{P}}
\newcommand{\calS}{\mathcal{S}}
\newcommand{\calU}{\mathcal{U}}
\newcommand{\calV}{\mathcal{V}}
\newcommand{\calZ}{\mathcal{Z}}

%mathfrak
\newcommand{\fraka}{\mathfrak{a}}
\newcommand{\frakH}{\mathfrak{H}}
\newcommand{\frakh}{\mathfrak{h}}
\newcommand{\fraki}{\mathfrak{i}}
\newcommand{\frakg}{\mathfrak{g}}
\newcommand{\frakk}{\mathfrak{k}}
\newcommand{\frakn}{\mathfrak{n}}
\newcommand{\frakN}{\mathfrak{N}}
\newcommand{\frakp}{\mathfrak{p}}
\newcommand{\frakX}{\mathfrak{X}}

%mathbf
\newcommand{\bfs}{\mathbf{s}}
\newcommand{\bfe}{\mathbf{e}}
\newcommand{\bfa}{\mathbf{a}}

%operators
\newcommand{\rank}{\mathop{\mathrm{rank}}}
\newcommand{\corank}{\mathop{\mathrm{Corank}}}
\newcommand{\im}{\mathop{\mathrm{Im}}}
\newcommand{\Hom}{\mathop{\mathrm{Hom}}}
\newcommand{\op}{\mathrm{op}}

%Groups
\newcommand{\GL}{\mathrm{GL}}
\newcommand{\Sym}{\mathrm{Sym}}
\newcommand{\Sp}{\mathrm{Sp}}
\newcommand{\Mat}{\mathrm{Mat}}
\newcommand{\Mp}{\mathrm{Mp}}
\newcommand{\SL}{\mathrm{SL}}

%spaces
\newcommand{\automforms}{\mathcal{A}(\Gamma)}
\newcommand{\NHMFonG}{\mathcal{N}(\Gamma)}
\newcommand{\NHMFonGwtlam}{\mathcal{N}_\lambda(\Gamma)}
\newcommand{\gk}{(\frakg, K_\infty)}
\newcommand{\NHMFonGcharchi}{\mathcal{N}(\Gamma,\chi)}
\newcommand{\NHAFonPG}{\mathcal{N}(P \backslash G)}
\newcommand{\NHAFonG}{\mathcal{N}(G)}
\newcommand{\NHAFonBG}{\mathcal{N}(B \backslash G)}

\newcommand{\ul}{\underline}
\newcommand{\pr}{\mathop{\mathrm{pr}}}
\newcommand{\Ext}{\mathop{\mathrm{Ext}}}
\newcommand{\fin}{{\rm \mathchar`- fin}}
\newcommand{\fini}{\mathrm{fin}}
\newcommand{\Ind}{\mathrm{Ind}}
\newcommand{\ind}{\mathrm{ind}}
\newcommand{\Lie}{\mathrm{Lie}}
\newcommand{\Res}{\mathop{\mathrm{Res}}}
\newcommand{\bs}{\backslash}
\newcommand{\tr}{\mathrm{tr}}
\newcommand{\bole}{\bold{e}}
\newcommand{\supp}{\mathop{\mathrm{supp}}}
\newcommand{\sgn}{\mathrm{sgn}}
\newcommand{\bfi}{\mathbf{i}}
\newcommand{\Wh}{\mathrm{Wh}}
\newcommand{\cusp}{\mathrm{cusp}}
\newcommand{\bfm}{\mathbf{m}}

%%%%renewcommand

\title{Nearly holomorphic automorphic forms on $\mathrm{Sp}_{2n}$ with sufficiently regular infinitesimal characters and applications}
\author{Shuji Horinaga}

\begin{abstract}
In this paper, we decompose the space of nearly holomorphic Hilbert-Siegel automorphic forms as representations of the adele group under certain assumptions.
We also give an application for classical holomorphic Hilbert-Siegel modular forms.
In particular, we show the surjectivity of the global Siegel operator $\widetilde{\Phi}$ for certain congruence subgroups with large weights.
\end{abstract}

\maketitle
\markboth{Nearly holomorphic automorphic forms on $\Sp_{2n}$}{Nearly holomorphic automorphic forms on $\Sp_{2n}$}

\section{Introduction}
	
	The purpose of this paper is the generalization of \cite{Horinaga}.
	Roughly speaking, we aim the spectral decomposition of the space of nearly holomorphic automorphic forms.
	The spectral theory of the space of automorphic forms can be found in \cite{Franke, langlands} and \cite{MW}.
	As a consequence of their results \cite[Appendix II]{MW}, \cite[Corollary 1]{Franke}, any automorphic form can be written as a sum of cusp forms and the Taylor coefficients of Eisenstein series.
	Some refinements are known in the case of holomorphic modular forms.
	By \cite[Theorem 8.3]{85_shimura}, the space of holomorphic Hilbert modular forms is spanned by cusp forms and the leading terms of Eisenstein series.
	Also, cusp forms and the leading terms of Eisenstein series span the space of full-modular scalar valued Siegel modular forms of large weight by \cite[Proposition 8]{Klingen}.
	In \cite[Conjecture A.2]{Horinaga}, we conjecture that any nearly holomorphic modular form (Hilbert-Siegel modular forms, Hermitian modular forms, etc.) can be written as a sum of cusp forms and the leading terms of Eisenstein series.
	In this paper, we show that the conjecture is true if the modular form has a ``sufficiently regular'' infinitesimal character.
	
	In order to state our main result, we introduce some notation.
	Let $F$ be a totally real field with degree $d$.
	We denote by $\bfa$ the set of embeddings of $F$ into $\bbR$.
	Put $G_n = \Res_{F/\bbQ}\Sp_{2n}$.
	Here $\Res$ is the Weil restriction and $\Sp_{2n}$ is the symplectic group of rank $n$.
	Let $\frakH_n$ be the Siegel upper half space of degree $n$.
	Then the Lie group $G_n(\bbR)$ acts on the $d$-fold product $\frakH_n^d$ by the linear fractional transformation.
	Put $\frakg_n = \Lie(G_n(\bbR)) \otimes _\bbR \bbC$.
	We denote by $K_{n,\infty}$ and $\calZ_n$ the stabilizer of $\mathbf{i} = (\sqrt{-1}\,\mathbf{1}_n,\ldots,\sqrt{-1}\,\mathbf{1}_n) \in \frakH_n^d$ and the center of the universal enveloping algebra $\calU(\frakg_n)$, respectively.
	Set $\frakk_n = \Lie(K_{n,\infty}) \otimes_\bbR\bbC$.
	We then have the well-known decomposition:
	\[
	\frakg_n = \frakk_n + \frakp_{n,+} + \frakp_{n,-}.
	\]
	Here $\frakp_{n,+}$ (resp.~$\frakp_{n,-}$) is the Lie subalgebra of $\frakg_n$ corresponding to the holomorphic tangent space (resp.~anti-holomorphic tangent space) of $\frakH_n^d$ at $\mathbf{i}$.
	We take a Cartan subalgebra of $\frakk_n$.
	Then the Cartan subalgebra of $\frakk_n$ is a Cartan subalgebra of $\frakg_n$.
	The root system $\Phi$ of $\mathfrak{sp}_{2n}(\bbC)$ is 
	\[
	\Phi = \{ \; \pm (e_i + e_j), \; \pm(e_k - e_l), \quad 1\leq i \leq j \leq n, 1 \leq k < l \leq n \; \}.
	\] 
	We put the set
	\[
	\Phi^+ = \{ \; - (e_i + e_j), \; e_k - e_l, \quad 1\leq i \leq j \leq n, 1 \leq k < l \leq n \; \}
	\]
	to be a positive root system.
	Let $\rho$ be half the sum of positive roots.
	Note that $\frakg_n = \bigoplus_{v \in \bfa} \mathfrak{sp}_{2n}(\bbC)$.
	We say that a weight $\lambda = (\lambda_{1,v},\ldots,\lambda_{n,v})_{v\in \bfa}$ which lies in $\bigoplus_{v \in \bfa}\bbC^{n}$ is $\frakk_n$-dominant if $\lambda_{i,v}-\lambda_{i+1,v} \in \bbZ_{\geq 0}$ for any $1 \leq i \leq n-1$ and $v \in \bfa$.
	If any entry of a weight $\lambda$ is integral, we say that $\lambda$ is integral.
	For any $\frakk_n$-dominant integral weight $\lambda$, there exists a unique irreducible highest weight $(\frakg_n,K_{n,\infty})$-module $L(\lambda)$ of highest weight $\lambda$.
	
	Fix a Borel subgroup $B_n$ of $G_n$ with a maximal split torus $T_n$ as usual.
	For a parabolic subgroup $P$ of $G_n$, we say that $P$ is standard if $P$ contains $B_n$.
	We denote by $P_{i,n}$ and $Q_{i,n}$ the standard parabolic subgroups of $G_n$ with the Levi subgroup $\Res_{F/\bbQ}\GL_i \times G_{n-i}$ and $(\Res_{F/\bbQ}\GL_1)^i \times G_{n-i}$, respectively.
	Throughout this paper, for a standard parabolic subgroup $P$, we denote by $M_P$ and $N_P$ the standard Levi subgroup of $P$ and the unipotent subgroup of $P$, respectively.
	
	For a while, we assume $F=\bbQ$.
	For a character $\mu$ of $T_n(\bbR)$, let $I_{B_n}(\mu) = \Ind_{B_n(\bbR)}^{G_n(\bbR)} (\mu)$ be the principal series representation of $G_n(\bbR)$.
	By the result of Yamashita \cite[Theorem 2.6]{Yamashita}, $L(\lambda)$  can be embedded into $I_{B_n}(\mu)$ if and only if $\mu$ is equal to
	\[
	\mu = \mathrm{sgn}^{\lambda_n}|\cdot|^{\lambda_n - n} \boxtimes \mathrm{sgn}^{\lambda_{n-1}}|\cdot|^{\lambda_{n-1} - n+1} \boxtimes \cdots \boxtimes \mathrm{sgn}^{\lambda_1}|\cdot|^{\lambda_1 - 1}
	\]
	as in (\ref{princ_ser}).
	Here we regard $T_n$ as the $n$-fold product $(\Res_{F/\bbQ}\GL_1)^n$.
	For a $\frakk_n$-dominant integral weight $\lambda = (\lambda_1,\ldots,\lambda_n) \in \bbZ^n$, put
	\[
	I_{P_{i,n}}(\lambda) = \Ind_{P_{i,n}(\bbR)}^{G_n(\bbR)} \left(\mathrm{sgn}^{\lambda_n} |\cdot|^{\lambda_n-n+(i-1)/2} \boxtimes L(\lambda_1,\ldots,\lambda_{n-i})\right).
	\]
	Then there exists a canonical inclusion
	\[
	I_{P_n}(\lambda) \xhookrightarrow{\qquad} I_{B_n}(\lambda).
	\]
	In order to construct the holomorphic Klingen Eisenstein series, we need the following refinement of the Yamashita's result:
	
	\begin{thm}[Theorem \ref{emb_main}]
	Take a $\frakk_n$-dominant integral weight $\lambda = (\lambda_1,\ldots,\lambda_n) \in \bbZ^n$.
	Then $I_{P_{i,n}}(\lambda)$ contains a highest weight vector of weight $\lambda$ if $\lambda_n = \cdots = \lambda_{n-i+1}$.
	Moreover such a highest weight vector generates $L(\lambda)$.
	Conversely, if the induced representation
	\[
	\Ind_{P_{i,n}(\bbR)}^{G_n(\bbR)} \left(\mu \boxtimes L(\omega_1,\ldots,\omega_{n-i})\right)
	\]
	has a highest weight vector of weight $\lambda$, we get
	\begin{itemize}
	\item $\lambda_n = \cdots = \lambda_{n-i+1}$.
	\item $\mu = \mathrm{sgn}^{\lambda_n} |\cdot|^{\lambda_n-n+(i-1)/2}$.
	\item $(\omega_1,\ldots,\omega_{n-i}) = (\lambda_1,\ldots,\lambda_{n-i})$.
	\end{itemize}
	\end{thm}
	
	This theorem plays a fundamental roll in this paper.
	In the rest of the introduction, we assume $F$ is a totally real field.
	We denote by $\calA(G_n)$ the space of automorphic forms on $G_n(\bbA_\bbQ)$.
	Put
	\[
	\calN(G_n) = \{\varphi \in \calA(G_n) \mid \text{$\varphi$ is $\frakp_{n,-}$-finite}\}.
	\]
	We call an automorphic form in $\calN(G_n)$ a nearly holomorphic automorphic form.
	In this paper, we say that an automorphic representation is a $G_n(\bbA_\fini)\times(\frakg_n,K_\infty)$-subrepresentation in the space of automorphic forms.
	In the usual terminology, an automorphic representation is an irreducible quotient of it.
	However, for the purpose of this paper, we do not consider the irreducible quotient of it.
	Note that automorphic representations in this paper can not be expressed by restricted tensor products of local representations.
	
	In \cite[Appendix]{Horinaga}, for a cuspidal pair $(M,\tau)$, we define a certain space of Eisenstein series $\calE_0(M,\tau)$.
	The space $\calE_0(M,\tau)$ can be considered as the space of the leading terms of certain Eisenstein series.
	We then conjecture the following:
	
	\begin{conj}[\cite{Horinaga} Conjecture A.2]\label{conj}
	Let $\calN(G_n)_{(M,\tau)}$ be the space of nearly holomorphic automorphic forms with the cuspidal support $(M,\tau)$.
	Then the following statements hold:
	\begin{enumerate}
	\item $\calN(G_n)_{(M,\tau)} \subset \calE_0(M,\tau)$.
	\item The action of $\calZ_n$ on $\calN(G_n)$ is semisimple.
	\item If $\calN(G_n)_{(M,\tau)} \neq 0$, the infinitesimal character of $\tau$ is integral.
	\end{enumerate}
	\end{conj}
	
	For an associated class $\{P\}$ of a standard parabolic subgroup $P$ of $G_n$, let $\calN(G_n)_{\{P\}}$ be the subspace of $\calN(G_n)$ consisting of automorphic forms that is concentrated on the class $\{P\}$.
	By the general theory of automorphic forms \cite{langlands, MW}, the space $\calN(G_n)$ decomposes as the direct sum 
	\[
	\calN(G_n) = \bigoplus_{\{P\}} \calN(G_n)_{\{P\}},
	\]
	where $\{P\}$ runs through all associated classes of parabolic subgroups.
	Then we have the following:
	
	\begin{prop}[Proposition \ref{decomp_NHAF_parab}]
	With the above notation, we have
	\[
	\calN(G_n) = \bigoplus_{i=0}^n \calN(G_n)_{\{Q_{i,n}\}}.
	\]
	\end{prop}
	
	To show this, we need to investigate the constant terms of nearly holomorphic automorphic forms.
	Let $F_v$ be the $v$-completion of $F$ for a place $v$ of $F$.
	Put $F_\infty = \prod_{v \in \bfa}F_v$.
	We denote by $\calO_{F_v}$ the ring of integers of $F_v$.
	Put
	\[
	K_n = \left(\prod_{v} \Sp_{2n}(\calO_{F_v})\right) \times K_{n,\infty},
	\] 
	where $v$ runs through all finite places of $F$.
	Then $K_n$ is a maximal compact subgroup of $G_n(\bbA_\bbQ)$.
	For a $G_{n-i}(\bbA_{\bbQ,\fini})\times(\frakg_{n-i},K_{n-i,\infty})$-subrepresentation $(\pi,V)$ of $\calN(G_{n-i})$, a character $\mu$ of $\GL_i(\bbA_F)$ and a complex number $s$, we denote by $I_{P_{i,n}}(\mu,s,\pi)$ the space of smooth functions $\varphi$ on $N_{P_{i,n}}(\bbA_\bbQ)M_{P_{i,n}}(\bbQ) \bs G_n(\bbA_\bbQ)$ such that
	\begin{itemize}
	\item $\varphi$ is right $K_n$-finite.
	\item For any $k \in K_n$, the function $m \longmapsto \delta_{P_{i,n}}^{-1/2}(m)\varphi(mk)$ on $M_{P_{i,n}}(\bbA_\bbQ)$ lies in the representation space of $\mu|\det|^s \boxtimes \pi$.
	\end{itemize}
	Here $\delta_{P_{i,n}}$ is the modulas character of $P_{i,n}$.
	For a cuspidal automorphic representation $\pi$ of $G_{n-i}(\bbA_F)$, we consider the representation space of $\pi$ as the $\pi$-isotypic component of the space of cusp forms in $\calA(G_{n-i})$.
	Then the constant terms of nearly holomorphic automorphic forms satisfy the following:
	
	\begin{prop}[Proposition \ref{Wh_coeff_main}]
	For a nearly holomorphic automorphic form $\varphi$, there exist finite collections of Hecke characters $\{\mu_\ell\}_\ell$ of $F^\times F_{\infty,+}^\times \bs \bbA_F^\times$, complex numbers $\{s_\ell\}_\ell$ and subrepresentations $\{(\pi_\ell,V_\ell)\}_\ell$ of $\calN(G_{n-i})$ such that
	\[
	\varphi_{P_{i,n}} \in \sum_{\ell} I_{P_{i,n}}(\mu_\ell,s_\ell,\pi_\ell).
	\]
	Here $F_{\infty,+}^\times$ is the identity component of $F_{\infty}^\times$.
	\end{prop}
	
	In terms of classical modular forms, this is essentially proved in \cite[\S 2]{Weissauer} for holomorphic Siegel modular forms.
	
	For a weight $\lambda$, let $\chi_\lambda$ be the infinitesimal character with the Harish-Chandra parameter $\lambda+\rho$.
	Here $\rho$ is half the sum of positive roots.
	Then the irreducible representation $L(\lambda)$ has the infinitesimal character $\chi_\lambda$ for a $\frakk$-dominant weight $\lambda$.
	Put
	\[
	\calN(G_n,\chi_\lambda) = \{\varphi \in \calN(G_n) \mid \text{there exists an integer $m$ such that $(\chi_\lambda(z)-z)^m \cdot \varphi = 0$ for any $z \in \calZ_n$}\}
	\]
	and
	\[
	\calN(G_n,\chi_\lambda)_{\{Q_{i,n}\}} = \calN(G_n,\chi_\lambda) \cap \calN(G_n)_{\{Q_{i,n}\}}.
	\]
	
	Then the computations of constant terms of nearly holomorphic modular forms imply the following:
	\begin{lem}[Lemma \ref{ss_Z}, Lemma \ref{integrality_wt}]
	Let $\chi$ be an infinitesimal character.
	\begin{enumerate}
	\item If $\calN(G_n,\chi)$ is non-zero, the infinitesimal character is integral.
	\item Any nearly holomorphic automorphic form $\varphi$ in $\calN(G_n,\chi)$ is a $\chi$-eigenfunction, i.e., for any $z \in \calZ_n$, we have $z \cdot \varphi = \chi(z)\varphi$.
	\end{enumerate}
	\end{lem}
	
	Hence Conjecture \ref{conj} (2) is true for $\Res_{F/\bbQ}\Sp_{2n}$.
	
	\begin{rem}
	For an infinitesimal character $\chi_\lambda$, if $\calN(G_n,\chi_\lambda) \neq 0$, we may choose that $\lambda$ is $\frakk_n$-dominant integral.
	Indeed, for a non-zero nearly holomorphic automorphic form $\varphi \in \calN(G_n,\chi_\lambda)$, there exists $X \in \calU(\frakg_n)$ such that $\frakp_{n,-} \cdot (X \cdot \varphi) = 0$.
	Then $X \cdot \varphi$ generates a highest weight module.
	\end{rem}

	The main result of this paper is the decomposition of $\calN(G_n,\chi_\lambda)$ as a $G_n(\bbA_{\bbQ,\fini}) \times (\frakg_n,K_{n,\infty})$-module under certain assumptions for $\lambda$.
	Set
	\[
	\mathfrak{X}_{1} = \{\mu = \otimes_v \mu_v \in \mathrm{Hom}_{\mathrm{conti}}(F^\times F_{\infty,+}^\times \bs \bbA_{F}^\times,\bbC^\times) \mid \text{$\mu_v = \mathbf{1}_v$ for any $v \in \bfa$}\}
	\]
	and
	\[
	\mathfrak{X}_{-1} = \{\mu = \otimes_v \mu_v \in \mathrm{Hom}_{\mathrm{conti}}(F^\times F_{\infty,+}^\times \bs \bbA_{F}^\times,\bbC^\times) \mid \text{$\mu_v = \mathrm{sgn}_v$ for any $v \in \bfa$}\}.
	\]
	Here $\mathbf{1}_v$ is the trivial character of $\bbR^\times$ and $\mathrm{sgn}$ is the sign character of $\bbR^\times$.
	For an infinitesimal character $\chi_\lambda$ and a positive integer $i$, we say that $\chi_\lambda$ is sufficiently regular relative to $i$ if there exists a $\frakk_n$-dominant weight $\omega = (\omega_{1,v},\ldots,\omega_{n,v})_v$ such that $\chi_\lambda = \chi_\omega $ and $\omega_{n,v} > 2n-i+1$ for any $v$.
	In the rest of this paper, for a sufficiently regular infinitesimal character $\chi_\lambda$ relative to $i$, we assume that $\lambda$ is a $\frakk_n$-dominant integral weight and $\lambda_{n,v} > 2n -i + 1$ for any $v$.
	To simplify the notation, we denote by $\mu^{\boxtimes i} \boxtimes \pi$ the cuspidal automorphic representation $\mu \boxtimes \cdots \boxtimes \mu \boxtimes \pi$ of $(\Res_{F/\bbQ}\GL_1(\bbA_F))^i \times G_{n-i}(\bbA_F)$.
	We then state the main theorem:
	
	\begin{thm}[Theorem \ref{main}]
	Fix a positive integer $i \leq n$ and a weight $\lambda$.
	\begin{enumerate}
	\item If $\calN(G_n,\chi_\lambda)_{\{Q_{i,n}\}}$ is non-zero, there exists a $\frakk_n$-dominant weight $\omega = (\omega_{1,v},\ldots,\omega_{n,v})_v$ such that $\chi_\omega = \chi_\lambda$, $\omega_{n,v} = \cdots = \omega_{n-i+1,v}$ for any $v$ and $\omega_{n,v} = \omega_{n,v'}$ for any $v, v' \in \bfa$.
	\item Suppose $\chi_\lambda$ is sufficiently regular relative to $i$ and a weight $\lambda = (\lambda_v)_v$ satisfies the following two conditions:
	\begin{itemize}
	\item $\lambda_{n,v} = \cdots = \lambda_{n-i+1,v}$ for any $v \in \bfa$.
	\item $\lambda_{n,v}$ is independent of $v \in \bfa$.
	\end{itemize}
	Let $\pi$ be an irreducible holomorphic cuspidal representation of $G_{n-i}(\bbA_F)$ with $\pi_\infty = \boxtimes_{v \in \bfa} L(\lambda_v)$ such that $\lambda_{n-i,v}>2n-i+1$ for any $v \in \bfa$.
	We then have the isomorphism
	\[
	\calN(G_n,\chi_\lambda)_{(M_{Q_{i,n}}, \mu^{\boxtimes i} \boxtimes \pi)} 
	\cong 
	\left(\Ind_{P_{i,n}(\bbA_{F,\fini})}^{G(\bbA_{F,{\fini}})}(\mu |\cdot|^{\lambda_{n,v}-n+(i-1)/2} \boxtimes \pi)\right) 
	\boxtimes \left(\boxtimes_{v \in \bfa} L(\lambda_v)\right)
	\]
	if $\mu \in \mathfrak{X}_{(-1)^{\lambda_{n,v}}}$.
	If $\mu \not\in \mathfrak{X}_{(-1)^{\lambda_{n,v}}}$, we have $\calN(G_n,\chi_\lambda)_{(M_{Q_{i,n}}, \mu^{\boxtimes i} \boxtimes \pi)} = 0$.
		\end{enumerate}
	\end{thm}
	
	\begin{rem}
	\begin{enumerate}
	\item By the theorem, the archimedean component is $\boxtimes_{v \in \bfa} L(\lambda_v)$-isotypic.
	\item If $\lambda$ is not sufficiently regular, the space $\calN(G_n,\chi_\lambda)$ is much more complicated.
	In particular, the archimedean component is not isobaric.
	For details, see \cite[Theorem 1.6]{Horinaga}.
	\end{enumerate}
	\end{rem}
	
	The theorem and its proof show that any nearly holomorphic automorphic form with a sufficiently regular infinitesimal character can be written as a sum of cusp forms and the leading term of Eisenstein series.
	We may regard the results of this paper as the first step to show Conjecture \ref{conj} (1) in the general case.
	
	The following statement follows immediately from the main theorem.
	
	\begin{cor}[Corollary \ref{suff_reg_case}]
	Let $\lambda = (\lambda_v)_v = (\lambda_{1,v},\ldots,\lambda_{n,v})_v$ be a $\frakk$-dominant integral weight such that $\lambda_{n,v}$ is independent of $v \in \bfa$ and $\lambda_{n-1,v} = \lambda_{n,v}$ for any $v \in \bfa$.
	For any positive integer $i \leq n$ and a sufficiently regular infinitesimal character $\chi_\lambda$ relative to $n$, we have
	\[
	\calN(G_n,\chi_\lambda)_{\{Q_{i,n}\}} \cong \bigoplus_{\mu,\pi} \left(\Ind_{P_{i,n}(\bbA_{F,\fini})}^{G(\bbA_{F,{\fini}})}(\mu |\cdot|^{\lambda_{n,v}-n+(i-1)/2} \boxtimes \pi)\right) 
	\boxtimes \left(\boxtimes_{v \in \bfa} L(\lambda_v)\right),
	\]
	where $\mu$ runs through all Hecke characters in $\frakX_{(-1)^{\lambda_{n,v}}}$ and $\pi$ runs through all irreducible holomorphic cuspidal representation of $G_{n-i}(\bbA_F)$ such that $\pi_\infty \cong \boxtimes_{v \in \bfa} L(\lambda_{1,v},\ldots,\lambda_{n-i,v})$.
	\end{cor}
	
	Finally we give an application to classical modular forms.
	For a $\frakk_n$-dominant integral weight $\lambda$, we denote by $(\rho_\lambda,V_\lambda)$ the irreducible representation of $K_{n,\bbC}$ with highest weight $\lambda$.
	Here $K_{n,\bbC}$ is the complexification of $K_{n,\infty}$.
	For an integral ideal $\frakn$ of $F$, put
	\[
	\Gamma (\frakn) = \left\{\begin{pmatrix}a & b \\ c & d \end{pmatrix} \in \Sp_{2n}(\calO_F) \,\middle|\, a,b,c,d \in \Mat_{n}(\calO_F), c \in \Mat_n(\frakn) \right\}.
	\]
	We say that an integral ideal $\frakn = \prod_{i} \frakp_i^{e_i}$ is square-free if $e_i \leq 1$ for all $i$.
	Here $\frakp_i$ runs through all prime ideals of $\calO_F$.
	For an irreducible finite-dimensional representation $(\rho_\lambda,V_\lambda)$ of $K_{n,\bbC}$ with highest weight $\lambda = (\lambda_{1,v},\ldots,\lambda_{n,v})_{v \in \bfa}$ and a congruence subgroup $\Gamma$ of $\Sp_{2n}(F)$, we denote by $M_{\rho_\lambda}(\Gamma)$ the space of holomorphic modular forms $f$ of weight $\rho_\lambda$ with respect to $\Gamma$.
	Put $\lambda' = (\lambda_{1,v},\ldots,\lambda_{n-1,v})_v \in \bigoplus_{v \in \bfa} \bbZ^n$.
	We define the Siegel operator $\Phi$ of a holomorphic modular form $f \in M_{\rho_\lambda}(\Gamma)$ by
	\[
	\Phi(f)(z) = \lim_{t \rightarrow \infty} f \left( \begin{pmatrix} t & 0 \\ 0 & z \end{pmatrix} \right), \qquad z \in \frakH_{n-1}^d.
	\]
	By the same method of \cite[Chap.~1]{1991_Freitag}, we may regard $\Phi(f)$ as a holomorphic modular form of weight $\rho_{\lambda'}$ with respect to $\Gamma \cap \Sp_{2(n-1)}(F)$.
	Fix a set of complete representatives $\{g_1,\ldots,g_h\}$ of $P_{1,n}(F) \bs G_n(\bbQ) / \Gamma$.
	We then define the ``global'' Siegel operator $\widetilde{\Phi}$ by
	\[
	\widetilde{\Phi}(f) = \left(\Phi(f|_{\rho}g_\ell)\right)_{1 \leq \ell \leq h}. 
	\]
	Here $|_\rho$ is the slash operator.
	Then the global Siegel operator $\widetilde{\Phi}$ defines a linear map
	\[
	\widetilde{\Phi} \colon M_{\rho_\lambda}(\Gamma) \longrightarrow \bigoplus_{\ell=1}^h M_{\rho_{\lambda'}}(\Gamma_\ell),
	\]
	where $\Gamma_\ell = \Sp_{2(n-1)}(F) \cap g_\ell \Gamma g_\ell^{-1}$.
	One of the most interesting problem is to determine the image of $\widetilde{\Phi}$.
	In general, $\widetilde{\Phi}$ is not surjective, but in many cases, $\widetilde{\Phi}$ is surjective.
	For the details, see \cite[p.123 Remark]{boe-Ib}.
	In this paper, we show the following:
	
	\begin{thm}[Theorem \ref{surjectivity}]
	Let $\lambda = (\lambda_{1,v},\ldots,\lambda_{n,v})_v$ be a $\frakk_n$-dominant integral weight.
	Suppose $\lambda_{n,v}$ is independent of $v$.
	We assume the either of the following two conditions:
	\begin{itemize}
	\item $\lambda_{n-1,v} = \lambda_{n,v}$ for any $v$.
	\item There exist $v, v' \in \bfa$ such that $\lambda_{n-1,v} \neq \lambda_{n-1,v'}$.
	\end{itemize}
	If an integral ideal $\frakn$ is square-free and $\lambda_{n,v}>2n$ for any $v$, the global Siegel operator $\widetilde{\Phi}$
	\[
	\widetilde{\Phi} \colon M_{\rho_\lambda}(\Gamma_0(\frakn)) \longrightarrow \bigoplus_\ell M_{\rho_{\lambda'}} (\Gamma_\ell)
	\]
	is surjective.
	\end{thm}
	
	\begin{rem}
	If there exist places $v,v' \in \bfa$ such that $\lambda_{n,v} \neq \lambda_{n.v'}$, the Siegel operator $\Phi$ is zero map.
	Hence we may assume that $\lambda_{n,v}$ is independent of $v$.
	If there exists $v \in \bfa $ such that $\lambda_{n-1,v} \neq \lambda_{n,v}$, the image of Siegel operator is contained in the space of cusp forms.
	However, in such a case, the image $\bigoplus_{\ell=1}^h M_{\rho_{\lambda'}}(\Gamma_\ell)$ may have non-cusp forms if $\lambda_{n-1,v}$ is independent of $v$.
	Hence the Siegel operator may not be surjective.
	\end{rem}
	
	To show this, we need to translate the global Siegel operator $\widetilde{\Phi}$ to the operator from $\calN(G_n)$ to $\calN(G_{n-1})$ and use some results in this paper.

\section{Basic properties of Fourier coefficients of nearly holomorphic modular forms}\label{FC_NHMF}
	In this section, we discuss certain properties of nearly holomorphic modular forms and its Fourier coefficients.
	
	\subsection{Definition}
	We denote by $\mathrm{Mat}_{m,n}$ the set of $m \times n$-matrices.
	Put $\Mat_{n} = \Mat_{n,n}$.
	Let $\GL_n$ and $\Sp_{2n}$ be the algebraic groups over $\bbZ$ defined by
	\[
	\GL_n(R) = \{g \in \Mat_n \mid \det g \in R^\times \}
	\]
	and
	\[
	\Sp_{2n} (R) = \left\{ g \in \GL_{2n}(R) \, \middle| \, {^t{g}}\begin{pmatrix}0_n & -\mathbf{1}_n \\ \mathbf{1}_n & 0_n \end{pmatrix} g = \begin{pmatrix}0_n & -\mathbf{1}_n \\ \mathbf{1}_n & 0_n \end{pmatrix}\right\}
	\]
	for a ring $R$, respectively.
	Set $\Sym_n = \{g \in \Mat_n \mid {^tg}=g\}$.
	Put
	\[
	\Sym^{(j)}_{n} = \left\{g = \begin{pmatrix}0_j & 0 \\ 0 & *\end{pmatrix} \in \Sym_n \,\middle|\, * \in \Sym_{n-j}\right\}
	\]
	for $0 \leq j \leq n$.
	
	For $n \in \bbZ_{\geq 1}$, set
	\[
	\frakH_n = \{z \in \Sym_n(\bbC) \mid \mathrm{Im}(z) > 0\}.
	\]
	The space $\frakH_n$ is called the Siegel upper half space of degree $n$.
	The Lie group $\Sp_{2n}(\bbR)$ acts on $\frakH_n$ by the rule
	\[
	\begin{pmatrix}a & b \\ c& d\end{pmatrix}(z) = (az+b)(cz+d)^{-1} , \qquad \begin{pmatrix}a & b \\ c& d\end{pmatrix} \in \Sp_{2n}(\bbR), \, z \in \frakH_n.
	\]
	Put
	\[
	K_{n,\infty} = \left\{g = \begin{pmatrix} a & b \\ c & d \end{pmatrix} \in \Sp_{2n}(\bbR) \, \middle| \, a = d, \, c = -b\right\}.
	\]
	Then $K_{n,\infty}$ is the group of stabilizers of $\sqrt{-1} \, \mathbf{1}_n \in \frakH_n$.
	Since the action of $\Sp_{2n}(\bbR)$ on $\frakH_n$ is transitive, we have $\frakH_n \cong \Sp_{2n}(\bbR)/ K_{n,\infty}$.
	
	We define the functions $r_{i,j}$ on $\frakH_n$ by $\mathrm{Im}(z)^{-1} = (r_{i,j}(z))_{i,j} \in \Sym_n(\bbR)$ for $z \in \frakH_n$.
	Note that $r_{i,j} = r_{j,i}$.
	For a polynomial $P$ in $n(n+1)/2$ variables with coefficients in $\bbC$, set $r_P = P((r_{i,j})_{1 \leq i \leq j \leq n})$.
	For a representation of $(\rho,V)$ of $K_{n,\bbC} \cong \GL_n(\bbC)$, the complexification of $K_{n,\infty}$, we say that a $V$-valued $C^\infty$-function $f$ is nearly holomorphic if there exists a finite collection of polynomials $P$ and $V$-valued holomorphic functions $f_P$ such that
		\[
			f(z) = \sum_P r_P(z) f_P(z), \qquad z \in \frakH_n.
		\]
		
	Let $F$ be a totally real field of degree $d$ with the ring of integers $\calO_{F}$.
	We denote by $\bfa = \{\infty_1,\ldots,\infty_d\}$ the set of embeddings of $F$ into $\bbR$.
	For a place $v$ of $F$, let $F_v$ be the $v$-completion of $F$.
	For an integral ideal $\frakn$ of $F$, set
	\[
	\Gamma(\frakn) = \left\{ \gamma \in \Sp_{2n}(\calO_F)\, \middle|\, \gamma - \mathbf{1}_{2n} \in \Mat_{2n}(\frakn)\right\}.
	\]
	The group $\Gamma(\frakn)$ is called the principal congruence subgroup of $\Sp_{2n}(F)$ of level $\frakn$.
	We say that a subgroup of $\Sp_{2n}(F)$ is a congruence subgroup if there exists an integral ideal $\frakn$ such that $\Gamma$ contains $\Gamma(\frakn)$ and $[\Gamma \colon \Gamma(\frakn)] < \infty$.
	In this section, we regard $\Sp_{2n}(F)$ as a subgroup of $\Sp_{2n}(\bbR)^d = \prod_{j=1}^d \Sp_{2n}(F_{\infty_j})$ by $\gamma \longmapsto (\infty_1(\gamma), \ldots, \infty_d(\gamma))$.
	Similarly, we regard a congruence subgroup $\Gamma$ of $\Sp_{2n}(F)$ as a subgroup of $\Sp_{2n}(\bbR)^d$.
	To simplify the notation, we denote by $K_{n,\infty}$ the maximal compact subgroup of $\Sp_{2n}(\bbR)^d$ which stabilizes $(\sqrt{-1}\,\mathbf{1}_n, \ldots, \sqrt{-1}\,\mathbf{1}_n) \in \frakH_n^d$ under the linear fractional transformation.
	Then the complexification of $K_{n,\bbC}$ is equal to $\GL_n(\bbC)^d$.
%	Let $K_{n,\bbC}$ be the complexification of $K_{n,v}$ for an archimedean place $v$.
%	Since $K_{n,\infty}$ is isomorphic to the unitary group $\mathrm{U}(n)$, $K_{n,\bbC}$ is isomorphic to $\GL_n(\bbC)$.
%	Take a finite-dimensional representation $(\rho,V)$ of $K_{n,\bbC}^d$.
	For a finite-dimensional representation $(\rho,V)$, we denote by $c_f(h,y,\gamma)$ the Fourier coefficient of $f |_\rho \gamma$ at $h$, i.e.,
	\[
	(f|_{\rho}\gamma)(z) = \sum_{h \in \Sym_n(F)} c_f(h,y,\gamma) \mathbf{e}(\tr({hz})), \qquad z \in \frakH_n^d, \, y = \mathrm{Im}(z)
	\]
	where $\mathbf{e}(\tr(hz)) = \exp(2\pi\sqrt{-1}\,\sum_{j=1}^h (\infty_j(h)z_j))$ for $(z_1,\ldots,z_d) \in \frakH_n^d$ and $h \in \Sym_n(F)$.
	We define the slash operator $|_\rho$ on $C^\infty(\frakH_n^d,V)$ by
	\[
	(f|_\rho \gamma) (z_1,\ldots,z_d) = \rho((\infty_1(r) z_1 +\infty_1(s)), \ldots, (\infty_d(r) z_d + \infty_d(s)))^{-1} f(\gamma(z_1,\ldots,z_d)),
	\]
	for $f \in C^\infty(\frakH_n,V)$, $\gamma = (\begin{smallmatrix} p & q \\ r & s\end{smallmatrix}) \in \Sp_{2n}(F),$ and $(z_1,\ldots,z_d) \in \frakH^d_n$.
	The cusp condition of $f$ means that for any $\gamma \in \Sp_{2n}(F)$ and $h \in \Sym_n(F)$ which is not positive semidefinite, we have $c_f(h,y,\gamma) = 0$ .
	We say that a $V$-valued $C^\infty$-function $f$ is a nearly holomorphic modular form of weight $\rho$ with respect to $\Gamma$ if $f$ satisfies the following conditions (NH1), (NH2), and (NH3):
		\begin{itemize}
			\item[(NH1)] $f$ is a nearly holomorphic function.
			\item[(NH2)] $f|_\rho \gamma = f$ for all $\gamma \in \Gamma$.
			\item[(NH3)] $f$ satisfies the cusp condition.
		\end{itemize}

	We denote by $N_\rho(\Gamma)$ the space of nearly holomorphic modular forms of weight $\rho$ with respect to $\Gamma$.
	By Koecher principle, we can remove the condition (NH3) if $n >1$ or $F \neq \bbQ$. 
	For the proof, see \cite[Proposition 4.1]{Horinaga} for $n>1$.
	We can give the same proof for the case of $F \neq \bbQ$.
	For simplicity, if $\rho = \det^k$, we say that a modular form of weight $\det^k$ is a modular form of weight $k$.

\subsection{Fourier coefficients of nearly holomorphic modular forms}
	Let $(\rho,V)$ be a finite-dimensional representation of $K_{n,\bbC} \cong \GL_n(\bbC)^d$.
%	We regard $\GL_n(\bbC)$ as the complexification of the compact group $\prod_{v \in \bfa}K_{n,v}$.
	Take a non-zero $V$-valued nearly holomorphic modular form $f$.
%	\[
%	(f|_\rho \gamma)(z) = \sum_h c_f(h,y,\gamma) \mathbf{e} (\tr(hz)) ,\qquad  z \in \frakH_n^d.
%	\]
%	Here, $\mathbf{e}(\tr(hz))=\exp(2\pi \sqrt{-1}\, \sum_{j=1}^d \tr(\infty_j(h)z_j))$.
	We define an embedding $\bfm$ of $\GL_n$ into $\Sp_{2n}$ by
	\[
	\bfm(a) = \begin{pmatrix} a & 0 \\ 0 & {^{t}a^{-1}}\end{pmatrix}, \qquad a \in \GL_n.
	\]
	For $j \leq n$, we also regard $\Sp_{2j}$ as the subgroup of $\Sp_{2n}$ by
	\[
	\Sp_{2j} \longrightarrow \Sp_{2n} \colon \begin{pmatrix} a & b \\ c & d \end{pmatrix} \longmapsto \begin{pmatrix} \mathbf{1}_{n-j} &&&\\ &  a & & b \\ && \mathbf{1}_{n-j} & \\ & c & & d \end{pmatrix}.
	\]
	We denote by $\Sym_{n}(\bbR)_{>0}$ the set of positive-definite real symmetric matrices.
	
	\begin{lem}\label{lem_inv_FC}
	Suppose $f$ is a nearly holomorphic modular form of weight $\rho$ with respect to a congruence subgroup $\Gamma$.
	For $\bfm(a) \in \bfm(\GL_n(F)) \cap \Gamma$ with $a \in \GL_n(F)$, we have
	\[
	c_f(h,y,1) = \rho({^{t}a}) c_f({^{t}a^{-1}}h{a^{-1}},ay{^{t}a},1), \qquad y \in (\Sym_n(\bbR)_{>0})^d.
	\]
	Here for $y = (y_1,\ldots,y_d)$, we put $ay{^{t}a}=(\infty_1(a)y_1\cdot \infty_1({^ta}),\ldots, \infty_d(a)y_d \cdot \infty_d({^ta}))$ with $\bfa = \{\infty_1,\ldots,\infty_d\}$.
	\end{lem}
	\begin{proof}
	By the definition of $f$, we have
	\[
	(f|_\rho \bfm(a))(z) = f(z),\qquad z \in \frakH_n^d
	\]
	for $\bfm(a) \in \bfm(\GL_n(F)) \cap \Gamma $.
	The Fourier expansion of the left hand side is
	\[
	\sum_h \rho({^{t}a}) c_f(h,ay{^{t}a},1) \bfe (\tr(haz{^{t}a})) = \sum_h \rho({^{t}a}) c_f(h,ay{^{t}a},1) \bfe (\tr({^{t}a}haz)).
	\]
	Replace $h$ to ${^{t}a^{-1}}ha^{-1}$.
	We then have
	\[
	(f|_\rho \bfm(a)) (z) = \sum_h \rho({^{t}a}) c_f({^{t}a^{-1}}ha^{-1},ay{^{t}a},1) \bfe(\tr(hz)).
	\]
	Hence for any $h \in \Sym_n(F)$, we have
	\[
	\rho({^{t}a}) c_f({^{t}a^{-1}}ha^{-1},ay{^{t}a},1) = c_f(h,y,1), \qquad y \in (\Sym_n(\bbR)_{> 0})^d.
	\]
	This completes the proof.
	\end{proof}
	
	Let $P(\Sym_n(\bbC)^d,V)$ be the space of all $V$-valued polynomials on $\Sym_n(\bbC)^d$ with coefficients in $\bbC$.
	For any Fourier coefficient $c_f(h,y,1)$, there exists a polynomial $P_{f,h} \in P(\Sym_n(\bbC)^d,V)$ such that
	\[
	P_{f,h}((y_1^{-1},\ldots,y_d^{-1})) = c_f(h,(y_1,\ldots,y_d),1), \qquad  (y_1,\ldots,y_d) \in (\Sym_n(\bbR)_{>0})^d,
	\]
	by the definition of near holomorphy of $f$.
	We suppose that $\GL_n(\bbC)^d$ acts on $P(\Sym_n(\bbC)^d,V)$ by the rule
	\[
	(a \cdot P)(z) = \rho(a) P(({a_1^{-1}}z_1{^{t}a_1^{-1}},\ldots,{a_d^{-1}}z_d{^{t}a_d^{-1}})), \qquad a = (a_1,\ldots,a_d) \in \GL_n(\bbC)^d, P \in P(\Sym_n(\bbC)^d,V).
	\]
	We denote by $\sigma$ this representation of $\GL_n(\bbC)^d$ on $P(\Sym_n(\bbC)^d,V)$.
	Then, if $\bfm({^{t}a}) \in \bfm(\GL_n(F)) \cap \Gamma$ and ${^ta}^{-1}ha^{-1} = h$, we have
	\[
	P_{f,h}((y_1^{-1},\ldots,y_d^{-1})) = (\sigma(a) P_{f,h}) ((y_1^{-1},\ldots,y_d^{-1})), \qquad (y_1,\ldots,y_d) \in (\Sym_n(\bbR)_{>0})^d.
	\]
	
	\begin{lem}\label{zariski_dense}
	The set $(\Sym_n(\bbR)_{>0})^d$ is Zariski dense in $\Sym_{n}(\bbC)^d$ over $\bbC$.
	\end{lem}
	\begin{proof}
	Take a polynomial $f$ in the variable $x = ((x_{i,j}^{(k)})_{1 \leq i,j\leq n})_{1 \leq k \leq d} \in \Sym_{n}(\bbC)^d$.
	Suppose that $f$ is zero on $(\Sym_{n}(\bbR)_{>0})^d$.
	We assume that the degree of $f$ as a polynomial in $x_{i,j}^{(k)}$ is at most $t_{i,j}^{(k)}>0$.
	We now take sets $S_{i,j}^{(k)}$ as follows:
	If $i \neq j$, put 
	\[
	S_{i,j}^{(k)} = \left\{1, 2, \ldots, t_{i,j}^{(k)} + 1\right\}.
	\]
	If $i = j$, put 
	\[
	S_{i,i}^{(k)} = \left\{n\left(\max_{p \neq q} t_{p,q}^{(k)}\right)^2, 1+ n\left(\max_{p \neq q} t_{p,q}^{(k)}\right)^2, \ldots, t_{i,i} + n\left(\max_{p \neq q} t_{p,q}^{(k)}\right)^2\right\}.
	\]
	Then $\#(S_{i,j}^{(k)})$ is equal to $t_{i,j}^{(k)} + 1$ and for any $x_{i,j}^{(k)} \in S_{i,j}^{(k)}$ with $x_{i,j}^{(k)} = x_{j,i}^{(k)}$, the symmetric matrix $(x_{i,j}^{(k)})_{1\leq i, j \leq n}$ is a positive definite real matrix.
	Hence the polynomial $f$ is zero on
	\[
	\prod_{1 \leq i \leq j \leq n, 1 \leq k \leq d} S_{i,j}^{(k)} \subset (\Sym_{n}(\bbR)_{>0})^d.
	\]
	By \cite[Lemma 2.1]{Noga-Alon}, we have $f \equiv 0$.
	This completes the proof.
 	\end{proof}
 	
 	By Lemma \ref{zariski_dense}, a polynomial $P_{f,h}$ is uniquely determined by $f$ and $h$, and we have
 	\[
 	P_{f,h} \equiv (\sigma(a) P_{f,h})
 	\]
 	as a polynomial on $\Sym_n(\bbC)^d$ for any $\bfm({^{t}a}) \in \bfm(\GL_n(F)) \cap \Gamma$ with ${a}^{-1}h{^{t}a^{-1}} = h$.
	
	Let $\SL_j$ and $N_{j,n,\GL}$ be the subgroup and a unipotent subgroup of $\GL_n$ defined by
	\[
	\SL_j = \left\{\begin{pmatrix} a & 0 \\ 0 & \mathbf{1}_{n-j}\end{pmatrix} \, \middle | \, a \in \SL_{j}\right\}
	\]
	and
	\[
	N_{j,n,\GL} = \left\{\begin{pmatrix} \mathbf{1}_j & * \\ 0 & \mathbf{1}_{n-j}\end{pmatrix} \in \GL_n \,\middle|\, * \in \Mat_{j,n-j}\right\},
	\]
	respectively.
	Put $N_{j,n,\GL}^\op = \{{^tg} \mid g \in N_{j,n,\GL}\}$.
	
	Set $F_\infty = \prod_{v \in \bfa}F_v$.
	Take $h \in \Sym_n^{(j)}(F)$.
	Then, we have 
	\[
	{a}^{-1}h{^{t}a^{-1}} = h
	\] 
	for any $a \in \SL_j(F_\infty)N_{j,n,\GL}(F_\infty)$.
	By Lemma \ref{lem_inv_FC}, the polynomial $P_{f,h}$ is fixed by a subgroup $\bfm^{-1}(\Gamma \cap \bfm(\SL_j(F_\infty)N_{j,n,\GL}(F_\infty)))$ of $\GL_n(F)$ under the action $\sigma$.
	Since $\bfm^{-1}(\Gamma \cap \bfm(N_{j,n,\GL}(F_\infty)))$ is Zariski dense in $N_{j,n,\GL}(F_\infty)$ and $\bfm^{-1}(\Gamma \cap \bfm(\SL_j(F_\infty)))$ is Zariski dense in $\SL_j(F_\infty)$ over $\bbR$ by the Borel density theorem \cite{1960_Borel}, the polynomial $P_{f,h}$ is fixed by $\SL_{j}(F_\infty)N_{j,n,\GL}(F_\infty)$.
	Indeed, the representation $\sigma$ of $\GL_n(\bbC)$ on $P(\Sym_n(\bbC))$ is rational and hence the restriction $\sigma|_{\SL_j(F_\infty)N_{j,n,\GL}(F_\infty)}$ is a polynomial representation.
	We then obtain the following result:
	
	\begin{prop}\label{Fourier_coeff}
	We suppose a symmetric matrix $h$ lies in $\Sym_n^{(j)}(F)$.
	Then, we have
	\[
	\rho({^ta}) c_f(h,{a}y{^ta},1) = c_f(h,y,1), \qquad y \in (\Sym_n(\bbR)_{>0})^d
	\]
	for any $a \in \SL_{j}(F_\infty)N_{j,n,\GL}(F_\infty)$.
	\end{prop}
	
	\begin{rem}
	If $f$ is holomorphic, the similar result is already known by \cite[p.~188]{Weissauer} and \cite[Chap.~1]{1991_Freitag}.
	The proof in this paper is totally same as their proofs.
	\end{rem}
	
	Let $j$ be the $\GL_n(\bbC)^d$-valued function on $\Sp_{2n}(\bbR)^d \times \frakH_n^d$ defined by 
	\[
	j((\begin{smallmatrix}p_1&q_1\\ r_1&s_1\end{smallmatrix}),\ldots,(\begin{smallmatrix}p_d&q_d\\ r_d&s_d\end{smallmatrix}),z_1,\ldots,z_d) = (r_1z_1+d_1,\ldots,r_dz_d+s_d).
	\]
	Then for $a = (a_1,\ldots,a_d) \in \GL_n(\bbR)^d$ and $z = (z_1,\ldots,z_d) \in \frakH_n^d$, we have $j(\bfm(a),z) = ({^ta_1^{-1}},\ldots,{^ta_{d}^{-1}})$.
%	Hence for any positive semidefinite matrix $h \in \Sym_n(F)$ of the form
%	\[
%	\begin{pmatrix}0_j & 0 \\ 0 & * \end{pmatrix}, \qquad * \in \Mat_{n-j}(F),
%	\]
%	the Fourier coefficient $c_f(h,y,1)$ is fixed by the group $N_{j,n,\GL}^\op(\bbR) = j(z,N_{j,n,\GL}(\bbR))$.
%	Note that for $a \in m(\GL_n(\bbR))$, $j(z,a)$ does not depend on $z$.
	The isomorphism classes of irreducible holomorphic finite-dimensional representations of $\GL_n(\bbC)^d$ is equal to the set of  highest weights.
	The set of highest weights is equal to 
	\[
	\left\{(\lambda_{1,v},\ldots,\lambda_{n,v})\in \bigoplus_{v \in \bfa} \bbZ^n \,\middle|\, \text{$\lambda_{1,v} \geq \cdots \geq \lambda_{n,v}$ for any $v \in \bfa$}\right\}.
	\]
	We denote by $\rho_\lambda$ the irreducible representation corresponding to the highest weight $\lambda$. 
	The choice of a positive root system is the same as in (\ref{pos_root_sys}).
	Put
	\[
	H_{j,n}^\bfa = \prod_{v \in \bfa}\SL_j N_{j,n,\GL}, \qquad {^tH_{j,n}^\bfa} = \prod_{v \in \bfa}\SL_jN_{j,n,\GL}^\op.
	\]
	Suppose $f$ is a holomorphic modular form on $\frakH_n^d$ of weight $(\rho_\lambda,V)$.
	Then $c_f(h) = c_f(h,y,1) \in V$ is meaningful, since $c_f(h,y,1)$ is independent of $y$.
	Hence $c_f(h)$ lies in the space of ${^tH_{j,n}^\bfa}$-fixed vectors in $V$.
	For a subgroup $H$ of $\GL_n(\bbC)^d$, we denote by $V^H$ the space of $H$-fixed vectors in $V$.
	Since $\rho_\lambda$ is a holomorphic representation, $c_f(h)$ lies in $V^{^tH^\bfa_{j,n}(\bbC)}$.
	We regard $\GL_j\times\GL_{n-j}$ as a subgroup of $\GL_n$ by
	\[
	\GL_j\times\GL_{n-j} = \left\{\begin{pmatrix}a&0\\0&d\end{pmatrix} \, \middle|\, a \in \GL_j , \, d \in \GL_{n-j}\right\}.
	\]
	As a $\GL_j(\bbC)^d \times \GL_{n-j}(\bbC)^d$-representation, $V^{{^tH_{j,n}^\bfa}(\bbC)}$ is irreducible, since the space of lowest weight vectors in $V^{^tH_{j,n}^\bfa}$ as a $\GL_j(\bbC)^d \times \GL_{n-j}(\bbC)^d$-representation is a one-dimensional space.
	Moreover, by the $\SL_j(\bbC)$-invariance, if $c_f(h) \neq 0$, we have $\lambda_{n,v} = \cdots = \lambda_{n-j+1,v}$ for any $v$.
	Note that the highest weight of $V^{\SL_j(\bbC)N_{j,n,\GL_n}^\op(\bbC)}$ is 
	\[
	\boxtimes_{v\in \bfa}((\lambda_{n-j+1,v},\ldots,\lambda_{n,v}) \boxtimes (\lambda_{1,v},\ldots,\lambda_{n-j,v}))
	\]
	as a $\GL_j(\bbC)^d \times \GL_{n-j}(\bbC)^{d}$-representation.
	For the details of the above discussion, see \cite[\S 2]{Weissauer}.
	Suppose $f$ is a modular form with respect to $\Gamma(\frakn)$ for an integral ideal $\frakn$ of $F$.
	We consider the group $\Gamma(\frakn) \cap \bfm(\GL_j(F))$.
	In this case, we have
	\[
	\rho_\lambda(a)c_f(h) = c_f(h).
	\]
	Hence we obtain $\prod_{v \in \bfa}(\det(v(a)))^{\lambda_{n,v}}c_f(h) = c_f(h)$, by $\lambda_{n,v} = \cdots = \lambda_{n-j+1,v}$ for any $v$.
	Since the image $\det(\Gamma\cap \bfm(\GL_j(F)))$ contains $1+\frakn$, the weight $\lambda_{n,v}$ is independent of $v$ by the Dirichlet's unit theorem.
	This is an analogue of Remark 4.8 in Chapter 1 of \cite{1990_Freitag}.
	Summarizing the above discussion, we have the following:
	
	\begin{prop}\label{rigidity_HMF}
	Let $f$ be a holomorphic modular form of weight $\rho_\lambda$.
	Suppose $c_f(h)$ is non-zero for some $h \in \Sym^{(j)}_n(F)$.
	Then we have $\lambda_{n,\infty_1} = \cdots = \lambda_{n,\infty_d}$ and $\lambda_{n,v} = \cdots = \lambda_{n-j+1,v,}$ for any $v\in\bfa$.
	\end{prop}
	
\section{An Embedding of a highest weight representation of $\Sp_{2n}(\bbR)$ into certain parabolic inductions}\label{local_arch}
	In this section, we prove some technical statements for induced representations of the real symplectic group of rank $n$.
	
	\subsection{Irreducible highest weight representations}
	
	Set $G_n = \Res_{F/\bbQ}\Sp_{2n}$ where $\Res$ is the Weil restriction and $F$ is a totally real field.
	In this section, we assume $F = \bbQ$ for simplicity.
	Put $\frakg_n = \mathrm{Lie}(G_n(\bbR)) \otimes_\bbR \bbC$ and $\frakk_n = \mathrm{Lie}(K_{n,\infty})\otimes_\bbR\bbC$.
	We denote by $\calZ_n$ the center of the universal enveloping algebra $\calU(\frakg_n)$.
	Note that in this setting, $G_n(\bbR)$ is equal to $\Sp_{2n}(\bbR)$ as $F=\bbQ$.
	We then obtain the well-known decomposition
	\[
	\frakg_n =  \frakk_n + \frakp_{n,+} +\frakp_{n,-}
	\]
	where $\frakp_{n,+}$ (resp.~$\frakp_{n,-}$) is the Lie subalgebra of $\frakg_n$ corresponding to the holomorphic tangent space (resp.~anti-holomorphic tangent space) of $\frakH_n$ at $\sqrt{-1}\, \mathbf{1}_n$.
	It is well-known that the Lie algebras $\frakg_n$ and $\frakk_n$ have the same Cartan subalgebra.
	We take such a Cartan subalgebra.
	Then the root system of $\frakg_n$ is 
	\[
	\Phi = \{ \; \pm (e_i + e_j), \; \pm(e_k - e_l), \quad 1\leq i \leq j \leq n, 1 \leq k < l \leq n \; \}.
	\]
	We put the set
	\begin{align}\label{pos_root_sys}
	\Phi^+ = \{ \; - (e_i + e_j), \; e_k - e_l, \quad 1\leq i \leq j \leq n, 1 \leq k < l \leq n \; \}
	\end{align}
	to be a positive root system.

	Let $\rho$ be half the sum of positive roots.
	Put 
	\[
	\Lambda = \{ \lambda = (\lambda_1, \ldots, \lambda_n) \in \bbC^n \mid \lambda_i - \lambda_{i+1} \in \bbZ, \; i = 1,\ldots,n-1 \}.
	\] 
	We say that a weight $\lambda = (\lambda_1,\ldots,\lambda_n) \in \Lambda$ is $\frakk_n$-dominant if $\lambda_i - \lambda_{i+1} \in \bbZ_{\geq 0}$ for any $1 \leq i \leq n-1$.
	Let $\Lambda^+$ be the set of $\frakk_n$-dominant weights.
	For a $\frakk_n$-dominant weight $\lambda \in \Lambda^+$, we denote by $\rho_\lambda$ an irreducible $\calU(\frakk_n)$-module with highest weight $\lambda$.
	Let $V_\lambda$  be any model of $\rho_\lambda$.
	We regard $V_\lambda$ as a module for $\frakp_n = \frakp_{n,-} + \frakk_n$ by letting $\frakp_{n,-}$ act trivially.
	Set 
	\[
	N(\lambda) = \calU(\frakg_n) \otimes_{\calU({\frakp_n})} V_\lambda .
	\]
	Then $N(\lambda)$ has a natural structure of a left $\calU(\frakg_n)$-module.
%	To simplify the notation, we identify the representation $\rho_\lambda$ as its representation space $V_\lambda$.
	The modules $N(\lambda)$ are often called the generalized Verma module with highest weight $\lambda$ with respect to a parabolic subalgebra  $\frakp_n$.
	It is well-known that the module $N(\lambda)$ has a unique irreducible quotient $L(\lambda)$.
%	We denote by $N(\lambda)^*$ the contragredient module of $N(\lambda)$ in the sense of \cite[section 3.2]{cat_o}.
%	Note that, as a $\calU(\frakk)$-module, $N(\lambda)$ is isomorphic to $N(\lambda)^*$.
	We denote by $\chi_\lambda$ an infinitesimal character of $L(\lambda)$.
	In this setting, $\chi_\lambda = \chi_\mu$ is equivalent to $\lambda = w \cdot \mu$ for some $w \in W_{n}$,
where $W_n$ is the Weyl group of $\frakg_n$ and $w \cdot \lambda = w(\lambda + \rho) - \rho$.
	Here $w \cdot \lambda$ is called the dot-action.
	See \cite[\S 1.8]{cat_o}.
	Note that by $\chi_0$ we mean the infinitesimal character of the trivial representation of $\calU(\frakg_n)$.
	For a $\calU(\frakg_n)$-module $\pi$, we say that $\pi$ has an infinitesimal character if there exists an infinitesimal character $\chi$ such that $z \cdot v = \chi(z)v$ for any $v \in \pi$ and $z \in \calZ_n$.
	We also say that $\pi$ has an infinitesimal character $\chi$ if we have $z \cdot v = \chi(z)v$ for any $v \in \pi$ and $z \in \calZ_n$.

	By the universality of $N(\lambda)$, any highest weight module of highest weight $\lambda$ is unique up to isomorphism and is isomorphic to the irreducible quotient $L(\lambda)$.
	For a $(\frakg_n, K_{n,\infty})$-module $\pi$, put
	\[
	HK(\pi) = \{v \in \pi \mid \frakp_{n,-} \cdot v = 0\}.
	\]
	Then $HK(\pi)$ is stable under the action of $K_{\infty,n}$, by $\mathrm{Ad}(K_{\infty,n})(\frakp_{n,-}) = \frakp_{n,-}$ and the definition of $(\frakg_n, K_{n,\infty})$-module.
	If $\pi$ is irreducible, $HK(\pi)$ is an irreducible $K_{n,\infty}$-representation.
	Moreover when $\pi \cong L(\lambda)$, $HK(\pi)$ is the irreducible representation with highest weight $\lambda$.
	If $\lambda = (\lambda_1, \ldots, \lambda_n) \in \Lambda^+ \cap \bbZ^n$ and $\lambda_n \geq n$, the irreducible representation $L(\lambda)$ is isomorphic to the holomorphic discrete series representation of weight $\lambda$.
	
	In this section, we study an embedding of $L(\lambda)$ into certain induced representations.
	The results play an essential role in this paper.
	Indeed, when we construct an Eisenstein series which is annihilated by $\frakp_{n,-}$, we need to consider a highest weight vector in induced representations.

	\subsection{The first reduction point and unitarizability}
	
	Recall the definition of the first reduction point in the sense of \cite{88_EHW}.
	Let $\lambda=(\lambda_1,\ldots,\lambda_n)$ be a $\frakk_n$-dominant weight with $\lambda_n=n$.
	We say that a real number $r_0$ is the first reduction point if the module $N(\lambda+r_0(-1,\ldots,-1))$ is reducible and $N(\lambda+r(-1,\ldots,-1))$ is irreducible for $r<r_0$.
	
	\begin{thm}[\cite{88_EHW} Theorem 2.10]\label{first_red_pt}
	Let $\lambda=(\lambda_1,\ldots,\lambda_n)$ be a $\frakk_n$-dominant weight with $\lambda_n=n$.
	Set $p(\lambda)=\#\{i\mid\lambda_i=n\}$ and $q(\lambda)  = \#\{i\mid \lambda_i = n+1\}$.
	Then, the first reduction point $r_0$ equals to $(p(\lambda)+q(\lambda)+1)/2$.
	\end{thm}
	
	Let $r_0$ be the first reduction point.
	Then for $r < r_0$, the irreducible representation $L(\lambda+r(-1,\ldots,-1))$ is unitary.
	More precisely, we have the following:
	
	\begin{thm}[\cite{88_EHW} Theorem 2.8]\label{unitary}
	With the same notation as in Theorem \ref{first_red_pt}, $L(\lambda + r(-1,\ldots,-1))$ is unitary if and only if either of the following conditions holds:
	\begin{itemize}
	\item $r \leq (p(\lambda)+q(\lambda)+1)/2$.
	\item $\lambda \in (1/2)\bbZ^n$ and $r \leq p(\lambda)+q(\lambda)/2$.
	\end{itemize}
	\end{thm}
	
	By this unitarizability criterion, we obtain the vanishing of the space of nearly holomorphic modular forms.
	See Proposition \ref{vanishing}.
	
	\subsection{Yamashita's result}
	Let $B_n$ be a Borel subgroup of $G_n$ defined by
	\[
	B_n = \left\{ \begin{pmatrix} a & b \\ 0_n & {^t a^{-1}} \end{pmatrix} \in G_n\,\middle| \, 
	\begin{matrix} \text{$a$ is an upper triangular matrix}\\ \text{and $b \in \Sym_n(\bbR)$}\end{matrix}\right\}.
	\]
	We denote by $T_n$ the diagonal subgroup of $B_n$ and by $N_n$ the unipotent subgroup of $B_n$.
	For a character $\mu$ of $T_n$, set 
	\begin{align*}
	I_{B_n}(\mu) = \Ind_{B_n(\bbR)}^{G_n(\bbR)} (\mu) = \{f \colon G_n \longrightarrow \bbC \mid f(u tg) = \mu(t)\delta^{1/2}_{B_n}(t)f(g), \quad u \in N_n, \, t \in T_n, \, g \in G_n\}.
	\end{align*}
	Here $\delta_{B_n}$ is the modulas character of $B_n$.
	This is called a principal series representation of $G_n$.
	For a character $\mu$ of $T_n$, there exist characters $\mu_1, \ldots, \mu_n$ of $\bbR^\times$ such that $\mu(\mathrm{diag}(t_1,\ldots,t_n,t_1^{-1},\ldots,t_n^{-1})) = \mu_1(t_1) \cdots \mu_n(t_n)$.
	In this case, we write $\mu = \mu_1 \boxtimes \cdots \boxtimes \mu_n$.
	If a character $\mu_1$ of $\bbR^\times$ is the identity map, we denote by $\mu_1 = \det$.
	Then Yamashita proved the following:
	
	\begin{thm}[Theorem 2.6 \cite{Yamashita}]\label{Yamashita}
	Let $\mu$ be a character of $T_n$.
	Then the principal series representation $I_{B_n}(\mu)$ contains a highest weight vector of weight $\lambda = (\lambda_1, \ldots, \lambda_n)$ if and only if $\mu\delta^{1/2}_{B_n} = \det^{\lambda_n} \boxtimes \cdots \boxtimes \det^{\lambda_1}$.
	Moreover if $I_{B_n}(\mu)$ contains a highest weight vector $v$ of weight $\lambda$, the vector $v$ is unique up to constant multiple and generates $L(\lambda)$.
	\end{thm}
	
	For a $\frakk_n$-dominant weight $\lambda$, let $I_{B_n}(\lambda)$ be the principal series representation such that it contains a highest weight vector of weight $\lambda$.
	Note that we have
	\begin{align}\label{princ_ser}
	I_{B_n}((\lambda_1,\ldots,\lambda_n)) = I_{B_n} \left(\sgn^{\lambda_n}|\cdot|^{\lambda_n - n} \boxtimes \sgn^{\lambda_{n-1}}|\cdot|^{\lambda_{n-1} - n +1} \boxtimes \cdots \boxtimes \sgn^{\lambda_1}|\cdot|^{\lambda_1-1}\right).
	\end{align}
	
	We now obtain the embeddings of highest weight vectors and irreducible highest weight representations into principal series representations.
	For a parabolic subgroup $P=MN$ of $G_n$, we say that $P$ is standard if $P$ contains the Borel subgroup $B_n$.
	The Levi subgroup $M$ is called standard if $M$ contains $T_n$.
	Then for any character $\mu$ of $T_n$, there exists a representation $\pi$ of $M$ such that the induced representation $\Ind_{P}^{G_n(\bbR)} (\pi)$ is contained in $I_{B_n}(\mu)$.
	Then we have the following question:
	
	\begin{quest}\label{emb_highest_weight_rep}
	For a $\frakk_n$-dominant weight $\mu$ and a standard parabolic subgroup $P$ with $P=MN$, can we find a representation $\pi$ of $M$ such that $\Ind_{P}^{G_n(\bbR)}(\pi)$ contains a highest weight vector of weight $\mu$?
	If we find such a representation $\pi$, can we assume that $\pi$ is irreducible?
	\end{quest}
	
	If $n=1$, the answer of this question is well-known.
	For example, see \cite[\S 2]{muic}.
	In \cite{muic}, G.~Mui\'{c} gives the complete answer for this question for $n=2$.
	Moreover, he determines the socle series of parabolic inductions of $G_2(\bbR) = \Sp_{4}(\bbR)$.
	For the case of degenerate principal series representations, see the next subsection.
	The author does not have the complete answer for this question.
	In Lemma \ref{emb_lem} and Theorem \ref{emb_main}, we give a partial answer.
	
	\subsection{The case of degenerate principal series representations}
	
	We denote by $P_{i,n}$ a standard parabolic subgroup of $G_n$ with the standard Levi subgroup $\Res_{F/\bbQ}\GL_i \times G_{n-i}$.
	If $i=n$, the parabolic subgroup $P_{n,n}$ is called the Siegel parabolic subgroup.
	For a unitary character $\mu$ of $\GL_n(\bbR)$ and a complex number $s$, put 
	\[
	I_{P_{n,n}}(\mu,s) = \Ind_{P_{n,n}(\bbR)}^{G_n(\bbR)} \left(\mu|\det|^s\right).
	\]
	The representation $I_{P_{n,n}}(\mu,s)$ is called the degenerate principal series representation.
	For simplicity, for a character $\mu$ of $\bbR^\times$, we denote by $\mu$ the character $\mu \circ \det$ of $\GL_n(\bbR)$.
	S.T.~Lee determines the algorithm which determines the socle series of $I_{P_{n,n}}(\mu,s)$ in \cite{ST_Lee}.
	By his result, we immediately obtain the following result:
	
	\begin{lem}\label{case_deg_princ_ser}
	Let $\mu$ be a unitary character of $\bbR^\times$.
	Then $I_{P_{n,n}}(\mu,s)$ contains a highest weight vector of weight $\lambda = (\lambda_1, \ldots, \lambda_n)$ if and only if all of the following conditions hold: 
	\begin{itemize}
	\item $\lambda_1 = \cdots =\lambda_n$. 
	\item $\mu=\mathrm{sgn}^{\lambda_n}$.
	\item $s = \lambda_n -(n+1)/2$.
	\end{itemize}
	\end{lem}
	
	This lemma is a partial answer of Question \ref{emb_highest_weight_rep}.
	We denote by $I_{P_{n,n}}(\lambda)$ the degenerate principal series representation of $G_n(\bbR)$ which contains a highest weight vector of weight $\lambda$.
	Note that $I_{P_{n,n}}(\lambda)$ is contained in $I_{B_n}(\lambda)$.
	In the next subsection, we give a generalization of the above lemma.

	\subsection{An embedding of $L(\lambda)$ into certain parabolic inductions}
	
	For a $\frakk_n$-dominant weight $\lambda$ with $\lambda=(\lambda_1,\ldots,\lambda_n) \in \bbZ^n$, let $r_1, \ldots, r_j$ be integers such that
	\[
	\lambda_1 = \cdots = \lambda_{r_1} \neq \lambda_{r_1+1} = \cdots = \lambda_{r_1+r_2} \neq \cdots \neq \lambda_{r_1+\cdots +r_{j-1} +1} = \cdots = \lambda_{r_1 +\cdots +r_{j}}.
	\]
	Put $s_j = \sum_{\ell = 1}^j r_\ell$.
	We denote by $\GL_\lambda$ a standard Levi subgroup of $G_n$ of the form
	\[
	\GL_{\lambda} 
	= \left\{
	\left(\begin{array}{@{}c|c@{}}
     \begin{matrix}
     a_j &			&			&		\\
     	&a_{j-1}	&			&		\\
     	&			& \ddots 	&		\\
     	&			&			& a_1	\\
     \end{matrix}
     &
     
	\\
	\hline
    
	&
     \begin{matrix}
    {^t{a}_j^{-1}}&					&			&				\\
     			&{^ta}^{-1}_{j-1}	&			&				\\
     			&					& \ddots 	&				\\
     			&					&			& {^ta}_1^{-1}	\\
     \end{matrix}
	\end{array}
	\right)
	\,\middle|\,
	a_j \in \GL_{r_j}, \ldots, a_1 \in \GL_{r_1}\right\}.
	\]
	This is isomorphic to $\GL_{r_j}\times \cdots\times\GL_{r_1}$.
	Let $P_{\lambda, n}$ be a standard parabolic subgroup of $G_n$ with the Levi subgroup $\GL_\lambda$.
	Then the induced representation 
	\[
	I_{P_{\lambda,n}}(\lambda) = \Ind_{P_{\lambda,n}(\bbR)}^{G_n(\bbR)} \left(\delta_{P_{\lambda,n}}^{-1/2} \otimes \left( \mathrm{det}_{\GL_{r_j}}^{\lambda_{s_j}} \boxtimes \cdots \boxtimes \mathrm{det}_{\GL_{r_1}}^{\lambda_{s_1}}\right)\right)
	\] 
	is a subspace of the principal series representation $I_{B_{n}} (\lambda)$ (see (\ref{princ_ser})).
	We then obtain the following result:
	
	\begin{lem}\label{emb_lem}
	The induced representation $I_{P_{\lambda,n}}(\lambda)$ contains a highest weight vector of weight $\lambda$.
	Moreover such a highest weight vector generates $L(\lambda)$.
	\end{lem}
	
	\begin{proof}
	Take a vector $f$ in $HK(I_{B_n}(\lambda))$.
%	It suffices to show that $f$ is invariant under the left translation of $\SL_{r_j}(\bbR) \times \cdots \times \SL_{r_1}(\bbR)$ which is the derived subgroup of $\GL_\lambda(\bbR)$.
	We denote by $G_\lambda$ the subgroup of $G_n$ defined by
	
	\[
	G_{\lambda} (\bbR)
	= \left\{
	\left(\begin{array}{@{}c|c@{}}
     \begin{matrix}
     a_j &			&			&		\\
     	&a_{j-1}	&			&		\\
     	&			& \ddots 	&		\\
     	&			&			& a_1	\\
     \end{matrix}
     &
     \begin{matrix}
     b_j &			&			&		\\
     	&b_{j-1}	&			&		\\
     	&			& \ddots 	&		\\
     	&			&			& b_1	\\
	\end{matrix}
	\\
	\hline
      \begin{matrix}
      c_j &			&			&		\\
     	&c_{j-1}	&			&		\\
     	&			& \ddots 	&		\\
     	&			&			& c_1	\\
	\end{matrix}
	&
     \begin{matrix}
    d_j &			&			&		\\
     	&d_{j-1}	&			&		\\
     	&			& \ddots 	&		\\
     	&			&			& d_1	\\
     \end{matrix}
	\end{array}\right)
	 \,\middle|\, 
	 \begin{pmatrix}a_j & b_j \\ c_j & d_j \end{pmatrix} \in G_{r_j}(\bbR), \ldots, \begin{pmatrix}a_1 & b_1 \\ c_1 & d_1 \end{pmatrix} \in G_{r_1}(\bbR)\right\}.
	\]
	This group is isomorphic to $G_{r_j}(\bbR) \times \cdots \times G_{r_1}(\bbR)$.
	We may regard $P_{\lambda,n}(\bbR) \cap G_\lambda(\bbR)$ and $B_n(\bbR) \cap G_\lambda(\bbR)$ as the Siegel parabolic subgroup of $G_\lambda(\bbR)$ and the Borel subgroup of $G_\lambda(\bbR)$, respectively.
	For $k \in K_{n,\infty}$, we define the function $f_k$ on $G_\lambda(\bbR)$ by $f_k(m) = f(mk)$.
	Then $f_k$ lies in the principal series representation of $G_\lambda(\bbR)$.
	More precisely, $f_k$ lies in 
	\[
	\bigotimes_{\ell=1}^j I_{B_{r_{j+1-\ell}}}((\lambda_{s_{j+1-\ell}},\ldots,\lambda_{s_{j+1-\ell}})).
	\]
	The Lie subalgebra $\frakp_{r_j,-} \oplus \cdots \oplus\frakp_{r_1,-}$ of $\mathrm{Lie}(G_\lambda(\bbR))\otimes_\bbR\bbC = \frakg_{r_j} \oplus \cdots \oplus \frakg_{r_1}$ acts on $f_k$ trivially for any $k \in K_{n,\infty}$.
	Indeed, for $X \in \frakp_{r_j,-} \oplus \cdots \oplus\frakp_{r_1,-}$, we have
	\[
	X \cdot f_k = X \cdot r(k) f = r(k)((\mathrm{Ad}(k^{-1})X)f) = 0
	\]
	by $\mathrm{Ad}(K_{n,\infty})\frakp_{n,-} = \frakp_{n,-}$ and the assumption for $f$.
	Here $X \cdot f_k$ is the action as the representation $\bigotimes_{\ell=1}^j I_{B_{r_{j+1-\ell}}}((\lambda_{s_{j+1-\ell}},\ldots,\lambda_{s_{j+1-\ell}}))$ and $r$ is the right translation.
	Hence $f$ lies in
	\[
	HK \left( \bigotimes_{\ell=1}^j I_{B_{r_{j+1-\ell}}}((\lambda_{s_{j+1-\ell}},\ldots,\lambda_{s_{j+1-\ell}})) \right).
	\]
	For any $\ell$, we have
	\[
	HK \left(I_{B_{r_{j+1-\ell}}}((\lambda_{s_{j+1-\ell}},\ldots,\lambda_{s_{j+1-\ell}}))\right) = HK\left( I_{P_{r_{j+1-\ell}, r_{j+1-\ell}}} \left(\lambda_{s_{j+1-\ell}},\ldots,\lambda_{s_{j+1-\ell}}\right)\right),
	\]
	by Lemma \ref{case_deg_princ_ser} and the uniqueness of highest weight vectors in the principal series representation.
	This shows that $f_k$ belongs to 
	\[
	HK\left( \bigotimes_{\ell=1}^j I_{P_{r_{j+1-\ell}, r_{j+1-\ell}}} (\lambda_{s_{j+1-\ell}},\ldots,\lambda_{s_{j+1-\ell}})\right)
	\]
	for any $k\in K_{n,\infty}$.
	We denote by $\SL_\lambda(\bbR)$ the derived subgroup of $\GL_\lambda(\bbR)$.
	By the definition of the degenerate principal series representation, $f_k$ is left $\SL_\lambda(\bbR)$-invariant.
	Hence $f$ is left $\SL_\lambda(\bbR)$-invariant.
	Since the induced representation $I_{P_{\lambda,n}}(\lambda)$ is equal to
	\[
	\{f \in I_{B_n}(\lambda) \mid \text{$f$ is left $\SL_\lambda(\bbR)$-invariant}\},
	\]
	the function $f$ lies in $I_{P_{\lambda,n}}(\lambda)$.
	This completes the proof.
	\end{proof}
	
	More precisely, we have the following:
	
	\begin{thm}\label{emb_main}
	Take a $\frakk_n$-dominant integral weight $\lambda =(\lambda_1,\ldots,\lambda_n) \in \bbZ^n$.
	Put 
	\[
	I_{P_{i,n}}(\lambda) = \Ind_{P_{i,n}(\bbR)}^{G_n(\bbR)} \left(\mathrm{sgn}^{\lambda_n} |\cdot|^{\lambda_n-n+(i-1)/2}\boxtimes L(\lambda_1,\ldots,\lambda_{n-i})\right).
	\] 
	Then $I_{P_{i,n}}(\lambda)$ contains a highest weight vector of weight $\lambda$ if $\lambda_n = \cdots = \lambda_{n-i+1}$.
	Moreover such a highest weight vector generates $L(\lambda)$.
	Conversely, if the induced representation
	\[
	\Ind_{P_{i,n}(\bbR)}^{G_n(\bbR)} \left(\mu \boxtimes L(\omega_1,\ldots,\omega_{n-i})\right)
	\]
	has a highest weight vector of weight $\lambda$, we have
	\begin{itemize}
	 \item $\lambda_n = \cdots = \lambda_{n-i+1}$.
	 \item $\mu = \mathrm{sgn}^{\lambda_n} |\cdot|^{\lambda_n-n+(i-1)/2}$.
	 \item $(\omega_1,\ldots,\omega_{n-i}) = (\lambda_1,\ldots,\lambda_{n-i})$.
	 \end{itemize}
	\end{thm}
	\begin{proof}
	Take a vector $f$ in $HK(I_{B_n}(\lambda))$.
	By Lemma \ref{emb_lem}, $f$ lies in $HK(I_{P_{\lambda,n}}(\lambda))$.
	By Theorem \ref{Yamashita}, the induced representation $I_{P_{\lambda,n}}(\lambda)$ can be embedded into
	\[
	\Ind_{P_{i,n}(\bbR)}^{G_n(\bbR)} \left(\sgn^{\lambda_n}|\cdot|^{\lambda_n-n+(i-1)/2} \boxtimes \left(I_{B_{n-i}}(\lambda_1,\ldots,\lambda_{n-i})\right)\right).
	\]
	Then $f(g)$ is an element of the representation space of $\det^{\lambda_n} \otimes I_{B_{n-i}}((\lambda_1,\ldots,\lambda_{n-i}))$ for any $g \in G_n(\bbR)$.
	As a $\frakg_{n-i}$-representation, $\det^{\lambda_n} \otimes I_{B_{n-i}}((\lambda_1,\ldots,\lambda_{n-i}))$ is equal to $I_{B_{n-i}}((\lambda_1,\ldots,\lambda_{n-i}))$.
	For $k \in K_{n,\infty}$ and $X \in \frakp_{n-i,-}$, we have
	\[
	X \cdot f(k) = (\mathrm{Ad}(k)X \cdot f)(k) = 0
	\]
	by $\mathrm{Ad}(K_{n,\infty}(\frakp_{n-i,-})) \subset \frakp_{n,-}$ and the assumption for $f$.
	Hence $f(k)$ lies in $HK(I_{B_{n-i}}((\lambda_1,\ldots,\lambda_{n-i})))$ for any $k \in K_{n,\infty}$.
	By the Iwasawa decomposition, for any $g \in G_n(\bbR)$, an element $f(g)$ lies in the subspace $L((\lambda_1,\ldots,\lambda_{n-i})) \subset I_{B_{n-i}}(\lambda_1,\ldots,\lambda_{n-i})$.
	Hence we have
	\[
	f \in \Ind_{P_{i,n}(\bbR)}^{G_n(\bbR)} \left(\sgn^{\lambda_n}|\cdot|^{\lambda_n - n +(i-1)/2} \boxtimes L((\lambda_1,\ldots,\lambda_{n-i}))\right).
	\]
	
	Next we assume the induced representation
	\[
	\Ind_{P_{i,n}(\bbR)}^{G_n(\bbR)} \left(\mu \boxtimes L(\omega_1,\ldots,\omega_{n-i})\right)
	\]
	has a highest weight vector of weight $\lambda$.
	We denote $\mu_1$ the character of $\bbR^\times$ such that $\mu_1(\det) = \mu$.
	Let $\mu' = \mu_1|\cdot|^{-n+(i-1)/2} \boxtimes \cdots \boxtimes \mu_1|\cdot|^{-n+(i-1)/2} \boxtimes \mathrm{sgn}^{\omega_{n-i}} \boxtimes \cdots \boxtimes \mathrm{sgn}^{\omega_1}$ be a character of $T_n(\bbR)$.
	Since $L(\omega_1,\ldots,\omega_{n-i})$ can be embedded into $I_{B_{n-i}}(\omega_{n-i},\ldots,\omega_{1})$, the induced representation can be embedded into the principal series representation
	\[
	\Ind_{B_n(\bbR)}^{G_n(\bbR)} \left(\delta_{B_n}^{-1/2} \mu' \otimes \left(\mathbf{1} \boxtimes \cdots \boxtimes \mathbf{1} \boxtimes |\cdot|^{\omega_{n-i}} \boxtimes \cdots \boxtimes |\cdot|^{\omega_1}\right)\right).
	\]
	Here $\mathbf{1}$ is the trivial character.
	By the assumption, this principal series representation has a non-zero highest weight vector of weight $\lambda$.
	By Theorem \ref{Yamashita}, we have 
	\begin{itemize}
	\item $\lambda_{n} = \cdots = \lambda_{n-i+1}$.
	\item $\mu = \mathrm{sgn}^{\lambda_n} |\cdot|^{\lambda_n-n+(i-1)/2}$.
	\item $(\omega_1,\ldots,\omega_{n-i}) = (\lambda_1,\ldots,\lambda_{n-i})$.
	\end{itemize}
	This completes the proof.
	\end{proof}

	\begin{rem}
	This theorem and the previous lemma can be proved easily if the following statement is true:
	Let $P$ be a standard parabolic subgroup of a reductive Lie group $G$.
	For a representation $\pi$ of $M$ with finite length, we have
	\begin{align}\label{soc_quest}
	\mathrm{Soc}\left(\Ind_{P}^G(\pi)\right) \subset \Ind_{P}^G\left(\mathrm{Soc}(\pi)\right).
	\end{align}
	Here $\mathrm{Soc}$ mean the sum of irreducible subrepresentations.
	Indeed, if (\ref{soc_quest}) is true, Lemma \ref{emb_lem} follows from the double induction formula and \cite[Theorem 3.4.2 (ii)]{1999_Howe-Lee}.
	But the auther has no counterexamples and no idea how to show the above argument (\ref{soc_quest}).
	\end{rem}

\section{Decomposition of the space of automorphic forms and Main theorem}
	In this section, we review the general theory of automorphic forms and state the main theorem.
	
	\subsection{Definition}
	
	Let $G$ be a connected reductive group over $\bbQ$.
	We denote by $\bbA_\bbQ$ and $\bbA_{\bbQ,\fini}$ the adele ring of $\bbQ$ and the finite part of $\bbA_\bbQ$.
%	In this paper, $G$ is sometimes equal to $\Res_{F/\bbQ}\Sp_{2n}$ with a tottaly real field $F$.
	Put $\frakg = \mathrm{Lie}(G(\bbR)) \otimes_\bbR \bbC$.
	Let $\calZ$ be the center of the universal enveloping algebra of $\frakg$.
	For a parabolic subgroup $P$, we denote by $A_P$ and $A_P^\infty$ the split component of $P$ and the identity component of $A_P(\bbR)$.
	Fix a minimal parabolic $\bbQ$-subgroup $P_0$ and a Levi decomposition $P_0=M_0N_0$.
	We assume that a maximal compact subgroup $K = \prod_v K_v$ of $G(\bbA_\bbQ)$ satisfies the following conditions (cf.~\cite[$\S$ I.1.4]{MW}):
	\begin{itemize}
	\item $G(\bbA_\bbQ) = P_{0} (\bbA_\bbQ) K$.
	\item $P(\bbA_\bbQ) \cap K = (M(\bbA_\bbQ) \cap K)(N(\bbA_\bbQ) \cap K)$.
	\item $M(\bbA_\bbQ) \cap K$ is a maximal compact subgroup of $M(\bbA_\bbQ)$. 
	\end{itemize}
	Here $P$ runs through all standard parabolic subgroups of $G$ with Levi decomposition $P=MN$ and $M$ is a standard Levi subgroup.

	\begin{dfn}(\cite[Definition I.2.17]{MW}).
	Let $P=MN$ be a standard parabolic subgroup of $G$.
	For a smooth function $\phi: N(\bbA_\bbQ)M(\bbQ) \backslash G(\bbA_\bbQ) \longrightarrow \bbC$, we say that $\phi$ is automorphic if it satisfies the following conditions:
	\begin{itemize}
	\item $\phi$ is right $K$-finite.
	\item $\phi$ is $\mathcal{Z}$-finite.
	\item $\phi$ is slowly increasing.
	\end{itemize}
	We denote by $\mathcal{A}(P \backslash G)$ the space of automorphic forms on $ N(\bbA_\bbQ)M(\bbQ) \bs G(\bbA_\bbQ)$.
	For simplicity, we write $\calA(G)$ when $P=G$.
	Let $\calA(A_G^\infty \bs G)$ be the space of automorphic forms on $A_G^\infty \bs G(\bbA_\bbQ)$.
	The space $\mathcal{A}(P \backslash G)$ is a $G(\bbA_\fini) \times \gk$-module by the right translation.
%	For an automorphic form $\varphi$, we say $\varphi$ is nearly holomorphic (resp.~holomorphic) if $\varphi$ is $\frakp_-$-finite (resp.~$\frakp_- \cdot \varphi =0$).
%	Let $\NHAFonPG$ be the subspace of $\mathcal{A}(P \backslash G)$ consisting of all nearly holomorphic automorphic forms.
%	We call a function $\phi \in \NHAFonPG$ a nearly holomorphic automorphic form.
%	Note that the space $\NHAFonPG$ is a $G(\bbA_\fini) \times \gk$-module by the right translation.
%	We define $\calA(AP \bs G)$ and $\calN(AP \bs G)$ similarly.
	\end{dfn}
	
	Fix a totally real field $F$.
	We denote by $\bbA_F$ and $\bbA_{F,\fini}$ the adele ring of $F$ and the finite part of $\bbA_F$, respectively.
	Set $G_n = \Res_{F/\bbQ} \Sp_{2n}$.
	For a non-archimedean place $v$ of $F$, put
	\[
	K_{n,v} = \Sp_{2n}(\calO_{F_v})
	\]
	where $\calO_{F_v}$ is the ring of integers of $F_v$.
	For an archimedean place $v$, put
	\[
	K_v = K_{n,\infty}.
	\]
	The maximal compact subgroup $K_n = \prod_{v}K_{n,v}$ of $G_n(\bbA_\bbQ)$ satisfies the conditions as above.
	The Lie algebra $\frakg_n=\mathrm{Lie}(G_n(\bbR))\otimes_\bbR \bbC$ is equal to $\bigoplus_{v \in \bfa} \mathfrak{sp}_{2n}(\bbC)$ and contains the abelian subalgebras $\frakp_{n,\pm}$.
	We denote by $\calZ_n$ the center of the universal enveloping algebra of $\frakg_n$.
%	When no confusion can arise, we write $\prod_{v \in \bfa}K_{n,\infty}$, $\bigoplus_{v \in \bfa} \frakg_n,\bigoplus_{v \in \bfa} \frakp_{n,-}$ and $\bigotimes_{v \in \bfa}\calZ_n$ simply $K_{n,\infty}$, $\frakg_n,\frakp_{n,\pm}$ and $\calZ_n$, respectively.
	Set
	\[
	\calN(G_n) = \left\{ \varphi \in \calA(G_n) \,\middle|\, \text{$\varphi$ is right $\frakp_{n,-}$-finite.}\right\}.
	\]
	We say that an automorphic form in $\calN(G_n)$ is nearly holomorphic.
	For an infinitesimal character $\chi$ of $\calZ_n$, set
	\[
	\calN(G_n,\chi) = \{\varphi \in \calN(G_n) \mid \text{there exist an integer $m > 0$ such that $(\chi(z)-z)^m \varphi = 0$ for any $z \in \calZ_n$}\}.
	\]
	The purpose in this paper is to decompose $\calN(G_n,\chi)$ as a $G_n(\bbA_{\bbQ,\fini}) \times (\frakg_n, K_{n,\infty})$-module under certain assumptions for $\chi$.
	
	We then obtain the following result:
	
	\begin{prop}\label{ss_Z}
	For any $\varphi \in \calN(G_n,\chi_\lambda)$, $\varphi$ is a $\chi$-eigenfunction, i.e., $z \cdot \varphi = \chi(z) \varphi$ for any $z \in \calZ_n$.
	\end{prop}
	
	This will be proved in \S \ref{pf_ss_Z}.
	
	\begin{rem}
	This result is proved in \cite{Horinaga} and \cite{pss2} for $n=1,2$, respectively.
	\end{rem}

	\subsection{Decomposition according to parabolic subgroups}
	
	For parabolic $\bbQ$-subgroups $P$ and $Q$ of $G$, we say that $P$ and $Q$ are associate if the split components $A_P$ and $A_Q$ are $G(\bbQ)$-conjugate. 
	We denote by $\{P\}$ the associated class of the parabolic subgroup $P$.
	For a locally integrable function $\varphi$ on $N_P(\bbQ) \bs G(\bbA_\bbQ)$, set
	\[
	\varphi_{P}(g) = \int_{N_P(\bbQ) \bs N_P(\bbA_\bbQ)} \varphi(ng) \, dn
	\]
	where $P = M_PN_P$ is the Levi decomposition of $P$ and the Haar measure $dn$ is normalized by
	\[
	\int_{N_P(\bbQ) \bs N_P(\bbA_\bbQ)} \, dn = 1.
	\]
	The function $\varphi_P$ is called the constant term of $\varphi$ along $P$.
	If $\varphi$ lies in $\calA(P \bs G)$, $\varphi_Q$ is an automorphic form on $N_Q(\bbA_\bbQ)M_Q(\bbQ) \bs G(\bbA_\bbQ)$ for a parabolic subgroup $Q \subset P$.
	We call $\varphi$ cuspidal if $\varphi_Q$ is zero for any standard parabolic subgroup $Q$ of $G$ with $Q \subsetneq P$.
	Note that $\varphi$ is cuspidal if and only if $\varphi_P$ is zero for any maximal standard parabolic subgroups of $P$.
	We denote by $\calA_{\cusp}(P \bs G)$ the space of cusp forms in $\calA(P \bs G)$.
	Set
	\[
	M(\bbA_\bbQ)^1 = \bigcap_{\chi \in \Hom_{\mathrm{conti}}(M(\bbA_\bbQ),\bbC^\times)} \mathrm{Ker}(|\chi|).
	\]
	For $g \in G(\bbA_\bbQ)$, let $\varphi_{P,g}$ be the function on $M_P(\bbA_\bbQ)^1$ defined by $m \longmapsto \varphi(mg)$.
	Put
	\[
	\calA(G)_{\{P\}} = \left\{\varphi \in \calA(G) \,\middle|\, \begin{matrix}\text{$\varphi_{Q,ak}$ is orthogonal to all cusp forms on $M_Q(\bbA_Q)^1$}\\ \text{for any $a \in A_Q, k\in K_n$, and $Q \not \in \{P\}$}\end{matrix}\right\}.
	\] 
	We define the space $\calA(A_G^\infty \bs G)_{\{P\}}$ similarly.
	By the definition, $\calA(A_G^\infty \bs G)_{\{G\}}$ is equal to $\calA_\cusp(A_G^\infty \bs G)$.
	Then Langlands had proven the following result in \cite{langlands}:
	\begin{thm}
	With the above notation, we have
	\[
	\calA(A_G^\infty \bs G) = \bigoplus_{\{P\}}\calA(A_G^\infty \bs G)_{\{P\}},
	\]
	where $\{P\}$ runs through all associated classes of parabolic subgroups.
	\end{thm}
	
	Next we state the result corresponding to the above theorem.
	Put
	\[
	\calN(G_n)_{\{P\}} = \calN(G_n) \cap \calA(G_n)_{\{P\}}.
	\]
	Note that $A_{G_n}^\infty = \{1\}$.
	For $1 \leq j \leq n$, we denote by $P_{j,n}$ and $Q_{j,n}$ the standard parabolic subgroups of $G_n$ with the standard Levi subgroups $\Res_{F/\bbQ}\GL_j \times G_{n-j}$ and $(\Res_{F / \bbQ} \GL_1)^j \times G_{n-j}$, respectively.
	We put $P_{0,n} = Q_{0,n} = G_n$.
	Note that $P_{1,n} = Q_{1,n}$.
	
	\begin{prop}\label{decomp_NHAF_parab}
	With the above notation, we have
	\[
	\calN(G) = \bigoplus_{i=0}^n \calN(G_n)_{\{Q_{i,n}\}}.
	\]
	\end{prop}
	
	This will be proved in \S \ref{pf_decomp_NHAF_parab}.
	In the following subsection, we state the refinement of the above decomposition.
	
	\subsection{Decomposition according to cuspidal components}
	
%	We regard $\fraka_M^* = \Hom(A_M^\infty,\bbC^\times) \cong \Hom(M(\bbA)/M^1(\bbA),\bbC^\times)$ as the set of characters of $M(\bbA)$.
%	Let $(\fraka_M^G)^* = \Hom(A_M^\infty / A_G^\infty,\bbC^\times)$.
%	We regard $\lambda \in (\fraka_M^G)^*$ as a character of $M(\bbA)$.
%	The modulus character $\rho_P$ of $P(\bbA)$ can be regarded as an element of $(\fraka_M^G)^*$.

%	Let $\mathcal{A}(A_G^\infty\backslash G(\bbA))$ be the space of automorphic forms on $G(\bbQ)A_G^\infty\backslash G(\bbA)$.
%	The subspace of cusp forms in $\mathcal{A}(A_G^\infty\backslash G(\bbA))$ is denoted by $\mathcal{A}_{\mathrm{cusp}}(A_G^\infty\backslash G(\bbA))$.

%	For a standard Levi subgroup $M$, we set
%\[
%	W(M) = \left\{ 
%	w \in W \,\middle| \,
%	\begin{array}{ll}
%	\text{$\bullet$ $wMw^{-1}$ is a standard Levi subgroup of $G$} 
%	\\
%	\text{$\bullet$ $w$ has a minimal length in $wW_M$}
%	\end{array}
%	\right\}.
%	\]
%	Recall that two standard parabolic subgroups $P=MN$ and $P'=M'N'$ are associated if there exists $w\in W(M)$ such that $M'=wMw^{-1}$.
%	If $P'$ and $P$ are associated we write $P\sim P'$.
%	Take $w\in W(M)$.
%	Put $M'=wMw^{-1}$ and let $P'=M'N'$ be the standard parabolic subgroup with Levi subgroup $M'$.
	
%	Let $W_G$ be the Weyl group of $G$.
	For a reductive group $H$, let $W_H$ be the Weyl group of $H$.
	Let $P$ be a standard parabolic subgroup of $G$ and $\tau$ an irreducible cuspidal automorphic representation of $M(\bbA_\bbQ)$.
	We assume that the central character $\chi_\tau$ of $\tau$ is trivial on $A_G^\infty$.
	We say that a cuspidal datum is a pair $(M, \tau)$ such that $M$ is a Levi subgroup of $G$ and that $\tau$ is an irreducible cuspidal automorphic representation of $M(\bbA_F)$.
%	Let $W_{G}$ be the Weyl group of $G$.
	Take $w \in W_G$.
	Put $M^w=wMw^{-1}$ and let $P^w=M^wN^w$ be the standard parabolic subgroup with Levi subgroup $M^w$.
	The irreducible cuspidal automorphic representation $\tau^w$ of $M^w(\bbA_\bbQ)$ is defined by $\tau^w(m')=\tau(w^{-1}m'w)$ for $m'\in M^w(\bbA_\bbQ)$.
	Two cuspidal data $(M, \tau)$ and $(M', \tau')$ are called equivalent if there exists $w\in W(M)$ such that $M'=wMw^{-1}$ and that $\tau'=\tau^w$.
	Here we put
	\[
	W(M) = \left\{ w \in W \,\middle|\, 
	\begin{matrix}
	\text{$w M w^{-1}$ is a standard Levi subgroup of $G$.}\\
	\text{$w$ has a minimal length in $wW_M$.}
	\end{matrix}
	\right\}.
	\]

	Let $\mathcal{A}(A_G^\infty \bs G)_{(M, \tau)}$ is the subspace of automorphic forms in $\mathcal{A}(A_G^\infty \bs G)$ with the cuspidal support $(M, \tau)$.
	For the definition, see \cite[\S III.2.6]{MW}.
	Then the following result is well-known.
	For details, see \cite[Theorem III.2.6]{MW}.
	
	\begin{thm}
	The space $\mathcal{A}(A_G^\infty \bs G)$ is decomposed as
	\[
	\mathcal{A}(A_G^\infty \bs G(\bbA_\bbQ))=\bigoplus_{(M, \tau)} \mathcal{A}(A_G^\infty \bs G)_{(M, \tau)}.
	\]
	Here, $(M, \tau)$ runs through all equivalence classes of cuspidal data.
	\end{thm}
	
	Put $\calN(G_n)_{(M, \tau)}=\calN(G_n)\cap \mathcal{A}(G_n)_{(M, \tau)}$.
	We say that an irreducible cuspidal automorphic representation $\pi = \bigotimes_v \pi_v$ is holomorphic if $\pi_v$ is an irreducible unitary highest weight module for any $v \in \bfa$.
	
	\begin{prop}\label{decomp_NHAF_cusp_supp}
	Let $P$ be a standard parabolic subgroup of $G_n$ with the standard Levi subgroup $M$.
	\begin{enumerate}
	\item With the above notation, the space $\calN(G_n)_{(M,\pi)}$ is non-zero only if $P$ is associated to $Q_{i,n}$ for some $i$.
	\item Suppose an irreducible cuspidal automorphic representation $\Pi$ of $M_{Q_{i,n}} = (\Res_{F/\bbQ}\GL_1)^i \times G_{n-i}$ is isomorphic to $\mu_1 \boxtimes \cdots \boxtimes\mu_i \boxtimes \pi$.
	If $\calN(G_n)_{(Q_{i,n},\Pi)} \neq 0$, we have
	\begin{itemize}
	\item $\mu_1 = \cdots = \mu_{i}$.
	\item $\pi$ is a holomorphic cuspidal automorphic representation of $G_{n-i}(\bbA_\bbQ)$.
	\end{itemize}
	\end{enumerate}
	\end{prop}
	
	This will be proved in \S \ref{pf_decomp_NHAF_cusp_supp}.

	\subsection{Main theorem}
	
	For a weight $\lambda$, let $\chi_\lambda$ be the infinitesimal character with the Harish-Chandra parameter $\rho + \lambda$.
	Here $\rho = (-1,\ldots,-n)$ is half the sum of positive roots. 
	Then, for weights $\lambda$ and $\mu$, we have $\chi_\lambda = \chi_\mu$ if and only if there exists $w \in W_{G_n}$ such that $w \cdot \lambda = \mu$, where $w \cdot \lambda$ is the dot action.
	Infinitesimal characters of $\calZ_n = \calZ(\calU(\frakg_n))$ are parametrized by the set
	\[
	- \rho + \Lambda^+_\bbC = (1,\ldots,n) + \left\{(\lambda_{1},\ldots,\lambda_{n}) \in \bbC^n \,\middle|\, \text{$\lambda_{i} - \lambda_{i+1} \in \bbZ_{\geq 0}$ for any $1 \leq i \leq n-1$} \right\}.
	\]
	We say that an infinitesimal character $\chi_\lambda$ is integral if the weight $\lambda$ is integral.
	We also say that $\chi_\lambda$ is regular (resp. singular) if $\#(W_{G_n} \cdot \lambda) = \# W_{G_n}$ (resp. $\#(W_{G_n} \cdot \lambda) < \# W_{G_n}$).
	Here $W_{G_n} \cdot \lambda$ is the orbit under the dot action.
	Note that the integral infinitesimal characters are parametrized by the set $- \rho + \Lambda_{\bbZ}^+$ and the regular integral infinitesimal characters are parametrized by the set
	\[
	\{(\lambda_1,\ldots,\lambda_n) \in \bbZ^n \mid \lambda_1 \geq \cdots \geq \lambda_n \geq n\},
	\]
	where $\Lambda_\bbZ^+ = \Lambda_\bbC^+ \cap \bbZ^n$.
	We first show the integrality of infinitesimal characters:
	
	\begin{lem}\label{integrality_wt}
	Suppose $\calN(G_n,\chi)$ is non-zero.
	Then $\chi$ is integral.
	\end{lem}
	This will be proved in \S \ref{pf_main}.
	
	By Lemma \ref{integrality_wt}, we may assume $\chi$ is integral.
%	Put $\Lambda_\bbZ^+ =\{(\lambda_1,\ldots,\lambda_n) \in \bbZ^n \mid \lambda_1 \geq \cdots \geq\lambda_n \geq 0 \}$.
	Take $\lambda = (\lambda_v)_v = (\lambda_{1,v},\ldots,\lambda_{n,v})_v \in \bigoplus_{v \in \bfa} \{-\rho + \Lambda_{\bbZ}^+\}$.
	We say that a regular infinitesimal character $\chi_\lambda$ is sufficiently regular relative to $i$ if $\lambda_{n,v} > 2n- i +1$ for any $v$.
	Put
	\[
	\calN(G_n,\chi)_{\{P\}} = \calN(G_n,\chi) \cap \calN(G_n)_{\{P\}}
	\]
	for a parabolic subgroup $P$ of $G$.
	Note that this space is non-zero only if $P$ is associated to $Q_{i,n}$ for some $i$ by Proposition \ref{decomp_NHAF_parab}.
	Set
	\[
	\mathfrak{X}_{1} = \{\mu = \otimes_v \mu_v \in \mathrm{Hom}_{\mathrm{conti}}(F^\times F_{\infty,+}^\times \bs \bbA_{F}^\times,\bbC^\times) \mid \text{$\mu_v = \mathbf{1}_v$ for any $v \in \bfa$}\}
	\]
	and
	\[
	\mathfrak{X}_{-1} = \{\mu = \otimes_v \mu_v \in \mathrm{Hom}_{\mathrm{conti}}(F^\times F_{\infty,+}^\times \bs \bbA_{F}^\times,\bbC^\times) \mid \text{$\mu_v = \mathrm{sgn}_v$ for any $v \in \bfa$}\}.
	\]
	Here $F_{\infty,+}^\times$ is the identity component of $F_{\infty}^\times$.
	For a parabolic subgroup $P$ of $G_n$ and an irreducible subrepresentation $\pi$ of $L^2(M_P(\bbQ) \bs M_P(\bbA_\bbQ))$, let $L^2(M_P)_\pi$ be the $\pi$-isotypic component of it.
	We denote by $\delta_P$ the modulas character of $P$.
	We denote by $\Ind_{P(\bbA_\bbQ)}^{G(\bbA_\bbQ)}(\pi)$ the space of functions $\varphi$ on $G_n(\bbA_\bbQ)$ such that
	\begin{itemize}
	\item $\varphi$ lies in $\calA(P \bs G_n)$.
	\item For any $k \in K_n$, the function $m \longmapsto \delta^{-1/2}_P(m)\varphi(mk)$ on $M_P(\bbA_\bbQ)$ lies in $L^2(M_P)_\pi$.
	\end{itemize}
	We then state the main theorem:
	
	\begin{thm}\label{main}
	Fix a positive integer $i \leq n$ and a weight $\lambda$.
	\begin{enumerate}
	\item If $\calN(G_n,\chi_\lambda)_{\{Q_{i,n}\}}$ is non-zero, there exists a $\frakk_n$-dominant weight $\omega = (\omega_{1,v},\ldots,\omega_{n,v})_v$ such that
	\begin{itemize}
	\item $\chi_\omega = \chi_\lambda$.
	\item $\omega_{n,v} = \cdots = \omega_{n-i+1,v}$ for any $v \in \bfa$.
	\item $\omega_{n,v}$ is independent of $v \in \bfa$.
	\end{itemize}
	\item Suppose $\chi_\lambda$ is sufficiently regular relative to $i$ and a weight $\lambda = (\lambda_v)_v = (\lambda_{1,v},\ldots,\lambda_{n,v})_v \in \bigoplus_{v \in \bfa}\Lambda_\bbZ^+$ satisfies the following two conditions:
	\begin{itemize}
	\item $\lambda_{n,v} = \cdots = \lambda_{n-i+1,v}$ for any $v \in \bfa$.
	\item $\lambda_{n,v}$ is independent of $v \in \bfa$.
	\end{itemize}
	Let $\pi$ be an irreducible holomorphic cuspidal automorphic representation of $G_{n-i}(\bbA_\bbQ)$ with $\pi_\infty = \boxtimes_{v \in \bfa} L(\lambda_{1,v}, \ldots, \lambda_{n-i,v})$.
	We then have the isomorphism
	\[
	\calN(G_n,\chi_\lambda)_{(M_{Q_{i,n}}, \mu^{\boxtimes i} \boxtimes \pi)} 
	\cong 
	\left(\Ind_{P_{i,n}(\bbA_{F,\fini})}^{G(\bbA_{F,{\fini}})}(\mu |\cdot|^{\lambda_{n,v}-n+(i-1)/2} \boxtimes \pi)\right) 
	\boxtimes \left(\boxtimes_{v \in \bfa} L(\lambda_v)\right)
	\]
	if $\mu \in \mathfrak{X}_{(-1)^{\lambda_{n,v}}}$.
	If $\mu \not\in \mathfrak{X}_{(-1)^{\lambda_{n,v}}}$, we have $\calN(G_n,\chi_\lambda)_{(M_{Q_{i,n}}, \mu^{\boxtimes i} \boxtimes \pi)} = 0$.
	\end{enumerate}
	\end{thm}
	
	This will be proved in \S \ref{pf_main}.
	
	\begin{rem}
	If an infinitesimal character $\chi$ is not sufficiently regular relative to $i$, then $\calN(G_n,\chi)_{\{Q_{i,n}\}}$ may be very complicated.
	In the case of $G = \SL_2$-case, see \cite{Horinaga}.
	\end{rem}
	
	For a regular infinitesimal character $\chi_\lambda$ with the parameter $\lambda = (\lambda_v)_v = (\lambda_{1,v},\ldots,\lambda_{n,v})_v \in \bigoplus_{v \in \bfa} \{-\rho + \lambda_{\bbZ}^+\}$, we say that $\chi_\lambda$ is sufficiently regular if we have $\lambda_{n,v} > 2n$ for any $v$.
	This condition is equivalent to that $\lambda$ is sufficiently regular for any $i$ with $1 \leq i \leq n$.
	
	\begin{cor}\label{suff_reg_case}
	Let $\lambda = (\lambda_v)_v = (\lambda_{1,v},\ldots,\lambda_{n,v})_v$ be a $\frakk_n$-dominant integral weight such that $\lambda_{n,v}$ is independent of $v \in \bfa$ and $\lambda_{n-1,v} = \lambda_{n,v}$ for any $v \in \bfa$.
	For any positive integer $i \leq n$ and a sufficiently regular infinitesimal character $\chi_\lambda$ relative to $n$, we have the direct sum decomposition
	\[
	\calN(G_n,\chi_\lambda)_{\{Q_{i,n}\}} \cong \bigoplus_{\mu,\pi} \left(\Ind_{P_{i,n}(\bbA_{F,\fini})}^{G(\bbA_{F,{\fini}})}(\mu |\cdot|^{\lambda_{n,v}-n+(i-1)/2} \boxtimes \pi)\right) 
	\boxtimes \left(\boxtimes_{v \in \bfa} L(\lambda_v)\right),
	\]
	where $\mu$ runs through all Hecke characters in $\frakX_{(-1)^{\lambda_{n,v}}}$ and $\pi$ runs through all irreducible holomorphic cuspidal representation of $G_{n-i}(\bbA_F)$ such that $\pi_\infty \cong \boxtimes_{v \in \bfa} L(\lambda_{1,v},\ldots,\lambda_{n-i,v})$.
	
	\end{cor}
	
	This will be proved in \ref{proof_cor_main}.
	
\section{Whittaker Fourier coefficients of nearly holomorphic automorphic forms}
	In this section, we compute the constant terms of nearly holomorphic automorphic forms.
%	The computation as in this section is based on the computations as in \S \ref{FC_NHMF}.
	
	\subsection{Definition}
	Let $F$ be a totally real field with degree $d$.
	We fix a non-trivial additive character $\psi = \otimes_v \psi_v$ of $F \bs \bbA_F$ as follows:
	If $F=\bbQ$, let
	\begin{align*}
	\psi_p(x) &= \exp(-2\pi\sqrt{-1}\, y), \qquad x \in \bbQ_p,  \\
	\psi_\infty(x) &= \exp(2\pi\sqrt{-1} \, x),\qquad x \in \bbR, 
	\end{align*}
	where $y \in \cup_{m=1}^\infty p^{-m}\bbZ$ such that $x-y \in \bbZ_p$.
	In general, for an archimedean place $v$ of $F$, put $\psi_v = \psi_\infty$ and for a non-archimedean place $v$ with the rational prime $p$ divisible by $v$, put $\psi_v(x) = \psi_p(\mathrm{Tr}_{F_v/\bbQ_p}(x))$.
	Then for any continuous character $\Psi$ of $N_{P_{n,n}}(\bbQ) \bs N_{P_{n,n}}(\bbA_\bbQ)$, there exists a unique symmetric matrix $h \in \Sym_n(F)$ such that
	\[
	\Psi(n) = \psi(\mathrm{tr}(hn))
	\]
	for any $n \in N_{P_{n,n}}(\bbA_\bbQ)$.
	Let $\psi_h$ be the additive character of $N_{P_{n,n}}(\bbQ) \bs N_{P_{n,n}}(\bbA_\bbQ)$ defined by $n \longmapsto \psi(\mathrm{tr}(hn))$.
	
	Take an automorphic form $\varphi \in \calN(G_n)$.
	For any $h \in \Sym_n(F)$, put
	\[
	\Wh_{\varphi,h}(g) = \int_{N_{P_{n,n}}(F) \bs N_{P_{n,n}}(\bbA_F)} \varphi(ng) \overline{\psi}_h(n) \, dn.
	\]
	Here $\overline{\psi}_h$ is the complex conjugate of $\psi_h$.
	We then have the Whittaker-Fourier expansion
	\[
	\varphi(ng) = \sum_{h \in \Sym_{n}(F)} \Wh_{\varphi,h}(g)\psi_h(n), \qquad n \in N_{P_{n,n}}(\bbA_\bbQ), \, g \in G_n(\bbA_\bbQ).
	\]
	In the following subsections, we study the Whittaker-Fourier coefficients $\Wh_{\varphi,h}$.
	
	\subsection{The relation between automorphic forms on $G_n(\bbA_\bbQ)$ and modular forms on $\frakH_n^d$}
	
	We denote the complexification of $K_{n,\infty}$ by $K_{n,\bbC} \cong \prod_{v \in \bfa} \GL_n(\bbC)$.
	Let $j$ be a $K_{n,\bbC}$-valued function on $\prod_{v \in \bfa}\frakH_n \times \prod_{v \in\bfa} \Sp_{2n}(F_v)$ defined by
	\[
	j((g_v)_v, (z_v)_v) = (c_v z_v + d_v)_{v}, \qquad (g_v)_v = \left(\begin{pmatrix}a_v & b_v \\ c_v & d_v\end{pmatrix}\right)_v \in \prod_{v \in\bfa} \Sp_{2n}(\bbR), \qquad (z_v)_v \in \prod_{v \in \bfa}\frakH_n.
	\]
	For a finite-dimensional representation $\rho$ of $K_{n,\bbC}$, put $j_\rho = \rho \circ j$.
	For a congruence subgroup $\Gamma$ of $G_n(\bbQ)$, let $N_\rho(\Gamma)$ be the space of nearly holomorphic modular forms of weight $\rho$ with respect to $\Gamma$ and $M_\rho(\Gamma)$ the subspace of $N_\rho(\Gamma)$ consisting of holomorphic functions.
	We embed $\Gamma$ into $G_n(\bbA_{\bbQ,\fini})$ diagonally.
	Then we obtain the open compact subgroup $K_\Gamma$ of $G_n(\bbA_{\bbQ,\fini})$ which is the closure of $\Gamma$.
		
	We denote by $(\rho,V)$ a finite-dimensional representation of $K_{n,\infty}$.
	To simplify the notation, we write $(\rho,V)$ the holomorphic representation of $K_{n,\bbC}$ corresponding to the representation $(\rho,V)$ of $K_{n,\infty}$.
	Let $K_{\text{fin}}$ be an open compact subgroup of $G_n(\bbA_{\bbQ,\text{fin}})$.
	Let $\Gamma$ be the projection of $G_n(\bbQ) \cap K_\fini G_n(\bbR)$ to $G_n(\bbR)$.
	By the strong approximation, we have $G_n(\bbA_\bbQ) = G_n(\bbQ)G_n(\bbR)K_\fini$.
	We now obtain the isomorphism
	\[
	G_n(\bbQ) \bs G_n(\bbA_\bbQ) /  K_{n,\infty} K_\fini \cong \Gamma \bs \frakH_n^d
	\]
	by the map 
	\[
	\gamma  g_\infty k \longmapsto g_\infty(\mathbf{i})
	\]
	for any $\gamma \in G_n(\bbQ), g_\infty \in G_n(\bbR)$, and $k \in K_\fini$.
	Here $\mathbf{i} = (\sqrt{-1}\, \mathbf{1}_n,\ldots,\sqrt{-1}\, \mathbf{1}_n) \in \frakH_n^d$.
	For a nearly holomorphic modular form $f \in N_\rho (\Gamma)$, we define a function $\varphi_{f}$ on $G_n(\bbA_\bbQ)$ by
	\begin{align}\label{lift_to_group}
	\varphi_{f}(g) = (f|_\rho g_\infty) (\mathbf{i}), \qquad g = \gamma g_\infty k, \,\, \gamma \in G_n(\bbQ), \,\, g_\infty\in G_{n}(\bbR), \, \, k\in K_\fini.
	\end{align}
	Note that the $V$-valued function $\varphi_{f}$ is left $G_n(\bbQ)$-invariant and right $K_{n}$-finite.
	Moreover we have
	\begin{align}\label{act_on_F_f}
	\varphi_{f}(gk_\infty) = \rho(k_\infty)^{-1} \varphi_{f}(g)
	\end{align}
	for $g \in G_n(\bbA_\bbQ)$ and $k_\infty\in K_{n,\infty}$.
	Let $(\rho^*,V^*)$ be the contragredient representation of $(\rho,V)$.
	For $v^* \in V^*$, we have a scalar valued function $g \longmapsto \langle \varphi_{f}(g), v^* \rangle$ on $G_n(\bbA_\bbQ)$.
	For simplicity of notation, we denote by $\langle \varphi_{f}, v^* \rangle$ the scalar valued function.
	For an automorphic form $\varphi$ and an irreducible finite-dimensional representation $\tau$ of $K_{n,\infty}$, we say that $\varphi$ has the $K_{n,\infty}$-type $\tau$ if the representation $\langle r(k) \varphi \mid k \in K_{n,\infty}\rangle_\bbC$ of $K_{n,\infty}$ is isomorphic to $\tau$ where $r$ is the right translation.
	Then the automorphic form $\langle \varphi_{f}, v^* \rangle$ has the $K_{n,\infty}$-type $\rho^*$ by the formula (\ref{act_on_F_f}).
	By the definition of $\varphi_{f}$, the open compact group $K_\fini$ fixes the automorphic form $\langle \varphi_{f}, v^* \rangle$ under the right translation.
	We denote by $\calA(G_n)^{K_\fini}_\rho$ the space of $K_\fini$-fixed automorphic forms with the $K_{n,\infty}$-type $\rho$.
	We then obtain the map
	\[
	C^\infty(\frakH,\rho)^{\Gamma} \otimes V^* \longrightarrow \calA( G_n)^{K_\fini}_{\rho^*} \colon f \otimes v^* \longmapsto \langle \varphi_{f},v^* \rangle.
	\]
	Here $C^\infty(\frakH_n^d,\rho)^{\Gamma}$ is the space of slowly increasing $V$-valued $C^\infty$-functions $f$ on $\frakH_n^d$ such that $f|_\rho \gamma = f$ for any $\gamma \in \Gamma$.
	
	\begin{lem}
	Let $(\rho,V)$ be an irreducible finite-dimensional representation of $K_{n,\infty}$.
	We regard $\rho$ as the irreducible holomorphic representation of $K_{n,\bbC}$.
	Let $K_\fini$ be an open compact subgroup of $G(\bbA_{\bbQ,\fini})$.
	Let $\Gamma$ be the congruence subgroup of $G_{n}(\bbR)$ corresponding to the open compact subgroup $K_\fini$.
	With the above notation, we obtain the isomorphism
	\[
	C^\infty(\frakH,\rho)^{\Gamma} \otimes V^* \cong \calA (G_n)^{K_\fini}_{\rho^*} \colon f \otimes v^* \longmapsto \langle \varphi_{f},v^* \rangle.
	\]
	\end{lem}
	
	For the proof, see \cite[\S 2.2]{Horinaga}.
	More precisely, we have the following isomorphism:
	\begin{align}\label{corr_MF_AF}
	N_\rho(\Gamma) \otimes V^* \cong \calN(G_n)^{K_\fini}_{\rho^*},
	\end{align}
	where $\calN(G_n)^{K_\fini}_{\rho^*}$ is the subspace of all $\frakp_{n,-}$-finite functions in $\calA(G_n)^{K_\fini}_{\rho^*}$.
	Similarly, we have
	\begin{align}\label{corr_MF_AF_hol}
	M_\rho(\Gamma) \otimes V^* \cong HK \left(\calN(G_n)^{K_\fini}_{\rho^*} \right).
	\end{align}
	For details, see section $5$ and $7$ in \cite{90_shimura}, and \cite[Remark 2.2]{Horinaga}.
	Under this correspondence, for a nearly holomorphic modular form $f \in N_\rho(\Gamma)$,
	we have
	\begin{align}\label{rel_FC_MF_AF}
	\Wh_{\langle\varphi_f,v^*\rangle,h}\left(\begin{pmatrix}\mathbf{1}_n & n_\infty \\ & \mathbf{1}_n\end{pmatrix} \begin{pmatrix}a_\infty &  \\ & {^ta_\infty^{-1}}\end{pmatrix}k_\infty g_\fini\right) = \left\langle \rho \left({^ta_\infty} \right) c_f(h,y, \gamma), \rho^*(k_\infty)v^* \right\rangle \mathbf{e}(\tr(hz))
	\end{align}
	for
	\[
	\begin{pmatrix}\mathbf{1}_n & n_\infty \\ & \mathbf{1}_n\end{pmatrix} \begin{pmatrix}a_\infty &  \\ & {^ta_\infty^{-1}}\end{pmatrix} \in P_{n,n}(\bbR), \, y = a_\infty{^t a_\infty},\, z= n_\infty+\sqrt{-1}\, y , k_\infty \in K_{n,\infty}, \, g_\fini \in G_n(\bbA_{\bbQ,\fini}).
	\]
	Here $\gamma$ is an element of $G_n(\bbQ)$ such that $\gamma^{-1} g_\fini \in K_\fini G_n(\bbR)$.
	
	For a nearly holomorphic modular form $f$ of weight $\rho$ and $\gamma \in G_n(\bbQ)$, let 
	\[
	(f|_\rho\gamma)(z) = \sum_h c_f(h,y,\gamma) \bfe(\tr (hz))
	\]
	be the Fourier expansion of $f|_\rho\gamma$.
	The following statement follows from (\ref{rel_FC_MF_AF}) and Proposition \ref{Fourier_coeff}.
	
	\begin{lem}\label{left_SL_inv}
	For any $h \in \Sym_n^{(j)}(F)$ and any nearly holomorphic automorphic form $\varphi \in \calN(G_n)$, the Whittaker-Fourier coefficient $\Wh_{\varphi,h}$ is left $\bfm(\SL_j(\bbA_F)N_{j,n,\GL}(\bbA_F))$-invariant.
	\end{lem}
	\begin{proof}
	The left $\bfm(\SL_j(F)N_{j,n,\GL}(F))$-invariance is clear by the definition of $\varphi$ and $\Wh_{\varphi,h}$.
	For any $g \in G_n(\bbA_F)$ and $m \in \bfm(\SL_j(\bbA_F)N_{j,n,\GL}(\bbA_F))$, put $\Wh_{\varphi,h,g}(m) = \Wh_{\varphi,h}(mg)$.
	Since $\varphi$ is right $K_n$-finite, there exists an open compact subgroup $K$ of $\bfm(\SL_j(\bbA_{F,\fini})N_{j,n,\GL}(\bbA_{F,\fini}))$ such that $K$ fixes the function $\Wh_{\varphi,h,g}$ on $\bfm(\SL_j(\bbA_F)N_{j,n,\GL}(\bbA_F))$ under the right translation.
	By the strong approximation, it remains to show that $\Wh_{\varphi,h}$ is left $\bfm(\SL_j(F_\infty)N_{j,n,\GL}(F_\infty))$-invariant.
	Take $a \in \SL_j(F_\infty)$ and $n \in N_{j,n,\GL}(F_\infty)$.
	We assume that $\varphi$ is equal to $\langle \varphi_f,v^*\rangle$ for a nearly holomorphic modular form $f$ of weight $(\rho,V)$ and $v^* \in V^*$.
	For
	\[
	\begin{pmatrix}\mathbf{1}_n & n_\infty \\ & \mathbf{1}_n\end{pmatrix} \begin{pmatrix}a_\infty &  \\ & {^ta_\infty^{-1}}\end{pmatrix} \in P_{n,n}(\bbR), \, y = a_\infty{^t a_\infty},\, z= n_\infty+\sqrt{-1}\, y , k_\infty \in K_{n,\infty}, \, g_\fini \in G_n(\bbA_{\bbQ,\fini}),
	\]
	we have
	\begin{align*}
	\Wh_{\varphi,h}\left(\bfm(an)\begin{pmatrix}\mathbf{1}_n & n_\infty \\ & \mathbf{1}_n\end{pmatrix} \begin{pmatrix}a_\infty &  \\ & {^ta_\infty^{-1}}\end{pmatrix}k_\infty g_\fini\right)
	&=\Wh_{\langle\varphi_f,v^*\rangle,h} \left(\bfm(an)\begin{pmatrix}\mathbf{1}_n & n_\infty \\ & \mathbf{1}_n\end{pmatrix} \bfm(a_\infty)k_\infty g_\fini\right)\\
	&= \langle \rho({^t(ana_\infty)}) c_f(h,any{^t(an)}, \gamma), \rho^*(k_\infty)v^*\rangle \\
	&\qquad \qquad \qquad \qquad \qquad \qquad \qquad  \times \mathbf{e}(\tr(hanz{^t(an)}))
	\end{align*}
	by (\ref{rel_FC_MF_AF}).
	Since $h$ lies in $\Sym_n^{(j)}(F)$, we have $\mathbf{e}(\tr(haz{^ta})) = \mathbf{e}(\tr(hz))$.
	By Proposition \ref{Fourier_coeff}, we have
	\[
	\rho({^t(ana_\infty)}) c_f(h,any{^t(an)}, \gamma) = \rho({^t(a_\infty)})\rho({^t(an)}) c_f(h,any{^t(an)}, \gamma) = \rho({^ta_\infty})c_f(h,y,\gamma).
	\]
	Hence, we have
	\begin{align*}
	\Wh_{\varphi,h}\left(\bfm(an)\begin{pmatrix}\mathbf{1}_n & n_\infty \\ & \mathbf{1}_n\end{pmatrix} \begin{pmatrix}a_\infty &  \\ & {^ta_\infty^{-1}}\end{pmatrix}k_\infty g_\fini\right)
	&=\left\langle \rho \left({^ta_\infty}\right) c_f(h,y, \gamma), \rho^*(k_\infty)v^* \right\rangle \mathbf{e}(\tr(hz)) \\
	&=\Wh_{\varphi,h}\left(\begin{pmatrix}\mathbf{1}_n & n_\infty \\ & \mathbf{1}_n\end{pmatrix} \begin{pmatrix}a_\infty &  \\ & {^ta_\infty^{-1}}\end{pmatrix}k_\infty g_\fini\right).
	\end{align*}
	This completes the proof.
	\end{proof}

	\subsection{Whittaker-Fourier coefficients and constant terms}
%	Let $P$ be a parabolic subgroup of $\Sp_{2n}$ with Levi decomposition $P=MN$.
%	We assume $M=\GL_1 \times \Sp_{2(n-1)}$.
	Fix $1 \leq j \leq n$.
	Let $U_1$ and $U_2$ be unipotent subgroups of $\Sp_{2n}$ defined by
	\[
	U_1= \left\{\begin{pmatrix}\mathbf{1}_j&*&0&0\\0&\mathbf{1}_{n-j}&0&0_{n-j}\\&&\mathbf{1}_j&0\\&&*&\mathbf{1}_{n-j}\end{pmatrix} \right\}, \qquad 
	U_2= \left\{\begin{pmatrix}\mathbf{1}_j&0&*&*\\0&\mathbf{1}_{n-j}&*&0_{n-j}\\&&\mathbf{1}_j&0\\&&0&\mathbf{1}_{n-j}\end{pmatrix} \right\}.
	\]
	Then $U_1$ and $U_2$ are subgroups of the unipotent subgroup $N_{P_{j,n}}$.
	Take $\varphi \in \calN(G_n)$.
	We compute a constant term of $\varphi$ along $P_{j,n}$.
	
	\begin{lem}\label{Wh_coeff}
	For $\varphi \in \calN(G_n)$, we have
	\[
	\varphi_{P_{j,n}}(g) = \sum_{h\in \Sym^{(j)}_n(F)} \Wh_{\varphi,h}(g)
	\]
	for $g \in G_n(\bbA_\bbQ)$.
	\end{lem}
	\begin{proof}
	For $g \in G_n(\bbA_\bbQ)$, one has
	\begin{align*}
	\varphi_{U_2}(g) 
	&= \int_{U_2(F) \bs U_2(\bbA_F)} \varphi(ng) \, dn = \int_{U_2(F) \bs U_2(\bbA_F)} \left(\sum_{h\in \Sym_n(F)} \Wh_{\varphi,h}(g) \psi_h(n)\right) \, dn\\
	&= \int_{U_2(F) \bs U_2(\bbA_F)} \left(\sum_{h\in \Sym^{(j)}_n(F)} \Wh_{\varphi,h}(g)\right) \, dn = \sum_{h\in \Sym^{(j)}_n(F)} \Wh_{\varphi,h}(g).
	\end{align*}
	Indeed, if $h \not\in \Sym_n^{(j)}(F)$, the character $\psi_h$ is non-trivial on the group $U_2(F) \bs U_2(\bbA_F)$.
	By $N_{P_{j,n}}(\bbA_\bbQ)=U_1(\bbA_F)U_2(\bbA_F)$, $U_1(\bbA_F) = \bfm(N_{j,n,\GL}(\bbA_F))$, and Lemma \ref{left_SL_inv}, the constant term $\varphi_{P_{j,n}}(g)$ is equal to
	\begin{align*}
	\varphi_{P_{j,n}}(g) = \int_{U_1(F) \bs U_1(\bbA_F)} \left(\sum_{h\in \Sym^{(j)}_n(F)} \Wh_{\varphi,h}(ng)\right) \, dn = \sum_{h\in \Sym^{(j)}_n(F)} \Wh_{\varphi,h}(g) = \varphi_{U_2}(g).
	\end{align*}
	This completes the proof.
	\end{proof}
	
	Next proposition is a key statement in this paper.
	
	\begin{prop}\label{Wh_coeff_main}
	For any nearly holomorphic modular form $\varphi \in \calN(G_n)$, there exists a finite number of Hecke characters $\mu_1,\ldots,\mu_\ell$ of $F^\times \bs \bbA_F^\times$ and subrepresentations $(\pi_1,V_1),\ldots,(\pi_\ell,V_\ell)$ of $\calN(G_{n-j})$ such that $\varphi_{P_{j,n}}$ lies in the sum
	\[
	\sum_{k=1}^\ell \Ind_{P_{j,n}(\bbA_\bbQ)}^{G_n(\bbA_\bbQ)} (\mu_k \boxtimes \pi_k).
	\]
	\end{prop}
	
	\begin{proof}
	We may assume that $\varphi$ has a $K_{n,\infty}$-type $\rho^*$.
	We denote by $V_{\rho^*}$ a model of $\rho^*$.
	We assume that $\varphi$ corresponds to $f \otimes v^*$ by the isomorphism (\ref{corr_MF_AF}).
	For $k \in K_n$, let $(\varphi_{P_{j,n}})_k$ be the function on $\GL_j(\bbA_F) \times G_{n-j}(\bbA_\bbQ)$ defined by
	\[
	(\varphi_{P_{j,n}})_k(m_1,m_2) = \varphi_{P_{j,n}}(\bfm(m_1)m_2k) = \varphi_{P_{j,n}}(m_2\bfm(m_1)k), \qquad (m_1,m_2) \in \GL_j(\bbA_F) \times G_{n-j}(\bbA_\bbQ).
	\]
%	We first consider the automorphic representation $\tau$ of $\GL_j(\bbA_F)$ generated by $(\varphi_{P_{j,n}})_k$.
	Note that $(\varphi_{P_{j,n}})_k$ is right $\SL_j(\bbA_F)$-invariant.
	Indeed, for any $(m_1,m_2) \in \GL_j(\bbA_F) \times G_{n-j}(\bbA_F)$ and $h \in \SL_j(\bbA_F)$, one obtains
	\[
	(\varphi_{P_{j,n}})_k(m_2 \, \bfm(m_1h)) = (\varphi_{P_{j,n}})_k(m_2 \, \bfm(h \cdot h^{-1}m_1hm_1^{-1} \cdot m_1)) = (\varphi_{P_{j,n}})_k(m_2 \, \bfm(m_1)).
	\]
	Fix $m_2 \in G_{n-j}(\bbA_\bbQ)$.
	Put $\phi = (\varphi_{P_{i,n}}(\,\cdot\,,m_2))_k$.
	Let $\tau_{k,m_2,\infty}$ be the representation of $\GL_j(F_\infty)$ generated by $(\varphi_{P_{i,n}}(\,\cdot\,,m_2))_k$.
	By the formula (\ref{rel_FC_MF_AF}) and Lemma \ref{P=Q}, the automorphic form $(\varphi_{P_{i,n}})_k(\,\cdot\,,m_2)$ on $\GL_j(\bbA_F)$ can be realized on the space $P(\Sym_n(F_\infty\otimes_\bbR\bbC),V_{\rho^*})$ of $V_{\rho*}$-valued polynomials on $\Sym_n(F_\infty \otimes_\bbR \bbC)$.
	Indeed, for any $h \in \Sym^{(j)}_n(F)$, the Whittaker-Fourier coefficient $r(m_2k)\Wh_{\varphi,h}|_{\GL_j(\bbA_F)}$ is an element of $P(\Sym_n(F_\infty\otimes_\bbR\bbC),V_\rho)$ by (\ref{rel_FC_MF_AF}).
	By the near holomorphy of $\varphi$, there exists a constant $N_\varphi$, which is independent of $h, m_2$ and $k$, such that the total degree of $r(m_2k)\Wh_{\varphi,h}|_{\GL_j(\bbA_F)}
	$ is bounded by $N_\varphi$.
	Hence, the representation $\tau_{k,m_2,\infty}$ can be realized on the space of polynomials $P(\Sym_n(F_\infty\otimes_\bbR\bbC),V_\rho)_{\leq N_\varphi}$ with the total degree $\leq N_\varphi$.
%	Recall that $r(m_2k)Wh_{\varphi,h}|_{\GL_j(\bbA_F)}$ is left $\SL_j(\bbA_F)$-invariant, by Lemma \ref{left_SL_inv}.
	Note that the automorphic form $r(m_2k)\Wh_{\varphi,h}|_{\GL_j(\bbA_F)}$ is right $\SL_j(\bbA_F)$-invariant.
	Therefore, there exist characters $\chi_1, \ldots,\chi_\ell$ of $\GL_j(F_\infty)$, which are independent of $m_2$, such that $\tau_{k,m_2,\infty}$ decomposes as a sum of $\chi_1,\ldots,\chi_\ell$ with finite multiplicities.
	Let $\tau$ be the representation of $\GL_j(\bbA_F)$ generated by $\phi$.
	Since $\phi$ is right $\SL_j(\bbA_F)$-invariant, the representation $\tau$ factors through the determinant map.
	By the above discussion, the infinity component of $\tau$ decomposes as a sum of characters.
	Hence $\tau$ decomposes as a sum of characters $\mu_1,\ldots,\mu_\ell$ with finite multiplicities.
	
	Let $\sigma$ be the automorphic representation of $M_{P_{i,n}}(\bbA_\bbQ)$ generated by $(\varphi_{P_{j,n}})_{k}$.
	Then we have the decomposition $\sigma \cong \bigoplus_{\mu} \mu \boxtimes \sigma[\mu]$, where $\sigma[\mu]$ is the $\mu$-isotypic component of $\sigma$.
	Note that $\sigma[\mu]$ is an automorphic representation of $G_n(\bbA_\bbQ)$ and $\sigma[\mu]$ is zero for almost all $\mu$.
	Hence the automorphic form $(\varphi_{P_{j,n}})_k$ decompose as $\sum_{\mu} \phi_\mu$ such that $\phi_\mu \in \mu \boxtimes \sigma[\mu]$.
	Note that $\sigma[\mu]$ is generated by $\phi_\mu$.
	We claim that $\phi_\mu$ is nearly holomorphic for any $\mu$.
	Let $a$ be a positive integer such that $\frakp_{n,-}^a \cdot \varphi = 0$.
	For any $X_1,\ldots,X_a \in \frakp_{n-j,-}$, we have
	\[
	\sum_\mu X_1 \cdots X_a \cdot \phi_\mu = X_1\cdots X_a \cdot (\varphi_{P_{j,n}})_k = ((\mathrm{Ad}(k)X_1 \cdots X_a) \cdot \varphi_{P_{j,n}})_k = 0.
	\]
	Since automorphic forms $\phi_\mu$ are linearly independent, the automorphic form $X_1 \cdots X_a \cdot \phi_\mu$ is zero for any $X_1,\ldots, X_a \in \frakp_{n-j,-}$.
	Hence the claim holds.
	By the above discussion, the automorphic representation $\sigma$ decomposes as $\bigoplus_\mu \mu \boxtimes \sigma[\mu]$, where $\sigma[\mu]$ is zero or nearly holomorphic.
	This means that for any $k \in K_n$, there exist Hecke characters $\mu_1,\ldots,\mu_\ell$ and subrepresentations $(\pi_\ell,V_\ell)$ of $\calN(G_{n-j})$ such that the automorphic form $(\varphi_{P_{j,n}})_k$ lies in the representation space of $\bigoplus_q \mu_q \boxtimes \pi_q$.
	Note that the choices of $\mu_q$ and $\pi_q$ depend on $k$.
	
	Let $K_\fini$ be an open compact subgroup of $G_n(\bbA_{F,\fini})$ such that $K_\fini \cdot \varphi = \varphi$.
	Since the double coset $P_{n,j}(\bbA_\bbQ) \bs G_n(\bbA_\bbQ)/ K_\fini$ is finite, it suffices to consider a finite set of elements in $K_n$.
	Let $\{k_1,\ldots,k_\ell\}$ be a set of representatives of it.
	There exist Hecke characters $\mu_{1,k_q}, \ldots, \mu_{i_q,k_q}$ and subrepresentations $\pi_{1,k_q}, \ldots, \pi_{i_q,k_q}$ such that the automorphic form $(\varphi_{P_{j,n}})_k$ lies in the representation space of $\sum_{q=1}^{\ell}\sum_{r = 1}^{i_q} \mu_{r,k_q} \boxtimes \pi_{r,k_q}$.
	This completes the proof.
	\end{proof}
	
	We compute the constant term along $Q_{i,n}$.
	
	\begin{cor}\label{P=Q}
	For $\varphi \in \calN(G_n)$, we have
	\[
	\varphi_{P_{j,n}} = \varphi_{Q_{j,n}}.
	\]
	for any $j$.
	\end{cor}
	\begin{proof}
	Put $U_3 = N_{P_{j,n}} \bs N_{Q_{j,n}}$.
	Then we have
	\[
	\varphi_{Q_{j,n}}(g) = \int_{U_3(\bbQ) \bs U_3(\bbA_\bbQ)} \varphi_{P_{j,n}} (ng) \, dn.
	\]
	We may identify the group $U_3$ with a unipotent subgroup of $\bfm(\SL_j(\bbA_F))$.
	Since $\varphi_{P_{j,n}}$ is left $\bfm(\SL_j(\bbA_F))$-invariant, we have $\varphi_{P_{j,n}} = \varphi_{Q_{j,n}}$.
	\end{proof}

	We say that a modular form $f$ is cuspidal if the Fourier coefficient $c_f(h,y,\gamma)$ is zero for any $h \not>0$ and and $\gamma \in \Sp_{2n}(F)$.
	
	\begin{prop}\label{equiv_def_cusp}
	Let $f$ be a nearly holomorphic modular form of weight $(\rho,V)$.
	Let $\langle\varphi_{f},v^*\rangle$ be an automorphic form corresponding to $f$ and $v^* \in V^*$.
	Then, the following conditions are equivalent:
	\begin{itemize}
	\item[(1)] $f$ is a cusp form.
	\item[(2)] $\langle\varphi_{f},v^*\rangle$ is cuspidal for any $0 \neq v^* \in V^*$.
	\item[(3)] The constant term of $\langle\varphi_{f},v^*\rangle$ along $P_{1,n}$ is zero for any $0 \neq v^* \in V^*$.
	\end{itemize}
	\end{prop}
	\begin{proof}
	Suppose $f$ is a cusp form.
	Put $\varphi = \langle\varphi_{f},v^*\rangle$.
	By (\ref{rel_FC_MF_AF}), if $\Wh_{\varphi,h}$ is non-zero, $\langle c_f(h,y,\gamma), \rho^*(k_\infty)v^* \rangle$ is non-zero for some $y, \gamma$, and $k_\infty$.
	Hence, by the definition of cusp forms, the Whittaker-Fourier coefficient $\Wh_{\varphi,h}$ is zero if $h$ is not positive definite.
	Let $P$ be a maximal parabolic subgroup of $G_n$.
	The constant term $\varphi_P$ is zero by Lemma \ref{Wh_coeff}.
	Hence (1) implies (2).
	The statement (2) implies (3) clearly.
	
	We assume the constant term $\varphi_{P_{1,n}}$ is zero.
	Assume that $f$ is not cuspidal.
	Then there exists a Fourier coefficient $c_{f}(h,y,\gamma)$ is non-zero for some non-positive definite matrix $h$.
	Take $a \in \GL_n(F)$ such that the $(1,1)$-th entry of $ah{^ta}$ is zero.
	We denote by $a_\fini$ an element of $G_n(\bbA_{\bbQ,\fini})$ such that
	\[
	(a_{\fini})_v = a
	\]
	for any non-archimedean place $v$ of $F$.
	Put $\varphi' = r(\bfm(a_\fini))\varphi$. 
	Then we have 
	\[
	\Wh_{\varphi',h}|_{G_n(\bbR)} = \Wh_{\varphi,ah{^ta}}|_{G_n(\bbR)}.
	\]
	Since $\varphi_{P_{1,n}}$ is zero, $\varphi'_{P_{1,n}}$ is zero.
	This contradicts to $c_f(h,y,\bfm(a)) \neq 0$ by (\ref{rel_FC_MF_AF}).
	Hence $f$ is cuspidal.
	This completes the proof.
	\end{proof}
	
	\begin{rem}
	This proposition is a generalization of \cite[Lemma 5]{Asgari-Schmidt}.
	\end{rem}

	\subsection{Proof of Proposition \ref{decomp_NHAF_parab}}\label{pf_decomp_NHAF_parab}
	
	We proof inductively in $n$.
	When $n=1$, this is proved in \cite[\S 3.8]{Horinaga}.
	Let $P$ be a maximal parabolic subgroup of $G_n$ with $P \neq P_{1,n} = Q_{1,n}$ i.e., the Levi subgroup of $P$ is $\Res_{F/\bbQ}\GL_j \times G_{n-j}$ for some $j \geq 2$.
	We assume $\calN(\Sp_{2n})_{\{P\}} \neq 0$.
	Take $\varphi \in \calN(\Sp_{2n})_{\{P\}}$.
	By the definition of $P$ and $\calN(\Sp_{2n})_{\{P\}}$, the constant term $\varphi_{P_{1,n}}$ along $P_{1,n}$ is zero.
	Then $\varphi$ is a cusp form by Proposition \ref{equiv_def_cusp}.
	This contradicts to $G \neq P$.
	Hence $\calN(\Sp_{2n})_{\{P\}} =0$ if $P$ is maximal and $P \neq P_{1,n}$.
	
	Next we consider the restriction $(r(\bfm(a)k)\varphi_{P_1})|_{G_{n-1}(\bbA_\bbQ)}$ for $a \in \GL_1(\bbA_F)$ and $k \in K_n$.
	We may assume the function $(r(\bfm(a)k)\varphi_{P_1})|_{G_{n-1}(\bbA_\bbQ)}$ is non-zero for some $a$ and $k$.
	Then by the assumption there exists $j$ such that $(r(\bfm(a)k)\varphi_{P_1})|_{G_{n-1}(\bbA_\bbQ)}$ lies in $\calN(G_{n-1})_{\{Q_{j,n-1}\}}$.
	In this case, the Whittaker-Fourier coefficient $\Wh_{\varphi,h}$ is zero if $\rank(h)<j$.
	By Lemma \ref{Wh_coeff} and Proposition \ref{equiv_def_cusp}, the constant term $\varphi_{Q_{j,n}}$ is a cuspidal automorphic form in $\calA(Q_{j,n}\bs G_n)$.
	This completes the proof.
%	If the constant term $\varphi_{P_1}$ is non-zero, there exist $a \in \GL_1(\bbA)$ and $k \in K_n$ such that $(r(ak)\varphi_{P_1})|_{G_{n-1}(\bbA_\bbQ)}$ is non-zero.
%	Suppose that $(r(ak)\varphi_{P_1})|_{\Sp_{2(n-1)}} \in \calN(\Sp_{2(n-1)})_{P'}$ for some maximal parabolic subgroup of $\Sp_{2(n-1)}$
%	By the above discussion the Levi factor of $P'$ is $\GL_1 \times \Sp_{2(n-2)}$ if $P'$ is maximal.
%	Hence for a parabolic subgroup $P$ with $\Sp_{2(n-2)} \subset P$, if $\calN(\Sp_{2n})_P \neq 0$, we have $P=Q_{1,n}$ or $Q_{2,n}$.
%	Inductively, $P$ is equal to $Q_{j,n}$ for some $j$.
%	This completes the proof.

	\subsection{Proof of Proposition \ref{decomp_NHAF_cusp_supp}}\label{pf_decomp_NHAF_cusp_supp}
	
	The statement (1) is clear by Proposition \ref{decomp_NHAF_parab}.
	We show (2).
	Take $\varphi \in \calN(G_n)_{\{Q_{i,n}\}}$.
	For $k \in K_n$, we define a function $(\varphi_{Q_{i,n}})_k$ on $\GL_i(\bbA_F) \times G_{n-i}(\bbA_{\bbQ})$ by $(\varphi_{Q_{i,n}})_k(m) = \varphi_{Q_{i,n}}(mk)$.
	By Proposition \ref{Wh_coeff_main}, the restriction $(\varphi_{Q_{i,n}})_k|_{\GL_i(\bbA_F)}$ decomposes as a sum of characters.
	Hence one has $\mu_1 = \cdots = \mu_{i}$.
	It remains to show $(\varphi_{Q_{i,n}})_k|_{G_{n-i}(\bbA_\bbQ)}$ is nearly holomorphic, i.e., $\frakp_{n-i,-}$-finite.
	By $\mathrm{Ad}(K_{n,\infty})(\frakp_{n-i,-}) \subset \frakp_{n,-}$ and $\frakp_{n,-}$-finiteness of $\varphi$, the automorphic form $(\varphi_{Q_{i,n}})_k|_{G_{n-i}(\bbA_\bbQ)}$ is $\frakp_{n,-}$-finite for any $k$.
	This completes the proof.

	\subsection{Proof of Proposition \ref{ss_Z}}\label{pf_ss_Z}
	
	We proceed inductively for $n$.
	If $n=1$, this is proved in \cite[p.~10]{pss1}.
	We assume that for any $j < n$, any generalized $\calZ_j$-eigen nearly holomorphic automorphic form on $G_j(\bbA_\bbQ)$ has an infinitesimal character.
	Let $\varphi$ be a nearly holomorphic automorphic form on $G_n(\bbA_\bbQ)$.
	If $\varphi$ is cuspidal, $\varphi$ generates a unitary $G(\bbA_{\bbQ_\fini})\times (\frakg_n,K_{n,\infty})$-module of finite length.
	Hence there exists a finite collection of cusp forms $\varphi_1,\ldots,\varphi_\ell$ such that $\varphi_\ell$ generates an irreducible representation for any $\ell$.
	In particular, $\varphi_\ell$ has an infinitesimal character for any $\ell$.
	
	We may assume $\varphi$ is orthogonal to all cusp forms.
	Hence we may assume $\varphi$ lies in $\calN(G_n)_{\{Q_{j,n}\}}$ for some $0 < j$.
	By Proposition \ref{Wh_coeff_main}, there exist finite number of subrepresentations $\Pi_1,\ldots,\Pi_\ell$ of $\calA(G_{n-i})$, which are generalized eigenspace of some infinitesimal character, and characters $\mu_1,\ldots,\mu_\ell$ of $\GL_j(\bbA_F)$ such that $\varphi_{P_{j,n}}$ is non-zero and lies in 
	\[
	\sum_{j=1}^\ell \Ind_{P_{i,n}(\bbA_\bbQ)}^{G_n(\bbA_\bbQ)}(\mu_\ell \boxtimes \Pi_{\ell}).
	\]
	Then we may assume $\Pi_\ell$ is a subrepresentation of $\calN(G_{n-i})$ by Proposition \ref{decomp_NHAF_cusp_supp}.
	By the assumption, $\Pi_\ell$ has an infinitesimal character.
	Then the induced representation $\Ind_{P_{i,n}(\bbA_\bbQ)}^{G_n(\bbA_\bbQ)}(\mu_\ell \boxtimes \Pi_{\ell})$ has an infinitesimal character.
	Hence $\varphi$ decomposes as a sum of nearly holomorphic automorphic forms which have infinitesimal characters by the injectivity of constant terms along $P_{j,n}$ on $\calN(G_n)_{\{Q_{j,n}\}}$.
	This completes the proof.
	
	\section{A decomposition of the space of holomorphic modular forms}
	
	\subsection{A decomposition of $M_\rho(\Gamma)$}
	
	Let $(\rho,V)$ be an irreducible finite-dimensional representation of $K_{n,\bbC}$.
	By (\ref{corr_MF_AF}), we have
	\[
	M_\rho(\Gamma) \otimes V^* \cong HK\left(\calN(G_n)^{K_\fini}_{\rho^*}\right).
	\]
	Here put
	\[
	HK\left(\calN(G_n)^{K_\fini}_{\rho^*}\right) = \{\varphi \in \calN(G_n)^{K_\fini}_{\rho^*} \mid \frakp_{n,-}\cdot\varphi=0\}.
	\]
	By the decomposition along associated classes of parabolic subgroups, we obtain the following result:
	
	\begin{lem}
	There exist subspaces $M_\rho(\Gamma)_{\{Q_{i,n}\}}$ of $M_\rho(\Gamma)$ such that
	\[
	\bigoplus_{i=0}^n M_\rho(\Gamma)_{\{Q_{i,n}\}} = M_\rho(\Gamma).
	\]
	Moreover, for any $i$, we have
	\[
	\bigoplus_{j \leq i} M_\rho(\Gamma)_{\{Q_{j,n}\}} = \left\{f \in M_\rho(\Gamma) \,\middle|\, \begin{matrix}\text{ $c_f(h,y,\gamma) = 0$ for any $\gamma \in \Sp_{2n}(F)$}\\ 
	\text{and any $h \in \Sym^{(i)}_n(\bbR)$}\end{matrix}\right\}.
	\]
	\end{lem}
	\begin{proof}
	By Proposition \ref{decomp_NHAF_parab}, we have
	\[
	M_\rho(\Gamma) \otimes V^* \cong HK(\calN(G_n)^{K_\fini}_{\rho^*})) = \bigoplus_{i=0}^n HK\left(\calN(G_n)^{K_\fini}_{\rho^*,\{Q_{i,n}\}}\right).
	\]
	Since each space $HK\left(\calN(G_n)^{K_\fini}_{\rho^*,\{Q_{i,n}\}}\right)$ is closed under the right translation of $K_{n,\infty}$, there exist subspaces 
	\[
	M_\rho(\Gamma)_{\{Q_{i,n}\}}
	\] of $M_\rho(\Gamma)$
	such that
	\[
	M_\rho(\Gamma)_{\{Q_{i,n}\}} \otimes V^* \cong HK\left(\calN(G_n)^{K_\fini}_{\rho^*,\{Q_{i,n}\}}\right).
	\]
	Hence the first assertion holds.
	The second assertion follows from (\ref{rel_FC_MF_AF}).
	This completes the proof.
	\end{proof}
	
%	Suppose $\lambda_{n,v}>2n$.
%	Then any automorphic form in $\calN(G_n,\chi_{\lambda})_{\{Q_{i,n}\}}$ can be written as a sum of Eisenstein series for $i \neq 0$.
%	Hence in this case, we may regard $M_\rho(\Gamma)_{\{Q_{i,n}\}}$ as the space of Klingen Eisenstein series.
	
	\subsection{Square-integrable modular forms and the vanishing theorem}
	
	For $\varphi \in \calN(G_n)_{\{Q_{i,n}\}}$, there exist functions
	\[
	\varphi_{P_{i,n}}(\mu,s_0,\pi) \in  I_{P_{i,n}}(\mu|\cdot|^{s_0}\boxtimes\pi)
	\]
	such that $\varphi_{P_{i,n}} = \sum \varphi_{P_{i,n}}(\mu,s_0,\pi)$.
	
	 \begin{lem}
	 With the above notation, we consider the following condition ($*$):
	 \[
	 \text{If $\varphi_{P_{i,n}}(\mu,s_0,\pi)$ is non-zero, the real number $s_0$ is less than $0$.}
	 \]
	 Then $\varphi$ is square-integrable if and only if the condition ($*$) holds.
	 \end{lem} 
	 \begin{proof}
	 The statement follows immediately from \cite[Lemma I.4.11]{MW}:
	 \end{proof}
	 
	 We then obtain the following corollary:
	 \begin{cor}\label{leq_0}
	 Let $\varphi$ be a holomorphic automorphic form i.e., $\frakp_{n,-} \cdot \varphi = 0$.
	 Then as a $(\frakg_n,K_{n,\infty})$-module, $\varphi$ generates unitary $(\frakg_n,K_{n,\infty})$-module.
	 In particular, if $\lambda_{n,v} \leq 0$ for some $v$, a holomorphic modular form of weight $(\lambda_{1,v},\ldots,\lambda_{n,v})_v$ is zero or constant function.
	 \end{cor}
	 
	 The proof is the same as in \cite{RC_Horinaga}.
	 Then we can show the vanishing of the space of nearly holomorphic Hilbert-Siegel modular forms:
	 \begin{prop}\label{vanishing}
	 With the above notation, assume $n > 1$.
	 If $\lambda \neq 0$ and $\lambda_{n,v} = 0$ for some $v$, we have $N_{\rho_{\lambda}} (\Gamma) = 0$ for any congruence subgroup $\Gamma$. 
	 \end{prop}
	 \begin{proof}
	 We denote by $V$ the representation space of $\rho_\lambda$.
	 Suppose $N_{\rho_{\lambda}} (\Gamma)$ is non-zero.
	 Take $f \in N_{\rho_{\lambda}} (\Gamma)$.
	 Let $\varphi_f$ be a nearly holomorphic automorphic form corresponding to $f$ and $v^* \in V$.
	 Let $\pi_\infty$ be the $(\frakg_n,K_{n,\infty})$-module generated by $\varphi_f$.
	 Then there exist a highest weight $\mu = (\mu_{1,v},\ldots,\mu_{n,v})_v$ in $\pi_\infty$ such that $\mu_{n,v} \leq \lambda_{n,v}$ for any $v \in \bfa$.
	 By Corollary \ref{leq_0}, $\mu$ is $0$.
	 Hence there exists $Y \in \calU(\frakg_n)$ such that $Y \cdot \varphi_f$ is a constant function.
	 We denote by $\varphi'$ an element of $\pi_\infty$ which satisfies that there exists $Y' \in \frakp_{n,-}$ such that $Y'\cdot \varphi$ is a constant function.
	 Then we may assume $\varphi'$ is $\frakk_n$-dominant with non-zero weight $(\lambda'_{1,v},\ldots,\lambda'_{n,v})_v$.
	 If $\varphi' = r(g_\fini)\varphi'$ for any $g_\fini \in G_n(\bbA_{\bbQ,\fini})$, $\varphi'$ is a constant function, since $G_n(\bbQ)G_n(\bbA_{\bbQ,\fini})$ is dense in $G_n(\bbA_\bbQ)$.
	 This is contradiction, since the weight of $\varphi'$ is non-zero
	 .
	 Hence there exists $g_\fini \in G_n(\bbA_{\bbQ,\fini})$ such that $r(g_\fini)\varphi'-\varphi'$ is non-zero.
	 By the choice of $\varphi'$, the automorphic form $r(g_\fini)\varphi'-\varphi'$ is highest weight vector.
	 Then there exists a non-zero holomorphic modular form of weight $\mu$.
%	 Since $n$ is not equal to $1$, the weight $(\lambda'_{1,v},\ldots,\lambda'_{n,v})_v$ of $\varphi'$ satisfies $\lambda'_{n,v} = 0$ for any $v$.
	 This contradicts to Corollary \ref{leq_0}.
	 Hence $M_\rho(\Gamma)$ is zero.
	 This completes the proof.
	 \end{proof}

\section{Constant terms of Eisenstein series of Klingen type}
	\subsection{Definition and basic properties}
	Take a Hecke character $\mu$ and a cuspidal automorphic representation $\pi$ of $G_{2(n-j)}(\bbA_\bbQ)$.
	For a section $f_s$ of $\Ind_{P_{j,n}(\bbA_\bbQ)}^{G_n(\bbA_\bbQ)}(\mu|\cdot|^s \boxtimes \pi)$, we say that $f_s$ is standard if $f|_{K_n}$ is independent of $s$.
	Take a standard section $f_s$ of $\Ind_{P_{j,n}(\bbA_\bbQ)}^{G_n(\bbA_\bbQ)}(\mu|\cdot|^s \boxtimes \pi)$.
	Set
	\[
	E(g,s,f) = \sum_{\gamma \in P_{j,n}(\bbQ) \bs G_n(\bbQ)} f_s(\gamma g), \qquad g \in G_n(\bbA_\bbQ).
	\]
	We say that the Eisenstein series $E(\,\cdot\,,f,s)$ is of Klingen type or Klingen Eisenstein series.
	This is absolutely convergent if $\mathrm{Re}(s)>n-(j-1)/2$ and can be continued meromorphically to the whole $s$-plane .
	We compute the constant terms of Klingen Eisenstein series and discuss its properties. 
	
	For a parabolic subgroup $P$ and an automorphic form $\varphi$, we say that $\varphi$ is concentrated on $P$ if $\varphi$ lies in $\calA(G_n)_{\{P\}}$.
	The following statement is proved in \cite[Proposition 3.3]{WS_KR} when $j=n$.
	
	\begin{prop}\label{concentrate}
	Let $\mu$ be a Hecke character of $F^\times F_{\infty,+}^\times \bs \bbA_F^\times$ and $\pi$ an irreducible holomorphic representation of $G_{n-i}(\bbA_\bbQ)$.
	For a standard section $f_s$ of $\Ind_{P_{i,n}(\bbA_\bbQ)}^{G_n(\bbA_\bbQ)}(\mu|\cdot|^s \boxtimes \pi)$, the Eisenstein series $E(\, \cdot \,,s,f)$ is concentrated on $Q_{i,n}$.
	\end{prop}
	\begin{proof}
	Let $P$ be a standard parabolic subgroup of $G_n$.
	Suppose $E(\,\cdot\,,s,f)_P$ has a non-trivial cuspidal component i.e., there exists a cusp form $\varphi$ on $M_P(\bbA_\bbQ)^1$ such that the Petersson inner product of $\varphi$ and $E(\,\cdot\,,s,f)_P|_{M_P(\bbA_\bbQ)^1}$ is non-zero.
	
	It is well-known that
	\begin{align}\label{const_term_Klingen}
	E(g,s,f)_P = \sum_{w \in W_{P_{i,n}} \bs W_{G_n} / W_{P}} \sum_{\gamma \in M_P(\bbQ) \cap w^{-1}P_{i,n}(\bbQ)w \bs M_P(\bbQ)} f_{w,s}(\gamma g)
	\end{align}
	where
	\[
	f_{w,s}(g) = \int_{N_P(\bbQ) \cap w^{-1}P_{i,n}(\bbQ) w \bs N_P(\bbA_\bbQ)} f_s(wng) \, dn. 
	\]
	For details, see \cite[Proposition II.1.7]{MW}.
	By this formula and the cuspidality of $\pi$, we may assume $P$ contains the subgroup $G_{n-i}$ and an element $w$ normalizes the symplectic group $G_{n-i}$.
	In particular, $P$ contains $Q_{i,n}$.
	
	Since $f$ is left $N_P(\bbA_\bbQ)$-invariant, we have
	\[
	f_{w,s}(g) = \int_{N_P(\bbA_\bbQ) \cap w^{-1}P_{i,n}(\bbA_\bbQ)w \bs N_P(\bbA_\bbQ)} f_s(wng) \, dn.
	\]
	Clearly, for $m \in M_P(\bbA_\bbQ) \cap w^{-1}\GL_i(\bbA_\bbQ)w$, we have
	\[
	f_{w,s} (mg) = \omega_w(m) f_{w,s}(g)
	\]
	for some character $\omega_w$ on $M_P(\bbA_\bbQ) \cap w^{-1}\GL_i(\bbA_\bbQ)w$.
%	We denote by $\pi^w$ the cuspidal automorphic representation of $G_{n-i}$.
%	Then $f_{w,s}$ is a function in certain induced representation.
%	Hence the term in $E(\,\cdot\,,s,f)_P$ corresponding to $w$ may  be viewed as an Eisenstein series on $M_P(\bbA)$ if $M_P(\bbA) \cap w^{-1}\GL_i(\bbA)w$ is proper subgroup of $\GL_i(\bbA) \cap M_P(\bbA)$.
%	If $M_P(\bbA) \cap w^{-1}\GL_i(\bbA)w=\GL_i(\bbA)\cap M_P(\bbA)$, the corresponding term is a character on $M_P(\bbA) \cap w^{-1}\GL_i(\bbA)w$ and is a cusp form on $\Sp_{2(n-i)}$.
	Since $E(\,\cdot\,,s,f)_P$ has a non-trivial cuspidal component, the character $\omega_w$ is a cusp form on the group $M_P(\bbA_\bbQ) \cap w^{-1}\GL_i(\bbA_\bbQ)w$.
	Hence $M_P(\bbA_\bbQ) \cap w^{-1}\GL_i(\bbA_\bbQ)w$ is a product of $\Res_{F/\bbQ}\GL_1$.
	Then we have $P=Q_{j,n}$.
	This completes the proof.
	\end{proof}
	
	\begin{cor}
	For any $K$-finite standard section $f$, the Eisenstein series $E(\,\cdot\,,s,f)$ and $E(\,\cdot\,,s,f)_{P_{i,n}}$ have the same set of poles and zeros.
	\end{cor}
	\begin{proof}
	Fix a complex number $s_0$.
	Let $\Pi$ be the $G(\bbA_{\bbQ,\fini})\times(\frakg_n,K_{n,\infty})$-representation generated by $E(\,\cdot\,,s_0,f)$.
	By Proposition \ref{concentrate}, the map
	\[
	\Pi \longrightarrow C^\infty(G(\bbA_\bbQ)) \colon E \longmapsto E_{Q_{i,n}}
	\]
	is injective if the Eisenstein series is holomorphic at $s=s_0$.
	Moreover we have
	\[
	E(g,s,f)_{N_{Q_{i,n}}} = \int_{N_{Q_{i,n}}(\bbQ)N_{P_{i,n}}(\bbA_\bbQ) \bs N_{Q_{i,n}}(\bbA_\bbQ)} E(ng,s,f)_{P_{i,n}}\, dn
	\]
	and so they have the same set of zeros.
	
	We assume that $E(\,\cdot\,,s,f)$ has a pole of order $k$ at $s_0$ and $E(\,\cdot\,,s,f)_{Q_{i,n}}$ has a order of $k_0$ at $s_0$.
	Suppose $E(\,\cdot\,,s,f)_{P_{i,n}}$ has a pole of order $\ell$ at $s_0$.
	Then we have $k_0 \leq \ell \leq k$.
	Put $F = (s-s_0)^k E(\,\cdot\,,s,f)$.
	If $F_{Q_{i,n}} =0$, we have $F=0$ at $s=s_0$, by Proposition \ref{concentrate} and \cite[Lemma 3.7]{langlands}.
	Hence we have $k_0 \geq k$.
	This completes the proof.
	\end{proof}
	
	\subsection{Unramified calculations}
	In this subsection, put $G_n = \Sp_{2n}$.
	For $j \leq i \leq n$, put
	\[
	w_{i,j,n} =
	\left(
	\begin{array}{ccc|ccc}
	\mathbf{1}_{i-j}	&					&						&							&					&	\\
						&0_j 				&						& 							&-\mathbf{1}_j	&	\\
						&					&\mathbf{1}_{n-i}		&							&					&	\\
	\hline
						&					&						&\mathbf{1}_{i-j}			&					&	\\
						&\mathbf{1}_j		&						&							&0_j				&	\\
						&					&						&							&					&\mathbf{1}_{n-i}
	\end{array}
	\right) \in \Sp_{2n}(F).
	\]
	Fix $1\leq i \leq n$ and a non-archimedean place $v$.
	For a unitary character $\mu_v$ of $F_v^\times$, let $f_{s,v}$ be a standard section of $\Ind_{P_{i,n}(F_v)}^{G_n(F_v)} (\mu_v|\cdot|^s \boxtimes \pi_v)$ with $f(1)=1$, where $\pi_v$ is an unramified representation of $\Sp_{2(n-i)}(F_v)$ such that $\pi_v$ is a unique unramified constituent of the principal series representation
	\[
	\Ind_{B_{n-i}(F_v)}^{G_{n-i}(F_v)} \left(|\cdot|^{\alpha_1} \boxtimes \cdots \boxtimes |\cdot|^{\alpha_{n-i}}\right),
	\]
	i.e., the Satake parameter of $\pi_v$ is $\{q_v^{\pm\alpha_1},\ldots,q_v^{\pm\alpha_{n-i}}\} \in (\bbC^\times)^n/\mathfrak{S}_{n-i} \times (\bbZ/2\bbZ)^{n-i}$.
	Here $q_v$ is the order of the residue field of $F_v$.
	Put
	\[
	\xi_{i,v,\pi}(s) = \prod_{\ell=1}^{i} L_v(s+\ell-(i+1)/2,\pi_v,\mu_v) \times \prod_{1 \leq p < q \leq i} L_v(2s-i-1+p+q, \mu_v^2).
	\]
	In this case, by the Gindikin-Karpelevich formula, we have
	\[
	f_{w_{j},s}(1) = \frac{\xi_{j,v,\pi} (s+(i-j)/2)}{\xi_{j,v,\pi}(s+1+(i-j)/2)}.
	\]
%	where
%	\[
%	a_{i,j}(s) = 
%	\begin{dcases}
%	1 & \text{if $j=0$}\\
%	\prod_{1 \leq k < \ell \leq j} L_v(s+i+1-k-\ell,\mu_v^2) & \text{if $j \neq 0$}
%	\end{dcases}
%	\]
%	and
%	\[
%	b_{i,j}(s) = 
%	\begin{dcases}
%	1 & \text{if $j=0$}\\
%	\prod_{m=1}^j L_v(s+(i+1)/2-m,\pi_v,\mu_v) & \text{if $i \neq n$ and $j \neq 0$}\\
%	\prod_{m=1}^j L_v(s+(i+1)/2-m,\mu_v) & \text{if $i = n$ and $j \neq 0$}
%	\end{dcases}
%	\]
	Here $L_v(s,\pi_v,\mu_v)$ is the standard $L$-function of $\pi_v$ twisted by $\mu_v$, i.e., with the above notation, we define
	\[
	L_v(s,\pi_v,\mu_v) = 
	\begin{dcases}
	L_v(s,\mu_v) \times \prod_{k=1}^{n-i} \left(L_v(s-\alpha_k,\mu_v)L_v(s+\alpha_k,\mu_v) \right) & \text{if $i < n$}.\\
	L_v(s,\mu_v) & \text{if $i=n$}.
	\end{dcases}
	\]
	
	\subsection{Degenerate Eisenstein series on $\GL_n$}
	
	In this subsection, we recall the definition of degenerate Eisenstein series on $\GL_n$.
	For $j \leq n$, put
	\[
	P_{j,n,\GL} = \left\{\begin{pmatrix}a & b \\ 0 & d\end{pmatrix} \in \GL_n \, \middle|\, a \in \GL_j, d \in \GL_{n-j}, b \in \Mat_{j,n-j}\right\}.
	\]
	For Hecke characters $\mu_1$ and $\mu_2$ of $F^\times F^\times_{\infty,+} \bs \bbA_F^\times$ and complex numbers $s$ and $t$, set 
	\[
	I_{j,n,\GL}(\mu_1,\mu_2,s,t) = \Ind_{P_{j,n,\GL}(\bbA_F)}^{\GL_n(\bbA_F)} \left(\mu_1|\cdot|^s \boxtimes \mu_2|\cdot|^t\right).
	\]
%	and
%	\[
%	I_{j,n,\GL}^\op(\mu_1,\mu_2,s,t) = \Ind_{P_{j,n,\GL}^\op(\bbA_F)}^{\GL_n(\bbA_F)} \left(\mu_2|\cdot|^t \boxtimes \mu_1|\cdot|^s \right).
%	\]
%	Note that we have
%	\[
%	I_{j,n,\GL}(\mu_1,\mu_2,s,t) \cong I_{n-j,n,\GL}^\op(\mu_1,\mu_2,s,t).
%	\]
	For a standard section $f_s$ of $I_{j,n,\GL}(\mu_1,\mu_2,s,t)$, the Eisenstein series $E(g,f,s)$ is called degenerate Eisenstein series.
	In this case, degenerate Eisenstein series converges absolutely for $\mathrm{Re}(s-t) > n/2$.
	For other properties of degenerate Eisenstein series, see \cite{Hanzer-Muic}.
	In this paper, we only use the absolutely convergence of degenerate Eisenstein series for $\mathrm{Re}(s-t)>n/2$.
%	Some technichs of this paper is close to the one of \cite{Hanzer-Muic} and \cite{WS_KR}.
	
	\subsection{Constant terms of Klingen Eisenstein series}
	
	We rewrite the constant term of Klingen Eisenstein series along $P_{i,n}$ by intertwining operators and degenerate Eisenstein series on $\GL_i$.
	Let $f_s$ be a standard section of $\Ind_{P_{i,n}(\bbA_\bbQ)}^{G_n(\bbA_\bbQ)}(\mu|\cdot|^s \boxtimes \pi)$ for a Hecke character $\mu$ of $F^\times F_{\infty,+}^\times \bs \bbA_F^\times$ and a cuspidal representation $\pi$ of $G_{n-i}(\bbA_\bbQ)$.
	For $0 \leq j \leq i$, we recall that the constant term of $E(\cdot,f,s)$ along $P_{i,n}$ is expressed by
	\begin{align}\label{const_term_klingen}
	\sum_{j=0}^i \sum_{\gamma \in M_{P_{i,n}}(\bbQ) \cap w_j^{-1}P_{i,n}(\bbQ)w_j \bs M_{P_{i,n}}(\bbQ)} f_{w_j,s}(\gamma g).
	\end{align}
	Here we put
	\[
	f_{w,s}(g) = \int_{N_{P_{i,n}}(\bbA_\bbQ) \cap w^{-1}P_{i,n}(\bbA_\bbQ)w \bs N_{P_{i,n}}(\bbA_\bbQ)} f_s(wng) \, dn.
	\]
	Note that if $w = w_j$ for some $j$, the Levi subgroup $M_{P_{i,n}}(\bbQ) \cap w^{-1}P_{i,n}(\bbQ)w$ is isomorphic to $\GL_{i-j} \times \GL_j \times \Sp_{2(n-i)}$.
	For $w = w_j$ and $k \in K_n$, the restriction $r(k)f_{w_j}|_{\GL_i(\bbA_F)}$ is a standard section of 
	\[
	I_{j,i,\GL}(\mu,\mu^{-1}, s+n-(i+j-1)/2,-(s-n+i-(j+1)/2)).
	\]

	We compute the constant terms of Klingen Eisenstein series as follows: 
	
	\begin{prop}\label{const_term_Eis_ser}
	Let $\lambda = (\lambda_v)_v = (\lambda_{1,v},\ldots,\lambda_{n,v})_v$ be a $\frakk_n$-dominant weight.
	Suppose $\lambda_{n,v}$ is independent of $v$ and $\lambda_{n,v} = \cdots = \lambda_{n-i+1,v} > 2n-i+1$.
	Let $\pi$ be an irreducible holomorphic cuspidal automorphic representation of $G_{n-i}(\bbA_\bbQ)$ and $\mu$ a Hecke character of $F^\times F_{+,\infty}^\times\bs\bbA_F^\times$ such that $\pi_v \cong L(\lambda_{1,v},\ldots,\lambda_{n-i,v})$ and $\mu_v = \mathrm{sgn}^{\lambda_{n,v}}$ for any $v \in \bfa$.
	Let $f_s = \otimes_v f_{v,s}$ be a factorisable standard section of $\Ind_{P_{i,n}(\bbA_\bbQ)}^{G_n(\bbA_\bbQ)}(\mu|\cdot|^s\boxtimes\pi)$ such that $\otimes_{v \in \bfa} f_{v,\lambda_{n,v}-n+(j+1)/2}$ lies in the highest weight submodule $\boxtimes_{v\in \bfa}L(\lambda_v)$ of $\Ind_{P_{i,n}(\bbR)}^{G_n(\bbR)}(\mu|\cdot|^{\lambda_{n,v}-n+(i-1)/2}\boxtimes\pi_v))$.
	Then, we have
	\[
	E(g,f,\lambda_{n,v}-n+(i-1)/2)_{P_{i,n}} = \otimes_v f_{v,\lambda_{n,v}-n+(i-1)/2}(g).
	\]
	\end{prop}
	
	\begin{proof}
	To simplicity, we assume $F=\bbQ$.
	The same proof will work for any totally real field $F$.
	For a place $v$ of $\bbQ$, put
	\[
	f_{s,w,v}(g) = \int_{N_{P_{i,n}}(F_v) \cap w^{-1}P_{i,n}(F_v)w \bs N_{P_{i,n}}(F_v)} f_s(wng) \, dn.
	\]
	Since the standard section $f_s$ is factorisable, we have $f_{w,s}(g) = \prod_v f_{w,s,v}(g)$.
	
	At $s_0 = \lambda_{n,v} - n + (i - 1)/2$, by our assumption, the Eisenstein series $E(g,s,f)$ converges absolutely.
	Hence, the local integral $f_{w,v,s}$, the Euler products $\prod_{v} f_{w,v,s}$ and the sum over $M_{P_{i,n}}(\bbQ) \cap w_j^{-1}P_{i,n}(\bbQ)w_j \bs M_{P_{i,n}}(\bbQ)$ as in (\ref{const_term_klingen}) converge absolutely for $s=s_0$.
	Therefore, it remains to show that the archimedean local integral $f_{w,s,\infty}$  vanishes at $s=s_0$.
	
	By Theorem \ref{emb_main}, the induced representation $\Ind_{P_{i,n}(\bbR)}^{G_n(\bbR)}(\mathrm{sgn}^{\lambda_{n,v}}|\cdot|^{s_0}\boxtimes L(\lambda_1,\ldots,\lambda_n))$ has a highest weight vector of weight $\lambda$.
	The map $f_{s_0} \longmapsto f_{w,s_0,\infty}$ defines an intertwining map from $\Ind_{P_{i,n}(\bbR)}^{G_n(\bbR)}(\mathrm{sgn}^{\lambda_{n,v}}|\cdot|^{s_0}\boxtimes L(\lambda_1,\ldots,\lambda_n))$ to an induced representation $\Pi_\infty$ of $G_n(\bbR)$ induced from an irreducible representation of $\GL_{i-j}(\bbR) \times \GL_j(\bbR) \times G_{n-i}(\bbR)$.
	It is easy to show that $\Pi_\infty$ does not have a highest weight vector of weight $\lambda$ by Theorem \ref{emb_main} and the double induction formula, if $j \neq 0$.
	Hence if $j \neq 0$, the archimedean local integral $f_{w,s,\infty}$ vanishes at $s=s_0$.
	This completes the proof.

	\end{proof}
	
	\subsection{Computations of dot-orbits of certain weights}
	
	In this subsection, we compute the dot-orbits of certain weights.
	For simplicity, we assume $F=\bbQ$.
	Put
	\[
	I_{P_{i,n}} (s,\mu,L(\lambda_1,\ldots,\lambda_{n-i})) = \Ind_{P_{i,n}(\bbR)}^{G_n(\bbR)} \left(\mu|\cdot|^{s-n+(i-1)/2} \boxtimes L(\lambda_1,\ldots,\lambda_{n-i})\right)
	\]
	Recall that, we have
	\[
	\left(I_{P_{i,n}} (s,\mu,L(\lambda_1,\ldots,\lambda_{n-i})) \right)_{\frakp_-\fin} \neq 0
	\]
	only if $s \in \bbZ_{\leq \lambda_{n-i}}$, by Theorem \ref{emb_main}.
	Note that if $\left(I_{P_{i,n}} (s,\mu,L(\lambda_1,\ldots,\lambda_{n-i})) \right)_{\frakp_-\fin}$ contains a unitary highest weight module, we have $s \in \bbZ_{\geq 0}$.
	We will compute all candidates of $s$ such that
	\[
	\left(I_{P_{i,n}} (s,\mu,L(\lambda_1,\ldots,\lambda_{n-i}))\right)_{\frakp_-\fin} \neq 0
	\]
	and $I_{P_{i,n}} (s,\mu,L(\lambda_1,\ldots,\lambda_{n-i}))$ has a given infinitesimal character $\chi$.
	Put
	\[
	X = \{0,\ldots,\lambda_{n-i}\}, \qquad Y = \left(X \cap \bbZ_{\leq n-(i-1)/2}\right) \cup \left(X \cap \{2n-i+2, 2n-i+3, \ldots\}\right) \subset X.
	\]
	For elements $x_1,x_2 \in X$, we say that $x_1$ is equivalent to $x_2$ if and only if the induced representations $I_{P_{i,n}} (x_1,\mu,L(\lambda_1,\ldots,\lambda_{n-i}))$ and $I_{P_{i,n}} (x_2,\mu,L(\lambda_1,\ldots,\lambda_{n-i}))$ have a same infinitesimal character.
	This gives the equivalence condition $\sim$.
	Then the inclusion $Y \xhookrightarrow{\quad} X$ induces a surjective map $Y \longrightarrow X/\sim$, since $I_{P_{i,n}} (s,\mu,L(\lambda_1,\ldots,\lambda_{n-i}))$ and $I_{P_{i,n}} (2n-i+1-s,\mu,L(\lambda_1,\ldots,\lambda_{n-i}))$ have the same infinitesimal character.
	The goal of this subsection is the following:
	
	\begin{prop}\label{dot-orbit}
	The natural map $Y \longrightarrow X/\sim$ is bijective.
	\end{prop}
	\begin{proof}
	It is sufficient to show that the map is injective.
	Put $Z = \{0,\ldots,\lambda_{n-i}\} \cap \bbZ_{\leq n-(i-1)/2}$.
	We first show that the natural map $Z \xhookrightarrow{\quad} Y \longrightarrow X/\sim$ is injective.
	We may assume that the natural map from $Z$ is not injective, i.e., there exist integers $\ell, \ell' \in Z$ such that $\ell < \ell'$ and $\ell \sim \ell'$.
	Note that the infinitesimal character of the induced representation $I_{P_{i,n}} (s,\mu,L(\lambda_1,\ldots,\lambda_{n-i}))$ has the Harish-Chandra parameter $\rho + (\lambda_1,\ldots,\lambda_{n-i},s,\ldots,s)$.
	Put $\lambda(s) = \rho + (\lambda_1,\ldots,\lambda_{n-i},s,\ldots,s)$.
	Hence, there exists $w \in W$ such that
	\[
	w(\lambda_1-1,\ldots,\lambda_{n-i}-(n-i),\ell-(n-i+1),\ldots,\ell-n) = (\lambda_1-1,\ldots,\lambda_{n-i}-(n-i),\ell'-(n-i+1),\ldots,\ell'-n).
	\]
	Since $W \cong \mathfrak{S}_n \ltimes (\bbZ/2\bbZ)^n$, we can decompose $w$ as $w = \tau \cdot \sigma$ with $\tau \in (\bbZ/2\bbZ)^n$ and $\sigma \in \mathfrak{S}_n$.
	
	We first evaluate $\sigma(n)$.
	Assume that $\sigma(n) > n-i$.
	We then have
	\[
	|\ell - n| = |\ell' - \sigma(n)|.
	\]
	Hence one gets $\ell = n + \ell' -\sigma(n)$ or $n+\sigma(n)-\ell'$.
	Since $\ell < \ell'$ and $\sigma(n) \leq n$, we have $\ell = n + \sigma(n) -\ell'$.
	We then have $2n - i = n+n-i < \ell + \ell' < 2\ell'$ and hence $\rho_{n,i} < \ell'$.
	This contradicts to $\ell' \in Y$.
	
	Next, we may assume $\sigma(n) \leq n-i$.
	In this case, $|\ell -n |$ equals to $|\lambda_{\sigma(n)} - \sigma(n)|$.
	Since $\lambda_{\sigma(n)} \geq \ell \geq 0$, one obtains $\lambda_{\sigma(n)} = n + \sigma(n) - \ell$.
	Hence $\lambda_{\sigma(n)}$ is greater than or equal to $\sigma(n)$.
	
	We claim that for any $1 \leq j < \sigma(n)$, we have $\sigma(j) = j$.
	We may assume that there exists a positive integer $k$ such that $k \geq \sigma(n)$ and $\sigma(k) \leq \sigma(n)-1$.
	If $k \leq n-i$, we get $|\lambda_{k} - k| = |\lambda_{\sigma(k)} - \sigma(k)|$.
	Since $k \geq \sigma(k)$, one gets $\lambda_{\sigma(k)} = k + \sigma(k) - \lambda_{k}$.
	Now, we have $\lambda_{\sigma(n)} = n + \sigma(n) - \ell$.
	Hence, we obtain
	\begin{align*}
	k + \sigma(k) - \lambda_{k} = \lambda_{\sigma(k)} \geq  \lambda_{\sigma(n)} = n + \sigma(n) - \ell \\
	\ell \geq n-k + \sigma(n) - \sigma(k) + \lambda_k > \lambda_k.
	\end{align*}
	This contradicts to $\ell \leq \lambda_k$.
	We may assume $k > n-i$.
	We then have $|\ell - k| = |\lambda_{\sigma(k)} - \sigma(k)|$ and hence $\lambda_{\sigma(k)} = \sigma(k)+k-\ell$.
	Since $\ell = n + \sigma(n) - \lambda_{\sigma(n)}$, one gets $\lambda_{\sigma(k)} = \sigma(k) + k - \ell = \sigma(k) - \sigma(n) + k-n +\lambda_{\sigma(n)} < \lambda_{\sigma(n)}$.
	This contradicts to the choice of $k$.
	Hence the claim holds.
	
	By the claim, one gets $j = \sigma(\sigma(n)) > \sigma(n)$.
	Suppose that $j \leq n-i$.
	We then have $|\lambda_{j} - j| = |\lambda_{\sigma(n)} - \sigma(n)|$ and hence $\lambda_{j} = j + \sigma(n) - \lambda_{\sigma(n)} = j - n + \ell$.
	This contradicts to $\lambda_{j} \geq \ell$.
	Hence we conclude that $j > n-i$.
	In this case, we have $|\lambda_{\sigma(n)} - \sigma(n)| = |\ell' - j|$ and hence $\ell' = j + \sigma(n) - \lambda_{\sigma(n)} = j -n + \ell$.
	This contradicts to $\ell < \ell'$.
	To conclude that there are no pair $(\ell,\ell') \in Z \times Z$ such that $ \ell < \ell'$ and $\ell \sim \ell'$.
	Hence the natural map from $Z$ is injective.
	
	We finally show that the natural map from $Y$ is injective.
	We may assume $\lambda_{n,v} > 2n-i+1$.
	Indeed, if $\lambda_{n,v} \leq 2n-i+1$, the set $Y$ is equal to $Z$.
	Take $\ell \in Y \setminus Z$.
	We may assume that there exists $\ell \neq \ell' \in Y$ such that $\ell' \sim \ell$ and $w = \tau\sigma \in W$ such that $w(\lambda(\ell')) = \lambda(\ell)$.
	Since all entries of $\lambda(\ell)$ and $\lambda(\ell')$ are positive, $\tau$ is trivial.
	Then, it is easy to show that $w = 1$.
	This is contradiction.
	Hence we may assume $\ell' \in Z$.
	We claim that $\sigma(p) = p$ for $p = 1, \ldots, n-i$.
	If there exists $q \geq n-i+1$ such that $\sigma(q) \leq n-i$, we have $|\ell'-q| = \lambda_{\sigma(q)} - \sigma(q)$.
	Since $\lambda_{\sigma(q)} \geq \ell'$, we have 
	\[
	\ell' = q + \sigma(q) -\lambda_{\sigma(q)} < n + n-i - (2n - i + 1) < 0.
	\]
	This is contradiction.
	Hence the claim holds.
	Put $q = \sigma(n-i+1)$.
	By the above claim, one get $q \geq n-i+1$.
	We then have $\ell -(n- i + 1) = |\ell' - q|$ and hence $\ell' = q +n- i +1 - \ell \leq 2n - i +1 -\ell < 0$ by $\ell > 2n-i+1$.
	This is contradiction.
	Therefore the map from $Y$ is injective.
	This completes the proof.
	\end{proof}
	
	We then obtain the following corollary.
	
	\begin{cor}\label{p_-_fin_vects}
	We assume that $F = \bbQ$.
	Fix a positive integer $i \leq n$ and a $\frakk_n$-dominant weight $(\lambda_1,\ldots,\lambda_{n-i}) \in \bbZ^{n-i}$ such that $\lambda_{n-i} \geq n - (i-1)/2$.
	Then for any positive integer $s_0 \geq n-(i-1)/2$, the space of $\frakp_{n,-}$-vectors of $\Ind_{P_{i,n}(\bbR)}^{G_n(\bbR)} (\mu_\infty |\cdot|^{s_0-n+(i-1)/2}\boxtimes L(\lambda_1,\ldots,\lambda_{n-i}))$ is equal to
	\[
	L(\lambda_1,\ldots,\lambda_{n-i},s_0,\ldots,s_0)
	\]
	if $\mu_\infty = \mathrm{sgn}^{s_0}$.
	\end{cor}
	\begin{proof}
	The representation $L(\lambda_1,\ldots,\lambda_{n-i},s_0,\ldots,s_0)$ is contained in the induced representation by Theorem \ref{emb_main}.
	We may assume that the space of $\frakp_{n,-}$-finite vectors is not equal to it.
	We consider the constituents of the induced representation.
	Then there exists a $\frakk_n$-dominant weight $\mu$ such that $L(\mu)$ occur in the constituent of the induced representation.
	By the assumption, the representations $L(\mu)$ and $L(\lambda_1,\ldots,\lambda_{n-i},s_0,\ldots,s_0)$ have the same infinitesimal characters.
	We claim $\mu \not < (\lambda_1,\ldots,\lambda_{n-i},s_0,\ldots,s_0)$.
	We assume that $\mu < (\lambda_1,\ldots,\lambda_{n-i},s_0,\ldots,s_0)$.
	There exists an element $w \in W_{G_n}$ such that $w \cdot (\lambda_1,\ldots,\lambda_{n-i},s_0,\ldots,s_0) = \mu$.
	By the assumption $\mu < (\lambda_1,\ldots,\lambda_{n-i},s_0,\ldots,s_0)$, one has $\mu_n \geq s_0$.
	Recall the proof of Proposition \ref{dot-orbit}.
	In particular, by considering the value of the $n$-th entry of $w \cdot (\lambda_1,\ldots,\lambda_{n-i},s_0,\ldots,s_0)$ and $\mu_n \geq s_0$, one can show the equality $\mu = (\lambda_1,\ldots,\lambda_{n-i},s_0,\ldots,s_0)$ easily.
	This contradicts to $\mu < (\lambda_1,\ldots,\lambda_{n-i},s_0,\ldots,s_0)$.
	Hence one has $\mu \not < (\lambda_1,\ldots,\lambda_{n-i},s_0,\ldots,s_0)$.
	Then, there exists a highest weight vector $v$ of weight $\omega$ such that $v$ lies in the induced representation and $\omega \neq (\lambda_1,\ldots,\lambda_{n-i},s_0,\ldots,s_0)$.
	This contradicts to Theorem \ref{emb_main}.
	This completes the proof.
	\end{proof}
	\subsection{Proof of Lemma \ref{integrality_wt} and Theorem \ref{main}}\label{pf_main}
	
	\begin{proof}[Proof of Lemma \ref{integrality_wt}]
	Take a nearly holomorphic automorphic from $\varphi \in \calN(G_n,\chi_\lambda)$.
	Then there exists an element $X \in \calU(\frakg_n)$ such that $X \cdot \varphi$ is a highest weight vector.
	We denote by $\lambda' = (\lambda_1,\ldots,\lambda_n)$ the weight of $X \cdot \varphi$.
	Since $X \cdot \varphi$ generates a quotient representation of $N(\lambda')$ as a $(\frakg_n,K_\infty)$-module, $X \cdot \varphi$ has an infinitesimal  character $\chi_{\lambda'}$.
	Hence we have the equality $\chi_\lambda=\chi_{\lambda'}$.
	Then we have the integrality of $\lambda$.
	Indeed, since $K_{n,\infty}$ acts on $N(\lambda')$, the weight $\lambda'$ must be integral.
	This completes the proof of Lemma \ref{integrality_wt}.
	\end{proof}
	
	\begin{proof}[Proof of Theorem \ref{main}]
	We use the same notation as in the proof of Lemma \ref{integrality_wt}.
	By Proposition \ref{Wh_coeff_main} and Proposition \ref{dot-orbit}, there exists a number $s_0$ such that
	\[
	\calN(G_n,\chi_\lambda)_{(M_{Q_{i,n}}, \mu^{\boxtimes i} \boxtimes \pi)} \xhookrightarrow{\qquad} \Ind_{P_{i,n}(\bbA_{\bbQ})}^{G_{n}(\bbA_{\bbQ})} (\mu|\cdot|^{s_0} \boxtimes \pi).
	\]
	Then one has $s_0 = \lambda_{n,v}-n+(i-1)/2$.
	Moreover, by the near holomorphy, the right hand side contained in the space of $\frakp_{n,-}$-finite vectors in the right hand side.
	By Corollary \ref{p_-_fin_vects}, we have the embedding
	\begin{align}\label{emb_1}
	\calN(G_n,\chi_\lambda)_{(M_{Q_{i,n}}, \mu^{\boxtimes i} \boxtimes \pi)} \xhookrightarrow{\qquad} \left(\Ind_{P_{i,n}(\bbA_{\bbQ,\fini})}^{G_{n}(\bbA_{\bbQ,\fini})} (\mu|\cdot|^{s_0} \boxtimes \pi)\right) \boxtimes (\boxtimes_{v \in \bfa}L(\lambda_v)).
	\end{align}
	Conversely, take an element $\varphi$ in the right hand side.
	Let $\varphi_s$ be the standard section of $\Ind_{P_{i,n}(\bbA_\bbQ)}^{G_n(\bbA_\bbQ)}$ such that $\varphi_{s_0} = \varphi$.
	By Proposition \ref{const_term_Eis_ser}, the Klingen Eisenstein series gives the converse map of (\ref{emb_1}).
	This shows that the embedding (\ref{emb_1}) is isomorphism.
	This completes the proof.

	\end{proof}
	
	\subsection{Proof of Corollary \ref{suff_reg_case}}\label{proof_cor_main}
	
	We first compute the dot orbits of certain weights.

	\begin{lem}\label{lem1}
	Assume $F=\bbQ$.
	Let $\lambda$ be a $\frakk_n$-dominant integral weight such that $\lambda_{n} > 2n$ and $w$ an element of the Weyl group of $W_{G_n}$.
	Then, if $\omega = (\omega_1,\ldots,\omega_n)=w \cdot \lambda$ is $\frakk_n$-dominant, we have $\omega=\lambda$ or $\omega_n < 0$.
	\end{lem}
	\begin{proof}
	Suppose that $w \in W_{G_n}$ is a non-trivial element such that $w \cdot \lambda$ is $\frakk_n$-dominant.
	The element $w$ decomposes as $w = \tau_{I} \sigma$ such that $\tau_{I} \in (\bbZ/2\bbZ)^n$ and $\sigma \in \mathfrak{S}_n$ for some $I \subset \{1,\ldots,n\}$.
	Here $\mathfrak{S}_n$ is the symmetric group of degree $n$ and $\tau_I$ means that
	\[
	\tau_I(\lambda_1,\ldots,\lambda_n) = (\varepsilon_I(1)\lambda_1,\ldots,\varepsilon_I(n)\lambda_n), \qquad \varepsilon_I(k) = 
	\begin{cases}
	+1 & \text{if $k \not\in I$}\\
	-1 & \text{if $k \in I$}.
	\end{cases}
	\]
	In this case, $w\cdot \lambda$ is equal to
	\begin{align}\label{wt1}
	\tau_I\sigma(\lambda_1-1,\ldots,\lambda_n-n) + (1,\ldots,n).
	\end{align}
	Since this is $\frakk_n$-dominant, $I$ is non-empty, if $\sigma \neq 1$.
	Suppose $\sigma$ is non-trivial.
	Then some entry of (\ref{wt1}) is non-positive.
	Hence the lemma follows.
	Next we suppose $\sigma$ is trivial.
	Then $I$ is non-empty by $w \neq 1$.
	Hence some entry of (\ref{wt1}) is non-positive.
	This completes the proof.
	\end{proof}
	
	\begin{proof}[Proof of Corollary \ref{suff_reg_case}]
	Recall that we have
	\[
	\calN(G_n,\chi)_{\{Q_{i,n}\}} = \bigoplus_{\pi} \calN(G_n,\chi_\lambda)_{(M_{Q_{i,n}},\pi)},
	\]
	where $\pi$ runs through all cuspidal representations of $M_{Q_{i,n}}(\bbA_\bbQ)$.
	For an irreducible automorphic representation $\pi$ of $M_{Q_{i,n}}(\bbA_\bbQ)$, we denote by $\mu$ and $\tau$ the automorphic representations of $\GL_1(\bbA_\bbQ)$ and $G_{n-i}(\bbA_\bbQ)$ such that $\pi \cong \mu^{\boxtimes i} \boxtimes \tau$.
	By Proposition \ref{Wh_coeff_main}, if $\calN(G_n,\chi_\lambda)_{(M_{i,n},\pi)} \neq 0$, the automorphic representation $\tau$ is holomorphic.
	Let $(\omega_{1,v},\ldots,\omega_{n-i,v})$ be the highest weight of $\tau$.
	Then, by Lemma \ref{lem1}, we have $\omega_{j,v} = \lambda_{j,v}$ for any $1 \leq j \leq n-i$ and $v \in \bfa$.
	If $\mu \not \in \mathfrak{X}_{(-1)^{\lambda_{n,v}}}$, we have $\calN(G_n,\chi)_{(\mu^{\boxtimes i}\boxtimes \tau)} = 0$, by Theorem \ref{main} (2).
	Hence we have
	\[
	\calN(G_n,\chi)_{\{Q_{i,n}\}} = \bigoplus_{\pi} \calN(G_n,\chi_\lambda)_{(M_{Q_{i,n}},\pi)},
	\]
	where $(M_{Q_{i,n}},\pi)$ runs through the same set as in the statement.
	Therefore the statement follows from Theorem \ref{main} (2).
	\end{proof}

\section{Application for surjectivity of global Siegel operator}
	In this section, we prove the surjectivity of Siegel operators under certain conditions.

	\subsection{Siegel operators $\Phi$ and constant terms}
	
	We use the same notation as in \S \ref{FC_NHMF}.
	Let $f$ be a holomorphic Hilbert Siegel modular form on $\frakH_n^d$ of weight $(\rho_\lambda,V)$ with $\lambda=(\lambda_{1,v},\ldots,\lambda_{n,v})_v \in \bigoplus_{v \in \bfa} \bbZ^n$.
	We then define the Siegel operator $\Phi$ by
	\[
	\Phi(f)(z) = \lim_{t \rightarrow \infty} f \left(\begin{pmatrix}t & 0 \\ 0 & z \end{pmatrix}\right), \qquad z \in \frakH_{n-1}^d.
	\]
	Then the Fourier expansion of $\Phi(f)$ is
	\[
	\Phi(f)(z) = \sum_{h  \in \Sym_{n-1}(F)} c_f \left(\left(\begin{smallmatrix} 0 & 0 \\ 0 & h \end{smallmatrix}\right)\right) \mathbf{e}_{h}(z), \qquad z \in \frakH_{n-1}^d.
	\]
	By Proposition \ref{Fourier_coeff}, if $h$ lies in $\Sym_n^{(1)}(F)$, $c_f(h)$ lies in the space of $N_{1,n,\GL}^\op(\bbR)$-fixed vectors in $V$.
	Since the representation $\rho_\lambda$ is holomorphic, $c_f(h)$ lies in the space of $N_{1,n,\GL}^\op(\bbC)$-fixed vectors in $V$.
	The Levi subgroup $\GL_{1}(\bbC) \times \GL_{n-1}(\bbC)$ of $\GL_n(\bbC)$ acts on the space $V^{N_{j,n,\GL}^\op(\bbC)}$ naturally.
	Then as a $\GL_{n-1}(\bbC)$-representation, $V^{N_{j,n,\GL}^\op(\bbC)}$ is irreducible with highest weight $(\lambda_{1,v},\ldots, \lambda_{n-1,v})_v$.
	Hence we may regard $\Phi(f)$ as a Hilbert-Siegel modular form of weight $\rho_{\lambda'}$.
	Here $\lambda'$ is equal to $(\lambda_{1,v},\ldots,\lambda_{n-1,v})_v$.
	
	Suppose $f$ is modular with respect to a congruence subgroup $\Gamma$.
	Fix a set of complete representatives $\{\gamma_1,\ldots,\gamma_h\}$ of
	\[
	P_{i,n}(\bbQ) \bs G_n(\bbQ) / \Gamma.
	\]
	We then define the global Siegel operator $\widetilde{\Phi}$ with respect to $\Gamma$ by
	\[
	\widetilde{\Phi}(f) = (\Phi(f|_{\rho_\lambda}g_\ell))_{1 \leq \ell \leq h}.
	\]
	If $n=2$ and $F=\bbQ$, this operator $\widetilde{\Phi}$ is defined in \cite{boe-Ib}.
	For each $1 \leq \ell \leq h$, the function $\Phi(f|_{\rho}g_\ell)$ on $\frakH_{n-1}^d$ is modular with respect to $\Gamma_i = g_i\Gamma g_i^{-1} \cap \Sp_{2(n-1)}(F)$.
	Hence we obtain the linear map
	\[
	\widetilde{\Phi} \colon M_{\rho_\lambda}(\Gamma) \longrightarrow  \bigoplus_{\ell=1}^h M_{\rho_\lambda'}(\Gamma_i).
	\]
	Note that the operator depends on the choice of $g_i$ but not essentially.
	Indeed, if we replace $g_1$ to $g_1' \in P_{i,n}(F)g_1\Gamma$, we have the isomorphism
	\[
	M_{\rho_\lambda'}(g_i\Gamma g_i^{-1} ) \longrightarrow M_{\rho_\lambda'}(g_i'\Gamma g_i'^{-1}) \colon f \longmapsto f|_{\rho_\lambda}g_ig_i'^{-1}.
	\]
%	Therefore, there are no problem about choices of $\{g_1,\ldots,g_h\}$.
	
	\begin{quest}
	When the global Siegel operator $\widetilde{\Phi}$ is surjective?
	If it is not surjective, can we determine the image of it?
	\end{quest}
	
	In general, $\widetilde{\Phi}$ is not surjective.
	For a few cases, the images are determined.
	For the details, see \cite[p.~123]{boe-Ib}
	
	We represent the Siegel operator $\widetilde{\Phi}$ in terms of automorphic forms on $G_n(\bbA_\bbQ)$.
	We then give the general theory for surjectivity of Siegel operator $\widetilde{\Phi}$.
	For a $\frakk_n$-dominant integral weight $\lambda \in \bigoplus_{v\in \bfa}\bbZ^{n}$, put
	\[
	\mathcal{H}(G_n,\lambda) = \{\varphi \in \calN(G_n,\chi) \mid \text{$\varphi$ has $K_{n,\infty}$-type $\rho_\lambda$ and $\frakp_{n,-} \cdot \varphi = 0$}\}.
	\]
	Then by (\ref{corr_MF_AF}), we have the isomorphism
	\[
	M_{\rho_\lambda}(\Gamma)	\otimes V_\lambda \cong \mathcal{H}(G_n,\lambda).
	\]
	Note that if $\lambda_{n,v} > 2n$ for any $v \in \bfa$, the space $\mathcal{H}(G_n,\lambda)$ is equal to
	\[
	\mathcal{H}(G_n,\lambda) = \{\varphi \in \calN(G_n,\chi_\lambda) \mid \frakp_{n,-}\cdot \varphi = 0\}
	\]
	by Lemma \ref{lem1}.
	The following lemma states the relation of the Siegel operator $\widetilde{\Phi}$ and constant terms.
	
	\begin{lem}\label{comm_diag}
	With the above notation, the following diagram is commutative:
	\[
	\begin{CD}
	M_{\rho_\lambda}(\Gamma)	\otimes V_\lambda									@>{\sim}>>	\mathcal{H}(G_n,\lambda)^{K_\Gamma}\\
	@V{\widetilde{\Phi} \otimes \mathrm{Id}}VV														@VVV \\
	\left(\bigoplus_{\ell=1}^h M_{\rho_\lambda'}(\Gamma_\ell)\right) \otimes V_{\lambda}		@>{\quad}>>	\bigoplus_{\ell=1}^h\mathcal{H}(G_{n-1},\lambda')^{K_{\Gamma_\ell}} 
	\end{CD}
	\]
	Here the right vertical arrow is defined by $\varphi \longmapsto ((r(g_{\ell,\fini})\varphi_{P_{1,n}})|_{G_{n-1}})_\ell$.
	\end{lem}
	\begin{proof}
%	We decompose $V_\lambda$ as $V_{\lambda'}$ and $V_{\lambda'}^\bot$.
%	Here $V_{\lambda'}$ is the direct sum of irreducible $K_{n-1,\bbC}$-representations with highest weight $\lambda'$ and $V_{\lambda'}^\bot$ is the direct sum of $K_{n-1,\bbC}$-representations of highest weight $\mu$ with $\mu \neq \lambda'$.
%	Note that $V_{\lambda'}$ is irreducible and equal to the space of $N_{1,n,\GL}^\op(\bbC)$-fixed vectors.
%	We may regard $V_{\lambda'}^*$ and $V_{\lambda'}^{\bot,*}$ as the space of functions which are zero on $V_{\lambda'}^{\bot}$ and $V_{\lambda'}^*$, respectively.
%	Then we have the direct sum $V_{\lambda}^* = V_{\lambda'}^* \oplus V_{\lambda'}^{\bot,*}$.
%	We denote by $p_1$ and $p_2$ the projection of $V_{\lambda}^*$ to $V_{\lambda'}^*$ and $V_{\lambda'}^{\bot,*}$, respectively.
	
	Take $f \otimes v^* \in M_{\rho_\lambda}(\Gamma)	\otimes V_{\lambda}$.
	Then by the left vertical arrow, we have a modular form $(\Phi(f|_{\rho}\gamma_\ell))_\ell \otimes v^*$ which lies in $\left(\bigoplus_{\ell=1}^h M_{\rho_\lambda'}(\Gamma_\ell)\right) \otimes V_{\lambda}$.
	Put $\Phi_\ell(f) = \Phi(f|_{\rho}\gamma_\ell)$.
	For each $\ell$, the corresponding automorphic form has the $h$-th Whittaker-Fourier coefficient
	\[
	\Wh_{\langle\varphi_{\Phi_\ell(f)},v^*\rangle,h}\left(\begin{pmatrix}\mathbf{1}_n & n_\infty \\ & \mathbf{1}_n\end{pmatrix} \begin{pmatrix}a_\infty &  \\ & {^ta_\infty^{-1}}\end{pmatrix}k_\infty g_\fini\right) = \left\langle \rho_{\lambda'}\left({^ta_\infty}\right) c_{\Phi_\ell(f)}(h,y,\gamma), \rho_{\lambda'}^*(k_\infty)v^*\right\rangle \mathbf{e}(\tr(hz))
	\]
	at $h \in \Sym_{n-1}(F)$ for
	\[
	\begin{pmatrix}\mathbf{1}_{n-1} & n_\infty \\ & \mathbf{1}_{n-1} \end{pmatrix} \begin{pmatrix}a_\infty &  \\ & {^ta_\infty^{-1}}\end{pmatrix} \in P_{n-1,n-1}(\bbR), \, k_\infty \in K_{n-1,\infty}, \, g_\fini \in G_{n-1}(\bbA_{\bbQ,\fini})
	\]
	and
	\[
	y = a_\infty{^t a_\infty} \in (\Sym_{n-1}(\bbR)_{>0})^d,\, z= n+\sqrt{-1}\, y \in \frakH_{n-1}^{d}
	\]
	by (\ref{rel_FC_MF_AF}).
	Here $\gamma$ is an element of $G_{n-1}(\bbQ)$ such that $\gamma^{-1}g_\fini \in G_{n-1}(\bbR)K_{\Gamma_\ell}$.
	For $h \in \Sym_{n-1}(F)$, put
	\[
	h' = \begin{pmatrix}0 & 0 \\ 0 & h \end{pmatrix} \in \Sym_n^{(1)}(F).
	\]
	We then have
	\[
	c_{\Phi_\ell(f)}(h,y,1) = c_f\left(h',\left(\begin{smallmatrix}1 & 0 \\ 0 & y \end{smallmatrix}\right),\left(\begin{smallmatrix}1 & 0 \\ 0 & \gamma \end{smallmatrix}\right)\right),
	\]
	for $y \in (\Sym_{n}(\bbR)_{>0})^d$ and $\gamma \in G_n(\bbQ)$, by the definition of the Siegel operator $\Phi$.
	Hence the diagram is commutative.
	This completes the proof.
	\end{proof}

	\subsection{Surjectivity of $\Phi$}
	
	Let $\lambda = (\lambda_{1,v},\ldots,\lambda_{n,v})_v$ be a $\frakk_n$-dominant integral weight with $\lambda_{n,v} \leq n$ for any $v \in \bfa$.
	Put $\lambda' = (\lambda_{1,v},\ldots,\lambda_{n-1,v})_v$.
	Then $\chi_{\lambda'}$ is sufficiently regular if $\chi_\lambda$ is sufficiently regular.
	For an integral ideal $\frakn$ of $F$, put
	\[
	\Gamma_0(\frakn) = \left\{g = \begin{pmatrix} a & b \\ c & d \end{pmatrix}\Sp_{2n}(\calO_F) \, \middle|\, a,b,c,d \in \Mat_n(\calO_F), c \in \Mat_n(\frakn) \right\}.
	\]
	We say that an integral ideal $\frakn$ is square-free if there exist mutually distinct prime ideals $\frakp_\ell$ such that $\frakn = \prod_{\ell}\frakp_\ell$.
	
	\begin{thm}\label{surjectivity}
	Let $\lambda$ be a $\frakk_n$-dominant weight.
	Suppose that $\lambda_{n,v}$ is independent of $v$ and $\lambda_{n-1,v} = \lambda_{n,v}$ for any $v$.
%	Let $\Gamma$ be a congruence subgroup of $\Sp_{2n}(\calO_F)$ such that there exists a square-free integral ideal $\frakn$ such that $\Gamma_0 \subset \Gamma$. 
	For a square-free integral ideal $\frakn$ of $F$, if $n >1$ and $\lambda_{n,v} > 2n$, the global Siegel operator $\widetilde{\Phi}$ with respect to $\Gamma_0(\frakn)$ is surjective.
	\end{thm}
	\begin{proof}
	Put $\Gamma = \Gamma_0(\frakn)$.
	Fix a set $\{\gamma_1,\ldots,\gamma_h\}$ of complete representatives of $P_{1,n}(\bbQ) \bs G_n(\bbQ)/\Gamma$.
	By Lemma \ref{comm_diag}, it suffices to show the surjectivity of the map 
	\begin{align}\label{phi_op_AF}
	\mathcal{H}(G_n,\lambda)^{K_\Gamma} \longrightarrow \bigoplus_{\ell}\mathcal{H}(G_{n-1},\lambda')^{K_{\Gamma_\ell}} \colon \varphi \longmapsto ((r(g_{\ell,\fini})\varphi_{P_{1,n}})|_{G_{n-1}})_\ell.
	\end{align}
	Note that we have the decomposition 
	\[
	\mathcal{H}(G_n,\lambda)^{K_\Gamma} = \bigoplus_{i=0}^n \mathcal{H}(G_n,\lambda)^{K_\Gamma}_{\{Q_{i,n}\}}, \qquad \mathcal{H}(G_n,\lambda)^{K_\Gamma}_{\{Q_{i,n}\}} = \calH(G_n,\lambda)^{K_\Gamma} \cap \calN(G_n)_{\{Q_{i,n}\}}.
	\]
	Then the kernel of the map (\ref{phi_op_AF}) is equal to $\mathcal{H}(G_n,\lambda)^{K_\Gamma}_{\{G\}}$, the space of cusp forms in $\mathcal{H}(G_n,\lambda)^{K_\Gamma}$.
	Moreover, by Lemma \ref{Wh_coeff} and the assumption for $\lambda$, the image of $\mathcal{H}(G_n,\lambda)^{K_\Gamma}_{\{Q_{i,n}\}}$ is contained in the space $\mathcal{H}(G_{n-1},\lambda')_{\{Q_{i-1,n-1}\}}$ for any $i$.
	
	Fix $i\neq0$.
	It suffices to show the surjectivity of the following map
	\[
	\mathcal{H}(G_n,\lambda)^{K_\Gamma}_{\{Q_{i,n}\}} \longrightarrow \bigoplus_{\ell}\mathcal{H}(G_{n-1},\lambda')^{K_{\Gamma_\ell}}_{\{Q_{i-1,n-1}\}} \colon \varphi \longmapsto ((r(g_{\ell,\fini})\varphi_{P_{1,n}})|_{G_{n-1}})_\ell.
	\]
	By taking the constant term map along $Q_{i-1,n-1}$ and Theorem \ref{main}, we have
	\[
	\mathcal{H}(G_{n-1},\lambda')_{\{Q_{i-1,n-1}\}} \cong \bigoplus_{\mu,\pi} \left(\Ind_{P_{i-1,n-1}(\bbA_{\bbQ,\fini})}^{G_{n-1}(\bbA_{\bbQ,\fini})}\mu|\cdot|^{s'_0}\boxtimes \pi\right)\boxtimes HK(L(\lambda')).
	\]
	Here $\mu$ and $\pi$ runs through the same set as in Theorem \ref{main} and $s'_0$ is equal to $\lambda_{n-i+1,v}-n+i/2$.
	We take functions 
	\[
	f_\ell \in \bigoplus_{\mu,\pi,\ell} \left(\Ind_{P_{i-1,n-1}(\bbA_{\bbQ,\fini})}^{G_{n-1}(\bbA_{\bbQ,\fini})}\mu|\cdot|^{s'_0}\boxtimes \pi\right)^{K_{\Gamma_\ell}}\boxtimes HK(L(\lambda')).
	\]
	To show the surjectivity, it suffices to show that there exists a holomorphic automorphic form $\varphi \in \calH(G_n,\lambda)^{K_\Gamma}$ such that
	\[
	(r(g_\ell)\varphi)_{P_{i,n}}|_{G_{n-1}(\bbA_\bbQ)} = f_\ell.
	\]
	By Theorem \ref{main}, the constant term map along $P_{i,n}$ induces
	\[
	\mathcal{H}(G_{n},\lambda')_{\{Q_{i,n}\}} \cong \bigoplus_{\mu,\pi} \left(\Ind_{P_{i,n}(\bbA_{\bbQ,\fini})}^{G_{n}(\bbA_{\bbQ,\fini})}\mu|\cdot|^{s_0}\boxtimes \pi\right)^{K_\Gamma}\boxtimes HK(L(\lambda)).
	\]
	Here $\mu$ and $\pi$ runs through the same set as above and $s_0 = \lambda_{n,v}-n+(i+1)/2$.
	We define a function $F$ in this space as follows:
	Put
	\[
	F(n \, \bfm(m_1)m_2g_\ell k_\fini) = \mu(\det(m_1))|\det(m_1)|^{\lambda_{n,v}}r(m_2)(f_\ell)
	\]
	for $n \in N_{P_{i,n}}(\bbA_\bbQ), m_1 \in \GL_1(\bbA_F) \times \GL_{i-1}(\bbA_F), m_2 \in G_{n-i}(\bbA_\bbQ)$, and $k_\fini \in K_\Gamma$.
	Here $r$ is the right translation.
	Since $f_\ell$ is right $K_{\Gamma_\ell}$-invariant, $F$ is well-defined.
	Then $F$ is a function on $P_{1,n}(\bbA_\bbQ)G_n(\bbA_{\bbQ,\fini})$.
	
	In order to define the value $F(g)$ for $g \in G_n(\bbR)$, we discuss the induced representation at archimedean places.
	For a $\frakk_{n-1}$-dominant weight $\omega$, a character $\mu$, a complex number $s$ and any function 
	\[
	\alpha \in HK\left(\Ind_{P_{i,n}(\bbR)}^{G_n(\bbR)}(\mu|\cdot|^s \boxtimes L(\omega))\right), 
	\]
	we have
	\[
	\alpha(k) \in HK(L(\omega)), \qquad k \in K_{n,\infty}
	\]
	by $\mathrm{Ad}(K_{n,\infty})\frakp_{n-i,-} \subset \frakp_{n,-}$.
	Hence if the space $HK\left(\Ind_{P_{i,n}(\bbR)}^{G_n(\bbR)}(\mu|\cdot|^s \boxtimes L(\omega))\right)$ is non-zero, for any $w \in HK(L(\omega))$, there exists a function $\alpha \in HK\left(\Ind_{P_{i,n}(\bbR)}^{G_n(\bbR)}(\mu|\cdot|^s \boxtimes L(\omega))\right)$ such that $\alpha(1) = w$.
	Indeed, if $\alpha|_{K_{n,\infty}}=0$, we have $\alpha \equiv 0$.
	Take $k \in K_{n,\infty}$ such that $\alpha(k) \neq 0$.
	Since $HK(L(\omega))$ is irreducible, there exists finite collections of complex numbers $c_p$ and $k_p \in K_{n-1,\infty}$ such that
	\[
	w=\sum_p c_p (k_p\cdot\alpha(k)) = \sum_p c_p (r(kk_p) \alpha)(1).
	\]
	If we replace $\alpha$ by $\sum_p c_p (r(kk_p)\alpha)$, we have $\alpha(1) =w$.
	We decompose $f_\ell$ as a sum
	\[
	f_\ell = \sum_q {v_{\fini,\ell,q}} \otimes v_{\infty,\ell,q}
	\]
	where 
	\[
	v_{\fini,\ell,q} \in \bigoplus_{\mu,\pi,\ell} \left(\Ind_{P_{i-1,n-1}(\bbA_{\bbQ,\fini})}^{G_{n-1}(\bbA_{\bbQ,\fini})}\mu|\cdot|^{s'_0}\boxtimes \pi\right)^{K_{\Gamma_\ell}}, \quad v_{\infty,\ell,q} \in \bigoplus_{\mu,\pi,\ell} HK(L(\lambda')).
	\]
	We then take a function $F_{\infty,\ell,q}$ in 
	\[
	\bigoplus_{\mu,\pi}  \Ind_{P_{i,n}(\bbR)}^{G_{n}(\bbR)}(\mu|\cdot|^{s_0}\boxtimes L(\lambda'))
	\]
	such that
	\[
	F_{\infty,\ell,q}(1) = v_{\infty,\ell,q}.
	\]
	Put
	\[
	F(n \, \bfm(m_1)m_{2,\fini}g_\ell k_\fini g_\infty) = \mu(\det(m_1))|\det(m_1)|^{\lambda_{n,v}}\sum_q (m_{2,\fini} \cdot {v_{\fini,\ell,q}}) \otimes F_{\infty,\ell,q}(g_\infty)
	\]
	for $n \in N_{P_{i,n}}(\bbA_\bbQ), m_1 \in \GL_1(\bbA_F) \times \GL_{i-1}(\bbA_F), m_{2,\fini} \in G_{n-i}(\bbA_{\bbQ,\fini}), k_\fini \in K_\Gamma$, and $g_\infty \in G_n(\bbR)$.
	Then we have $\frakp_{n,-} \cdot F = 0$ and $F(g g_\ell k_\fini) = g \cdot f_\ell$ for $g \in G_{n-1}(\bbA_\bbQ)$ and $k_\fini \in K_\Gamma$.
	Note that the choice of $F$ is not unique.
	The function $F|_{\Sp_{2n}(G_n(\bbA_{\bbQ,\fini}))}$ lies in an induced representation induced from the parabolic subgroup $P$ with the Levi subgroup $\Res_{F/\bbQ} \GL_1 \times \Res_{F/\bbQ} \GL_{i-1} \times G_{n-i}$ of $G_n$.

	If $F$ is left $\bfm(\SL_j(\bbA_\bbQ))$-invariant, there exists $\varphi \in \calH(G_n,\lambda)^{K_\Gamma}$ such that $\varphi_{P_{i,n}} = F$.
	Indeed, if $F$ is so, $F$ lies in the space of induced representations
	\[
	\bigoplus_{\mu,\pi} \left(\Ind_{P_{i,n}(\bbA_{\bbQ,\fini})}^{G_{n}(\bbA_{\bbQ,\fini})}\mu|\cdot|^{s_0}\boxtimes \pi\right)^{K_\Gamma}\boxtimes HK(L(\lambda)).
	\]
	This space is equal to the image of the constant terms map along $P_{i,n}$ of $\calH(G_n,\lambda)^{K_\Gamma}$ by the assumption for $\lambda$.
	To show the left $\bfm(\SL_i(\bbA_\bbQ))$-invariance of $F$, we consider the following map :
	\begin{align}\label{bijective}
	P(\bbA_{\bbQ,\fini}) \bs G_n(\bbA_{\bbQ,\fini})/K_\Gamma \longrightarrow P_{i,n}(\bbA_{\bbQ,\fini}) \bs G_{n}(\bbA_{\bbQ,\fini})/K_\Gamma \colon P(\bbA_{\bbQ,\fini}) g K_\Gamma \longmapsto P_{i,n}(\bbA_{\bbQ,\fini})gK_\Gamma.
	\end{align}
	Here $P$ is the standard parabolic subgroup of $G_n$ with the Levi subgroup $\Res_{F / \bbQ}\GL_1 \times \Res_{F/\bbQ} \GL_{i-1} \times G_{n-i}$.
	Put $K_{n,\fini} = \prod_{v < \infty}\Sp_{2n}(\calO_{F_v})$.
	By the Iwasawa decomposition, we have
	\[
	P(\bbA_{\bbQ,\fini}) \bs G_n(\bbA_{\bbQ,\fini})/K_\Gamma = (P(\bbA_{\bbQ,\fini}) \cap K_{n,\fini}) \bs K_{n,\fini} / K_\Gamma.
	\]
	We then divide by the normal subgroup $K_{\Gamma(\frakn)}$ of $K_{n,\fini}$.
	It is well-known that $K_{n,\fini}/K_{\Gamma(\frakn)} \cong \Sp_{2n}(\calO_{F}/\frakn)$.
	If $\frakn = \prod_\ell \frakp_\ell$, the group $\Sp_{2n}(\calO_{F}/\frakn)$ is isomorphic to $\prod_{\ell}\Sp_{2n}(\calO_F/\frakp_\ell)$.
	The subgroups $(P_{i,n}(\bbA_{\bbQ,\fini}) \cap K_{n,\fini})/K_{\Gamma(\frakn)}, (P(\bbA_{\bbQ,\fini}) \cap K_{n,\fini})/K_{\Gamma(\frakn)}$ and $K_\Gamma / K_{\Gamma(\frakn)}$ are parabolic subgroups of $\prod_{\ell}\Sp_{2n}(\calO_F/\frakp_\ell$ with the Levi subgroups $\GL_i \times \Sp_{2(n-i)}$, $\GL_1 \times \GL_{i-1} \times \Sp_{2(n-i)}$, and $\GL_n$.
	By the Bruhat decomposition, the orders of each hand side of \ref{bijective} is same.
	Hence the map \ref{bijective} is bijective.
	It implies that $F$ is left $\bfm(\SL_i(\bbA_{\bbQ,\fini}))$-invariant.
	For the details, see the proof of \cite[Lemma 8.1]{B-SP}.
	By Theorem \ref{emb_main}, $F$ is left $\bfm(\SL_i(\bbR))$-invariant.
	Hence, $F$ is left $\bfm(\SL_i(\bbA_\bbQ))$-invariant.
	This completes the proof.
	\end{proof}

\end{document}